\documentclass[a4paper,10pt,reqno]{amsart}
\usepackage{amsmath}
\usepackage{amsthm}

\usepackage{graphicx}
\usepackage{amssymb}

\usepackage{appendix}

\usepackage[italian,english]{babel}

\usepackage[utf8x]{inputenc}
\usepackage{fancyhdr}

\usepackage{calc}

\usepackage[text={6in,8.6in},centering]{geometry}

\usepackage{srcltx}

\usepackage[T1]{fontenc}

\usepackage[usenames,dvipsnames]{color}

\makeatletter
\def\cleardoublepage{\clearpage\if@twoside \ifodd\c@page\else%
         \hbox{}%
     \thispagestyle{empty}%              % Empty header styles
     \newpage%
     \if@twocolumn\hbox{}\newpage\fi\fi\fi}
\makeatother

\let\cleardoublepage\clearpage

\addto\captionsitalian{}

\newtheorem{thm}{Theorem}[section]

\newtheorem{lem}[thm]{Lemma}
\newtheorem{pro}[thm]{Proposition}
\newtheorem{den}[thm]{Definition}
\newtheorem{oss}[thm]{Remark}
\numberwithin{equation}{section}

\begin{document}

\title[]{Fractional porous media equations: existence and uniqueness of weak solutions with measure data}

\author {Gabriele Grillo, Matteo Muratori, Fabio Punzo}

\address {Gabriele Grillo, Matteo Muratori: Dipartimento di Matematica ``F. Brioschi'', Politecnico di Milano, Piaz\-za Leonardo da Vinci 32, 20133 Milano, Italy}
\email{gabriele.grillo@polimi.it}
\email{matteo.muratori@polimi.it}

\address{Fabio Punzo: Dipartimento di Matematica ``F. Enriques'', Universit\`a degli studi di Milano, via Cesare Saldini 50, 20133 Milano, Italy}
\email{fabio.punzo@unimi.it}

\keywords{Weighted porous media equation; weighted Sobolev
inequalities; nonlinear diffusion equations; smoothing effect;
asymptotics of solutions.}
%
%\address {Gabriele Grillo, Matteo Muratori: Dipartimento di Matematica, Politecnico di Milano, Piaz\-za Leonardo da Vinci 32, 20133 Milano, Italy}
%
%\email {gabriele.grillo@polimi.it}
%
%
%\email {matteo.muratori@polimi.it}

%\vskip-20pt

\begin{abstract}
We prove existence and uniqueness of solutions to a class of
porous media equations driven by the fractional Laplacian when the
initial data are positive finite Radon measures on the Euclidean space
${\mathbb R}^d$. For given solutions without a prescribed initial
condition, the problem of existence and uniqueness of the initial
trace is also addressed. By the same methods we can also
treat weighted fractional porous media equations, with a weight that
can be singular at the origin, and must have a sufficiently slow
decay at infinity (power-like). In particular, we show that the
Barenblatt-type solutions exist and are unique. Such a result
has a crucial role in \cite{GMP-convergence}, where the
asymptotic behavior of solutions is investigated. Our uniqueness result solves a problem left open, even in the non-weighted case, in \cite{VazConstr}.
%Such kind of
%evolutions have not been treated before even as concerns their
%linear,
%non-fractional analogues: in fact as a preliminary tool we show self-adjointness of a related linear operator as well as the Markov property of the associated semigroup.
\end{abstract}
\maketitle

%\tableofcontents

\vskip -10pt

\section{Introduction}
The main goal of this note is to prove existence and uniqueness of
solutions to the following problem:
\begin{equation}\label{eq: barenblatt-regular}
\begin{cases}
  \rho(x) u_t + (-\Delta)^s\left( u^m \right) = 0  &  \ \textrm{in} \ \mathbb{R}^d \times \mathbb{R}^+ \, , \\
  \rho(x) u = \mu & \ \textrm{on} \ \mathbb{R}^d \times \{ 0 \} \, ,
\end{cases}
\end{equation}
where we assume that $ s \in (0,1) $, $ d>2s $, $ m>1 $, $\mu$ is a positive finite Radon measure on
$\mathbb{R}^d$ (so that $ u \ge 0 $) and that the (Lebesgue) measurable weight $ \rho $ satisfies
\begin{equation}\label{eq: ass-rho}
c |x|^{- \gamma_0 } \le \rho(x) \le C |x|^{- \gamma_0 } \ \ \textrm{a.e.\ in } B_1 \quad \textrm{and} \quad  c |x|^{-\gamma} \le \rho(x) \le C |x|^{-\gamma} \ \ \textrm{a.e.\ in } B_1^c
\end{equation}
for some $ \gamma \in [0,2s), \gamma_0 \in [0,\gamma] $ and $0< c < C $, where $B_r=B_r(0)$. Furthermore, for any given solution to the differential
equation in \eqref{eq: barenblatt-regular}, namely without a prescribed initial datum, we also prove that there exists a unique initial trace which is a positive finite Radon measure (see Theorem \ref{thm: initial}). Observe that this result suggests that is quite natural to consider a positive finite Radon measure $\mu$ as
the initial condition in \eqref{eq: barenblatt-regular}. We stress that the results concerning uniqueness are new even for $\rho\equiv 1$, which
obviously fulfills \eqref{eq: ass-rho}, thus solving an open problem posed in \cite{VazConstr} where such a problem is addressed for initial data
given by Dirac deltas, namely for \it Barenblatt solutions\rm. In this case, the problem is known as \emph{fractional porous media equation} and has
been thoroughly analysed in \cite{DQRV1/2,DQRV} for initial data in $L^1(\mathbb{R}^d)$. More in general, in view of various applications well
outlined in the literature (see e.g.~\cite{KR1}), we also consider the weight $\rho(x)$ since the same methods of proof work in this case as well. In
this regard, observe that even if $\rho\in C(\mathbb R^d)$ has a suitable decay at infinity, and $\mu=u_0\in L^1_\rho(\mathbb R^d)$, then the
asymptotics of \emph{any} solution can be determined by referring to the {\it Barenblatt solution} (i.e.~the solution to problem \eqref{eq:
barenblatt-regular} with $\mu=\delta$) for the problem with \it singular, homogeneous \rm weight $ \rho(x)=|x|^{-\gamma} $, which makes the latter
scale-invariant. Also for this reason we treat weights $\rho$ that satisfy \eqref{eq: ass-rho}, thus being allowed to be singular at $x=0$. However,
some further restrictions on $ s $, $ d $ and $ \gamma $ will be required and clarified later, see Theorems \ref{thm: teorema-esistenza} and
\ref{thm: teorema-uniqueness}. Let us mention that our results entailing the existence and uniqueness of Barenblatt solutions for singular weights are
used in a crucial way in \cite{GMP-convergence} to obtain the asymptotic behavior recalled above.

The analysis of the evolutions addressed here poses significant
difficulties especially as concerns uniqueness, as can be guessed even when considering their linear
analogues. In fact, the first issue we have to deal with is the
essential self-adjointness of the operator formally defined as
$  \rho^{-1}(-\Delta)^s$ on test functions, and the validity of
the Markov property for the associated linear evolution. This will
be crucial in the uniqueness proof and holds only if $\gamma$ is not too large. For larger
$\gamma$ one expects that suitable conditions at infinity
should be required to recover self-adjointness.

Notice that the study of
weighted linear differential operators of second order has a long
story, see for example \cite[Section 4.7]{D} or \cite{Pang}.
Recently, the analysis of the spectral properties of operators which
are modeled on the critical operator formally given by
$|x|^2\Delta$ has been performed in \cite{DFP}.

As for nonlinear evolutions, the study of porous media and fast
diffusion equations with measure data can be tracked back to the
pioneering papers \cite{AC, BF, Pierre, DK}. See
\cite[Section 13]{Vaz07} for details and additional references. The
fast diffusion case, which will not be dealt with here, is
investigated in \cite{CV1, CV2}: notice that for such evolutions the Dirac
delta may not be smoothed into a regular solution, so that
different techniques must be used, see the recent paper \cite{PST}
for a general approach. In \cite{DQRV1/2,
DQRV}, the fractional porous media and fast diffusion equations
have been introduced and thoroughly studied for initial data which are
integrable functions. The construction of Barenblatt solutions and
the analysis of their role as asymptotic attractors for general
integrable data is performed in \cite{VazConstr}. Existence and
uniqueness of solutions in the fractional, weighted case is
studied in \cite{PT1, PT2}: however, the weight there cannot be singular and
data cannot be measures.

Semilinear heat equations with measure data have a long history as well and have recently been studied also in the fractional case, see e.g.~\cite{MV, CVW} and references quoted. We remark that the terminology ``measure data'' is sometimes
used in different contexts in which a measure appears as a source
term in certain evolution equations: see e.g.~\cite{Min}.

There is a huge literature on the weighted porous
media equation: see for example \cite{DG+08,
DNS, Eid90, EK, GM, GMP, KRV10, KR1, KR2, P1, RV06, RV08, RV09}
and references quoted therein. It should be pointed out
that the possible singularity of the weight, and the fact that we
consider measure data as well, makes our problem
significantly different both from the non-weighted, fractional case
and from the weighted, non-fractional case: straightforward
modifications of the strategies used to tackle such problems turn out not to be
applicable here.

Finally, notice that fractional porous media
equations are being used as a model in several applied contexts,
see e.g.~\cite[Appendix B]{BV} and references quoted for details.

\smallskip\noindent{\bf Outline of the paper}. The paper is organized as follows. Section 2 briefly collects some
preliminary tools on measure theory, fractional Laplacians and
fractional Sobolev spaces. In Section 3 we state our main results. In Section 4 we prove existence of weak solutions and the result concerning existence and uniqueness of the initial trace, whereas in Section 5 uniqueness, which is by far the most delicate issue, is addressed: notice that, although we do not state this explicitly, the proofs work also in the case $s=1$ and the corresponding results are new in this context as well for the weighted case. In proving uniqueness, we use a ``duality method'', following the same line of reasoning introduced by M. Pierre in \cite{Pierre}. This entails serious new difficulties due to the presence of the fractional diffusion and of the weight $\rho$.
In Appendix A we recall some technical results on the fractional Laplacian, which are exploited in several approximating procedures developed in the proofs below. In Appendix B we sketch the proof of the main properties of the linear operator
formally given by $ \rho^{-1}(-\Delta)^s $. Such properties are of independent interest but are also crucial in order to establish uniqueness.
% qui nell'introduzione: togliere i due problemi con dato iniziale regolare (metterli nella sezione d'esistenza, dove si usano) e scrivere com'è organizzato il lavoro (cosa c'è nelle varie sezioni)

\section{Preliminary tools}\label{sec: not-def}
In this section we outline some basic notation, definitions and
properties that we shall make us of later, which concern weighted
Lebesgue spaces, measures, fractional Laplacians, fractional
Sobolev spaces and Riesz potentials of measures.

\noindent\textbf{Weighted Lebesgue spaces.} For a given measurable
function $ \rho: \mathbb{R}^d \rightarrow \mathbb{R}^{+} $ (that
is, a weight), we denote as $ L^p_{\rho}(\mathbb{R}^d) $ (let $ p
\in [1,\infty) $) the Banach space constituted by all (classes of
equivalence of) measurable functions $ f: \mathbb{R}^d \rightarrow
\mathbb{R} $ such that
\[%\label{eq: lebesgue-weight}
\left\| f \right\|_{p,\rho} := \left(\int_{\mathbb{R}^d} \left| f(x)
\right|^p \rho(x) \mathrm{d}x\right)^{1/p} < \infty \, .
\]
In the special case $ \rho(x)=|x|^{\alpha} $ (let $\alpha \in
{\mathbb R}$) we simplify notation and replace $
L^p_{\rho}(\mathbb{R}^d) $ by $ L^p_{\alpha}(\mathbb{R}^d) $ and $
\| f \|_{p,\rho} $ by $ \| f \|_{p,\alpha} $. For the usual
unweighted Lebesgue spaces we keep the symbol $
L^p(\mathbb{R}^d) $, denoting the corresponding norms as $\| f
\|_p$ or $\| f \|_{L^p(\mathbb{R}^d)} $.

\noindent\textbf{Positive finite Radon measures on ${\mathbb{R}^d}$.}
Since in \eqref{eq: barenblatt-regular} we deal with positive finite Radon
measures $ \mu $ on $\mathbb{R}^d $, we recall
some basic properties enjoyed by the set of such measures, which
we denote as $ \mathcal{M}(\mathbb{R}^d) $ (with a slight abuse of
notation: this is the usual symbol for the space of \emph{signed}
measures on $\mathbb{R}^d$). To begin with, consider a sequence
$\{ \mu_n \} \subset \mathcal{M}(\mathbb{R}^d) $. Following the
notation of \cite{Pierre}, we say that $ \{\mu_n \} $ converges to
$ \mu \in \mathcal{M}(\mathcal{\mathbb{R}}^d) $ in $
\sigma(\mathcal{M}(\mathbb{R}^d),C_c(\mathbb{R}^d)) $ if there
holds
\begin{equation}\label{eq: def-conv-weak-loc}
\lim_{n \to \infty}  \int_{\mathbb{R}^d} \phi \, \mathrm{d}\mu_n
= \int_{\mathbb{R}^d} \phi \, \mathrm{d}\mu   \ \ \ \forall \phi
\in C_c(\mathbb{R}^d) \, ,
\end{equation}
where $ C_c(\mathbb{R}^d) $ is the space of continuous, compactly
supported functions on $ \mathbb{R}^d$. This is usually referred
to as \emph{local weak$^\ast$\! convergence} (see \cite[Definition
1.58]{AFP}). A classical theorem in measure theory asserts that if
\begin{equation}\label{eq: conv-weak-loc-criterio}
\sup_n{\mu_n(\mathbb{R}^d)} < \infty
\end{equation}
then there exists $ \mu \in \mathcal{M}(\mathbb{R}^d) $ such that
$ \{ \mu_n \} $ converges to $ \mu $ in $
\sigma(\mathcal{M}(\mathbb{R}^d),C_c(\mathbb{R}^d)) $ up to
subsequences (see \cite[Theorem 1.59]{AFP}). The same holds if we replace $ C_c(\mathbb{R}^d) $ with $ C_0(\mathbb{R}^d) $, the latter being the closure of the former w.r.t.~$ \| \cdot \|_\infty $. A stronger notion of
convergence is the following. A sequence $\{ \mu_n \} \subset
\mathcal{M}(\mathbb{R}^d) $ is said to converge to $ \mu \in
\mathcal{M}(\mathbb{R}^d) $ in $
\sigma(\mathcal{M}(\mathbb{R}^d),C_b(\mathbb{R}^d)) $ if
\begin{equation}\label{eq: def-conv-fort}
\lim_{n \to \infty} \int_{\mathbb{R}^d} \phi \, \mathrm{d}\mu_n  =
\int_{\mathbb{R}^d} \phi \, \mathrm{d}\mu   \ \ \ \forall \phi \in
C_b(\mathbb{R}^d) \, ,
\end{equation}
where $ C_b(\mathbb{R}^d) $ is the space of continuous, bounded
functions on $ \mathbb{R}^d$. Trivially, \eqref{eq: def-conv-fort}
implies \eqref{eq: def-conv-weak-loc}. The opposite holds
under a further hypothesis. That is, if $ \{\mu_n\} $ converges to
$ \mu $ in $ \sigma(\mathcal{M}(\mathbb{R}^d),C_c(\mathbb{R}^d)) $
and
$%\label{eq: conv-fort-criterio}
\lim_{n \to \infty} \mu_n(\mathbb{R}^d) =  \mu(\mathbb{R}^d)
$,
then $ \{\mu_n\} $ converges to $ \mu $ also in $
\sigma(\mathcal{M}(\mathbb{R}^d),C_b(\mathbb{R}^d)) $ (see
\cite[Proposition 1.80]{AFP}). Notice that if $ \{\mu_n\} $ converges to
$ \mu $ in $ \sigma(\mathcal{M}(\mathbb{R}^d),C_c(\mathbb{R}^d)) $
and \eqref{eq: conv-weak-loc-criterio} holds, a priori one only
has a weak$^\ast$\! lower semi-continuity property:
\[%\label{eq: conv-fort-criterio-priori}
\mu(\mathbb{R}^d) \le \liminf_{n \to \infty} \mu_n(\mathbb{R}^d)
\]
(see again \cite[Theorem. 1.59]{AFP}).

\noindent\textbf{Fractional Laplacians and fractional Sobolev
spaces.} The fractional $s$-Laplacian operator which appears in
\eqref{eq: barenblatt-regular} is defined, at least for any $ \phi \in
\mathcal{D}(\mathbb{R}^d):=C^\infty_c(\mathbb{R}^d)$, as
\begin{equation*}\label{eq: def-frac-lap}
(-\Delta)^s(\phi)(x):= p.v.\ C_{d,s} \int_{\mathbb{R}^d} \frac{
\phi(x)-\phi(y) }{|x-y|^{d+2s}} \, \mathrm{d}y \ \ \ \forall x \in
\mathbb{R}^d \, ,
\end{equation*}
where $ C_{d,s} $ is a suitable positive constant depending only
on $ d $ and $ s $. However, since a priori we have no clue about
the regularity of solutions to \eqref{eq: barenblatt-regular}, it is
necessary to reformulate the problem in a suitable weak sense, see
Definition \ref{eq: weak-sol-1} below. Before doing it, we need to
introduce some fractional Sobolev spaces. Here we shall mainly
deal with $ \dot{H}^s(\mathbb{R}^d) $, that is the closure of $
\mathcal{D}(\mathbb{R}^d) $ w.r.t.\ the norm
\[%\label{eq: def-norma-Hs-punto}
\left\| \phi \right\|^2_{\dot{H}^s} := \frac{C_{d,s}}{2} \int_{\mathbb{R}^d}
\int_{\mathbb{R}^d} \frac{(\phi(x)-\phi(y))^2}{|x-y|^{d+2s}} \,
\mathrm{d}x \mathrm{d}y \ \ \ \forall \phi \in
\mathcal{D}(\mathbb{R}^d) \, .
\]
Notice that the space usually denoted as $ {H}^s(\mathbb{R}^d) $
is just $ L^2(\mathbb{R}^d) \cap \dot{H}^s(\mathbb{R}^d) $. For
definitions and properties of the general fractional Sobolev
spaces $ W^{r,p}(\mathbb{R}^d)$ we refer the reader e.g.~to
\cite{hiker}.

\noindent The link between the $s$-Laplacian and the space $\dot{H}^s(\mathbb{R}^d) $ can be seen by means of the identity
\begin{equation}\label{eq: id-parti-nonlocal}
\begin{aligned}
\frac{C_{d,s}}{2} \int_{\mathbb{R}^d} \int_{\mathbb{R}^d} \frac{(\phi(x)-\phi(y))(\psi(x)-\psi(y))}{|x-y|^{d+2s}} \, \mathrm{d}x \mathrm{d}y = & \int_{\mathbb{R}^d} (-\Delta)^{\frac{s}{2}}(\phi)(x) \, (-\Delta)^{\frac{s}{2}}(\psi)(x) \,  \mathrm{d}x \\
 = & \int_{\mathbb{R}^d} \phi(x) (-\Delta)^s(\psi)(x)  \, \mathrm{d}x \\
\end{aligned}
\end{equation}
for all $\phi,\psi \in \mathcal{D}(\mathbb{R}^d)$, see \cite[Section
3]{hiker}. In particular, $\left\| \phi \right\|^2_{\dot{H}^s}= \left\|
(-\Delta)^{\frac{s}{2}} (\phi) \right\|_{L^2}^2$ for all $\phi \in \mathcal{D}(\mathbb{R}^d) \, .$
Notice that \eqref{eq: id-parti-nonlocal} can be shown to hold, by approximation, also when $\phi\in \mathcal{D}(\mathbb{R}^d)$ is replaced by any $v \in \dot{H}^s(\mathbb{R}^d)$, where $(-\Delta)^{\frac{s}{2}}(v)$ is meant in the sense of distributions. By a further approximation procedure one then gets
%\begin{equation}\label{eq: id-lapl-s-v}
%\int_{\mathbb{R}^d} h(x) \, (-\Delta)^{\frac{s}{2}}(\psi)(x) \,
%\mathrm{d}x  =  \int_{\mathbb{R}^d} v(x) (-\Delta)^s(\psi)(x)  \,
%\mathrm{d}x  \ \ \ \forall \psi \in \mathcal{D}(\mathbb{R}^d) \, .
%\end{equation}
%Formula \eqref{eq: id-lapl-s-v} is nothing but the definition of
%$(-\Delta)^{\frac{s}{2}}(v)=h $ in the sense of distributions.
%Gathering all this information, one finally obtains the identities
%\begin{equation}\label{eq: id-parti-nonlocal-Hs}
%\begin{aligned}
%C_{d,s} \int_{\mathbb{R}^d} \int_{\mathbb{R}^d} \frac{(v(x)-v(y))(\psi(x)-\psi(y))}{|x-y|^{d+2s}} \, \mathrm{d}x \, \mathrm{d}y = & \int_{\mathbb{R}^d} (-\Delta)^{\frac{s}{2}}(v)(x) \, (-\Delta)^{\frac{s}{2}}(\psi)(x) \,  \mathrm{d}x \\
% = & \int_{\mathbb{R}^d} v(x) (-\Delta)^s(\psi)(x)  \, \mathrm{d}x
%\end{aligned}
%\end{equation}
%for all $v \in \dot{H}^s(\mathbb{R}^d)$ and $\psi \in
%\mathcal{D}(\mathbb{R}^d)$, which clearly illustrate the link
%between the fractional $s$-Laplacian and the space
%$\dot{H}^s(\mathbb{R}^d)$. Moreover, letting $\psi $ tend to $ w
%\in \dot{H}^s(\mathbb{R}^d)$ and passing to the limit in the first
%identity of \eqref{eq: id-parti-nonlocal-Hs} yields
\begin{equation}\label{eq: id-parti-nonlocal-Hs-Hs}
\frac{C_{d,s}}{2} \int_{\mathbb{R}^d} \! \int_{\mathbb{R}^d} \frac{(v(x)-v(y))(w(x)-w(y))}{|x-y|^{d+2s}} \, \mathrm{d}x \mathrm{d}y \! = \! \int_{\mathbb{R}^d} \! (-\Delta)^{\frac{s}{2}}(v)(x) \, (-\Delta)^{\frac{s}{2}}(w)(x) \, \mathrm{d}x \ \ \ \forall v,\! w \in \dot{H}^s(\mathbb{R}^d) \, . \\
\end{equation}
If we set $ v=w $ in \eqref{eq: id-parti-nonlocal-Hs-Hs} we deduce
that $\left\| v \right\|^2_{\dot{H}^s}= \left\|
(-\Delta)^{\frac{s}{2}} (v) \right\|_{L^2}^2$ also for any
$v\in\dot{H}^s(\mathbb{R}^d)$. In Sections \ref{sect: existence} and
\ref{sect: uniqueness} (and in Appendix \ref{sect: app-operatore})
we shall deal with functions which belong to
$\dot{H}^s(\mathbb{R}^d)$ and to weighted Lebesgue spaces.

\noindent\textbf{Riesz potentials.} Another mathematical object
deeply linked with the $s$-Laplacian is its Riesz
kernel, namely the function
\begin{equation*}%\label{eq: def-nucleo-riesz}
I_{2s}(x) := \frac{k_{d,s}}{|x|^{d-2s}} \, ,
\end{equation*}
where $k_{d,s}$ is again a positive constant depending only on $d$ and
$s$. For a given (possibly signed) finite Radon measure $ \nu $, one can show
that the convolution
\begin{equation*}%\label{eq: def-conv-riesz}
U^\nu := I_{2s} \ast \nu
\end{equation*}
yields an $ L^1_{\rm loc}(\mathbb{R}^d) $ function referred to as
the \emph{Riesz potential} of $ \nu $, which formally satisfies
\begin{equation*}%\label{eq: def-conv-riesz-solve}
(-\Delta)^s (U^\nu) = \nu \, .
\end{equation*}
That is, still at a formal level, the convolution against $ I_{2s}
$ coincides with the operator $(-\Delta)^{-s}$. One of the most
important and classical references for Riesz potentials is the
monograph \cite{Landkof} by N. S. Landkof. In the proof of Theorem
\ref{thm: teorema-esistenza} and throughout Section \ref{sect:
uniqueness} we shall exploit some crucial properties of Riesz
potentials collected in \cite{Landkof}, along with their
connections with the $s$-Laplacian.

\section{Statements of the main results} \label{sect: state}
We start by introducing a suitable notion of weak solution to \eqref{eq: barenblatt-regular}, in the spirit of
\cite{DQRV} and \cite{PT2}.
%Before going on note that, in the
%present and in the next sections, by the symbol $u(t)$ we shall
%mean the whole of the function $u(x,t)$ evaluated at time $t\ge0$.

\begin{den}\label{eq: weak-sol-1}
Given a finite positive finite Radon measure $\mu$, by a weak solution to
problem \eqref{eq: barenblatt-regular} we mean a nonnegative function $ u
$ such that
\begin{equation}\label{eq: weak-sol-spazi-u-1}
u \in L^\infty( (0,\infty); L^1_{\rho}(\mathbb{R}^d) ) \cap
L^\infty( \mathbb{R}^d \times (\tau , \infty ) )  \ \ \ \forall
\tau>0 \, ,
\end{equation}
\begin{equation}\label{eq: weak-sol-spazi-u-2}
u \in L^2_{\rm loc}((0,\infty); \dot{H}^s(\mathbb{R}^d))\,,
\end{equation}
\begin{gather}\label{eq: weak-sol-eq-debole}
 - \int_{0}^{\infty} \! \int_{\mathbb{R}^d}  u(x,t) \varphi_t(x,t) \, \rho(x) \mathrm{d}x \mathrm{d}t + \int_0^\infty \! \int_{\mathbb{R}^d} (-\Delta)^{\frac{s}{2}} (u^m)(x,t) \, (-\Delta)^{\frac{s}{2}} ( \varphi ) (x,t) \,  \mathrm{d}x \mathrm{d}t = 0 \\
  \forall \varphi \in C^\infty_c( \mathbb{R}^d \times (0,\infty) )  \nonumber
\end{gather}
and
\begin{equation}\label{eq: weak-sol-initial}
\operatorname{ess}\lim_{t \to 0} \rho(\cdot)\, u(\cdot,t)  = \mu  \ \ \
\textrm{in} \
\sigma(\mathcal{M}(\mathbb{R}^d),C_b(\mathbb{R}^d)) \, .
\end{equation}
%For problem \eqref{eq: barenblatt} weak solutions are
%understood analogously, provided one replaces $ |x|^{-\gamma} $
%with $ \rho $ accordingly.
\end{den}

Our first result concerns existence.
\begin{thm}\label{thm: teorema-esistenza}
Let $d>2s$ and assume that $\rho$ satisfies \eqref{eq: ass-rho} for some $\gamma \in [0,2s) \cap [0,d-2s]$ and $ \gamma_0 \in [0,\gamma] $. Let $ \mu
$ be a positive finite Radon measure. Then there exists a weak solution
$u$ to \eqref{eq: barenblatt-regular} according to Definition \ref{eq:
weak-sol-1}, which conserves the mass in the sense that $\mu\left({\mathbb R}^d\right)=\int_{{\mathbb R}^d}u(x,t)\rho(x){\rm d}x$ for all $t>0$, and satisfies the smoothing effect
\begin{equation}\label{eq: smoothing-effect-general}
 \left\| u(t) \right\|_\infty \le K \, t^{- \alpha } \, \mu(\mathbb{R}^d)^\beta   \ \ \ \forall t > 0 \, ,
\end{equation}
where $K$ depends only on $ m $,
$\gamma$, $s$, $ d $ and on the constant $C$ appearing in \eqref{eq: ass-rho}, and
\begin{equation*}%\label{eq: smoothing-effect-exp-teorema}
\alpha:= \frac{d-\gamma}{(m-1)(d-\gamma) + 2s-\gamma } \, , \
\ \ \beta:= \frac{2s-\gamma}{(m-1)(d-\gamma) + 2s-\gamma} \, .
\end{equation*}
In particular, $u(\cdot,t)\in L^p_{\rho}({\mathbb R}^d)$ for
all $t>0$ and $p\in[1,\infty]$. In addition, the solution satisfies the energy
estimates
\begin{equation}\label{eq: prima-esistenza-energy-1}%\label{eq: prima-esistenza-energy-1-teorema}
 \int_{t_1}^{t_2} \! \int_{\mathbb{R}^d} \left| (-\Delta)^{\frac{s}{2}} \left( u^m \right) (x,t) \right|^2  \mathrm{d}x \mathrm{d}t + \frac1{m+1}\int_{\mathbb{R}^d} u^{m+1}(x,t_2) \, \rho(x)\mathrm{d}x  = \frac1{m+1}\int_{\mathbb{R}^d} u^{m+1}(x,t_1) \, \rho(x) \mathrm{d}x
\end{equation}
and
\begin{equation}\label{eq: prima-esistenza-energy-2}%\label{eq: prima-esistenza-energy-2-teorema}
\int_{t_1}^{t_2} \! \int_{\mathbb{R}^d} \left| z_t(x,t)  \right|^2
\, \rho(x) \mathrm{d}x \mathrm{d}t  \le C^\prime \int_{\mathbb{R}^d} u^{m+1}\left(x,{t_1}/2\right) \, \rho(x)\mathrm{d}x
\end{equation}
for all $t_2>t_1>0 $, where $ z:=u^{\frac{m+1}{2}} $ and $C^\prime$ depends on $m$, $ t_1 $ and $ t_2 $.
%\begin{equation}\label{eq: prima-esistenza-energy-dependence}
% \int_{\mathbb{R}^d} u^{m+1}(x,t_\ast) \, |x|^{-\gamma} \mathrm{d}x \le  \int_{\mathbb{R}^d} u_0^{m+1}(x) \, |x|^{-\gamma} \mathrm{d}x  \, ,
%\end{equation}
%for some $ t_\ast \in (0,t_1) $.

%The same results hold for weak solutions to \eqref{eq:
%barenblatt-regular}, provided $ |x|^{-\gamma} $ is replaced by any
%$ \rho $ satisfying conditions \eqref{eq: ass-rho}.
\end{thm}

The method of proof of Theorem \ref{thm: teorema-esistenza} allows us to prove the following result on existence and uniqueness of the initial trace, in the spirit of \cite[Section 7]{BV} and \cite{BPSV}.
\begin{thm}\label{thm: initial}
Let $d>2s$ and assume that $\rho$ satisfies \eqref{eq: ass-rho} for some $\gamma \in [0,2s) \cap [0,d-2s]$ and $ \gamma_0 \in [0,\gamma] $. Consider a weak solution $u$ to $\rho(x) u_t + (-\Delta)^s\left( u^m \right) = 0$ in the sense that $u$ satisfies \eqref{eq:
weak-sol-spazi-u-1}, \eqref{eq: weak-sol-spazi-u-2} and \eqref{eq:
weak-sol-eq-debole}. Then there exists a unique positive finite Radon measure $\mu$ which is the
initial trace of $u$ in the sense of \eqref{eq: weak-sol-initial}.
%\begin{equation}\label{eq: weak-sol-initial2}
%\operatorname{ess}\lim_{t \to 0} \rho(\cdot)\, u(\cdot,t)  = \mu  \ \ \
%\textrm{in} \
%\sigma(\mathcal{M}(\mathbb{R}^d),C_b(\mathbb{R}^d)) \, .
%\end{equation}
The same result holds if the condition $u\in L^\infty( \mathbb{R}^d \times (\tau , \infty ) )$ in \eqref{eq:
weak-sol-spazi-u-1} is replaced by the weaker condition $\int_{t_1}^{t_2} u^m(\cdot,\tau)\,\mathrm{d}\tau\in L^1_\rho({\mathbb R}^d)$ for all $t_2>t_1>0$. In particular, $\mu\left({\mathbb R}^d\right)=\int_{{\mathbb R}^d}u(x,t)\rho(x){\rm d}x$ for all $t>0$.
%Property also \eqref{eq: weak-sol-initial2} holds in $\sigma(\mathcal{M}(\mathbb{R}^d),C_b(\mathbb{R}^d))$ provided one requires $\int_{0}^{t} u^m(\cdot,\tau)\,\mathrm{d}\tau\in L^1_\rho({\mathbb R}^d)$ for all $t>0$.
%The same results hold for weak solutions of $\rho(x) u_t + (-\Delta)^s\left( u^m \right) = 0$, provided conditions \eqref{eq: ass-rho} hold.
\end{thm}
%The above results can be extended, by the same methods, to yield the validity of \eqref{eq: weak-sol-initial2} in $\sigma(\mathcal{M}(\mathbb{R}^d),C_b(\mathbb{R}^d))$ provided the condition $\int_{0}^{t} u^m(\cdot,t)\,\mathrm{d}t\in L^1_\rho({\mathbb R}^d)$ holds for all $t>0$.

As for uniqueness of weak solutions we have the next result.
\begin{thm}\label{thm: teorema-uniqueness}
Let $d>2s$ and assume that $\rho$ satisfies \eqref{eq: ass-rho} for some $\gamma \in [0,2s) \cap [0,d-2s]$ and $ \gamma_0 \in [0,\gamma] $. Let $u_1, u_2$
be two weak solutions to \eqref{eq: barenblatt-regular} in the sense of
Definition \ref{eq: weak-sol-1}. Suppose that they take as
initial datum the same positive finite Radon measure $\mu$, in the sense
of \eqref{eq: weak-sol-initial}. Then $u_1=u_2$.
%The same result holds true for weak solutions to \eqref{eq:
%barenblatt-regular}, provided $ |x|^{-\gamma} $ is replaced by any
%$ \rho $ satisfying conditions \eqref{eq: ass-rho}.
\end{thm}

\begin{oss}\rm
Notice that, if $d\ge4s$, then the assumptions on $\gamma$ reduce to $\gamma\in[0,2s)$.
\end{oss}

Let us stress that, in order to prove Theorem \ref{thm: teorema-uniqueness}, we shall crucially exploit the properties of the operator $A=\rho^{-1}\,(-\Delta)^s$ contained in Theorem \ref{thm: self-adj} and Proposition \ref{pro: laplaciano-Lp} below. Such results are of independent interest; their proofs will be just sketched, to keep the paper in a reasonable length, in Appendix \ref{sect: app-operatore}. Some further details and extentions are given in \cite{Nota-oper}.

\begin{den}\label{den: spazio-Xs}
Let $d>2s $ and assume that $\rho$ satisfies \eqref{eq: ass-rho}
for some $\gamma \in [0,2s)$ and $ \gamma_0 \in
[0,d) $. We denote as $X_{s,\rho}$ the Hilbert space of all
functions $v \in L^2_{\rho}(\mathbb{R}^d)$ such that
$(-\Delta)^s(v)$ (as a distribution) belongs to
$L^2_{\rho^{-1}}(\mathbb{R}^d)$, equipped with the norm
\begin{equation*}%\label{eq: norma-xs}
\left\| v \right\|_{X_{s,\rho}}^2 := \left\| v
\right\|_{2,\rho}^2 + \left\| (-\Delta)^s(v)
\right\|_{2,\rho^{-1}}^2 \ \ \ \forall v \in X_{s,\rho} \, .
\end{equation*}
\end{den}
\begin{thm}\label{thm: self-adj}
Let $d>2s $ and assume that $\rho$ satisfies \eqref{eq: ass-rho} for some $ \gamma \in [0,2s) $ and $ \gamma_0 \in [0,d) $. Let $A: D(A):=X_{s,\rho}
\subset L^2_{\rho}(\mathbb{R}^d) \rightarrow
L^2_{\rho}(\mathbb{R}^d)$ be the operator
\begin{equation*}%\label{eq: def-operator-A}
A(v):=\rho^{-1}\,(-\Delta)^s(v)  \ \ \ \forall v \in  X_{s,\rho}
\, .
\end{equation*}
Then $A$ is densely defined, positive and self-adjoint on
$L^2_{\rho}(\mathbb{R}^d)$, and the quadratic form associated
to it is
\begin{equation*}%\label{eq: def-quadform-A}
Q(v,v):=\frac{C_{d,s}}{2} \int_{\mathbb{R}^d} \int_{\mathbb{R}^d}
\frac{(v(x)-v(y))^2}{|x-y|^{d+2s}} \, \mathrm{d}x \mathrm{d}y
\end{equation*}
with domain $ D(Q):=L^2_{\rho}(\mathbb{R}^d) \cap \dot{H}^s(\mathbb{R}^d)$. Moreover, $ Q$ is a Dirichlet form on
$L^2_{\rho}(\mathbb{R}^d)$ and $ A $ generates a Markov
semigroup $S_2(t)$ on $L^2_{\rho}(\mathbb{R}^d)$. In particular,
for all $p\in[1,\infty]$ there exists a contraction semigroup
$S_p(t)$ on $L^p_{\rho}(\mathbb{R}^d)$, consistent with
$S_2(t)$ on $L^2_{\rho}(\mathbb{R}^d)\cap
L^p_{\rho}(\mathbb{R}^d)$, which is furthermore analytic with a
suitable angle $\theta_p>0$ for $p\in(1,\infty)$.
\end{thm}
%\noindent Further technical but crucial results concerning the operator $A=\rho(x)^{-1}\,(-\Delta)^s$ are given in Appendix \ref{sect: app-operatore}, see in particular Proposition \ref{pro: laplaciano-Lp}.

\section{Existence of weak solutions}\label{sect: existence}
We start showing a direct consequence of Definition \ref{eq:
weak-sol-1}, namely the conservation in time of the ``mass''
$\int_{\mathbb{R}^d}u(x,t) \, \rho(x) \mathrm{d}x $ (recall that we are considering
nonnegative solutions).
\begin{pro}\label{oss: cons-mass}  \rm
Let $d>2s $ and assume that $\rho$ satisfies \eqref{eq: ass-rho}
for some $\gamma \in [0,2s) \cap [0,d-2s]$ and $ \gamma_0 \in
[0,\gamma] $. Let $ u $ be a weak solution to \eqref{eq:
barenblatt-regular} according to Definition \ref{eq: weak-sol-1}.
Then
\begin{equation}\label{eq: cons-mass}
\left\| u(t) \right\|_{1,\rho} = \int_{\mathbb{R}^d} u(x,t) \,
\rho(x) \mathrm{d}x = \mu(\mathbb{R}^d) \ \ \
\textrm{for a.e.} \ t>0 \, ,
\end{equation}
namely we have \emph{conservation of mass}.
\begin{proof}
We plug into \eqref{eq:
weak-sol-eq-debole} the test function $%\label{eq: cons-mass-test}
\varphi_{R}(x,t) := \vartheta(t) \xi_R(x)$,
where $ \xi_R $ is the same cut-off function as in Lemma
\ref{lem:decay-lap-cutoff} and $
\vartheta$ is a suitable positive, regular and compactly supported
approximation of $ \chi_{[t_1,t_2]}$ (let $
t_2>t_1>0 $). Using \eqref{eq: id-parti-nonlocal-Hs-Hs}, Lemma \ref{lem:decay-lap-1},
Lemma \ref{lem:decay-lap-cutoff} and letting $ \vartheta \to \chi_{[t_1,t_2]} $ in \eqref{eq: weak-sol-eq-debole}, it is straightforward to obtain the following estimate:
\begin{equation}\label{eq: uniq-trace-cons-mass}
\begin{aligned}
& \left| \int_{\mathbb{R}^d}  u(x,t_2) \xi_R(x) \, \rho(x) \mathrm{d}x - \int_{\mathbb{R}^d}  u(x,t_1) \xi_R(x) \, \rho(x)  \mathrm{d}x  \, \right| \\
\le &  c^{-1} \left( \frac{1}{R^{2s}} + \frac{1}{R^{2s-\gamma}} \right) \left\| (1+|x|^\gamma) (-\Delta)^{s}(\xi) \right\|_\infty  \int_{t_1}^{t_2} \! \int_{\mathbb{R}^d} u^m(x,t) \, \rho(x) \mathrm{d}x \mathrm{d}t \, ,
\end{aligned}
\end{equation}
where on the r.h.s.\ we exploited the inequality $\rho^{-1}(x) \le c^{-1} \left( 1+|x|^\gamma \right)$ for all $x \in \mathbb{R}^d$,
direct consequence of \eqref{eq: ass-rho}. Letting $ R \to \infty $ in \eqref{eq: uniq-trace-cons-mass} and recalling \eqref{eq: weak-sol-initial} we get the conclusion.\end{proof}
\end{pro}

The proof of existence of weak solutions to \eqref{eq:
barenblatt-regular} is based on an approximation procedure, that is on picking a sequence of initial data in $ L^1_{\rho}(\mathbb{R}^d) \cap L^\infty(\mathbb{R}^d) $ which suitably converges to $ \mu $. An additional approximation will be needed to deal with the possible singularity of the weight at the origin. The corresponding approximate problems are addressed in the next subsection. Since the procedure is in principle standard although technically delicate, we underline the main points only.
\subsection{Approximate problems with initial data in $L^1_{\rho}(\mathbb{R}^d) \cap
L^\infty(\mathbb{R}^d) $} We are concerned with existence of
solutions to the following problem:
\begin{equation}\label{eq: barenblatt-u_0}
\begin{cases}
\rho(x) u_t + (-\Delta)^s\left( u^m \right) = 0  &  \ \textrm{in} \ \mathbb{R}^d \times \mathbb{R}^+ \, , \\
u = u_0  & \ \textrm{on} \ \mathbb{R}^d \times \{0\} \,.
\end{cases}
\end{equation}
 Such solutions are meant in the sense of
Definition \ref{eq: weak-sol-1} with $\mu$ replaced by
$\rho u_0$.

\begin{lem}\label{lem: prima-esistenza}
Let $ d > 2s $ and assume that $\rho$ satisfies \eqref{eq:
ass-rho} for some $\gamma \in [0,2s) \cap [0,d-2s]$ and $ \gamma_0
\in [0,\gamma] $. Let $ u_0 \in L^1_{\rho}(\mathbb{R}^d) \cap
L^\infty(\mathbb{R}^d) $, with $ u_0 \ge 0 $. Then there exists a
weak solution $u$ to \eqref{eq: barenblatt-u_0} which satisfies
the energy estimates \eqref{eq: prima-esistenza-energy-1}, \eqref{eq: prima-esistenza-energy-2}
%\begin{equation}\label{eq: prima-esistenza-energy-1}
% \int_{t_1}^{t_2} \! \int_{\mathbb{R}^d} \left| (-\Delta)^{\frac{s}{2}} \left( u^m \right) (x,t) \right|^2  \mathrm{d}x \mathrm{d}t + \int_{\mathbb{R}^d} u^{m+1}(x,t_2) \, \rho(x) \mathrm{d}x  = \int_{\mathbb{R}^d} u^{m+1}(x,t_1) \,  \rho(x) \mathrm{d}x
% \ \ \forall t_2 > t_1 \ge 0
%\end{equation}
%and
%\begin{equation}\label{eq: prima-esistenza-energy-2}
%\int_{t_1}^{t_2} \! \int_{\mathbb{R}^d} \left| z_t(x,t)  \right|^2
%\,  \rho(x) \mathrm{d}x \mathrm{d}t  \le C \int_{\mathbb{R}^d} u^{m+1}\left(x,\frac{t_1}2\right) \, \rho(x)\mathrm{d}x  \ \ \ \forall
%t_2 > t_1 > 0  \, ,
%\end{equation}
with a constant $C^\prime$ depending only on $m$, $ t_1 $ and $ t_2 $.

%and on the initial datum $ u_0 $
%through the integral
%\begin{equation}\label{eq: prima-esistenza-energy-dependence}
% \int_{\mathbb{R}^d} u^{m+1}(x,t_\ast) \, |x|^{-\gamma} \mathrm{d}x \le  \int_{\mathbb{R}^d} u_0^{m+1}(x) \, |x|^{-\gamma} \mathrm{d}x  \, ,
%\end{equation}
%for some $ t_\ast \in (0,t_1) $.
% da qualche parte nella dimostrazione citare gli spagnoli
\medskip

{\em Let us outline the strategy of the proof.
We further approximate the problem
\eqref{eq: barenblatt-u_0} by regularizing the weight
$\rho(x)$ in a neighbourhood of $x=0$ (where it can be singular). More precisely,
we introduce for any $\eta>0$ the following
problem:
\begin{equation}\label{eq: barenblatt_approx_2}
\begin{cases}
\rho_\eta(x) \left(u_{\eta}\right)_t + (-\Delta)^s\left( u_{\eta}^m \right) = 0  &  \ \textrm{in} \ \mathbb{R}^d \times \mathbb{R}^+ \, , \\
  u_{\eta} =  u_0 & \ \textrm{on} \ \mathbb{R}^d \times \{ 0 \} \, ,
\end{cases}
\end{equation}
where $\{ \rho_\eta \} \subset C(\mathbb{R}^d)$ is a family of
strictly positive weights which behave like $ |x|^{-\gamma} $ at
infinity and approximate $ \rho(x) $ monotonically from below, as
$\eta\to 0$. Existence (and uniqueness) of weak solutions to
\eqref{eq: barenblatt_approx_2} for such weights and initial data
have been established in \cite[Theorem 3.1]{PT2}. We get suitable
a priori estimates (namely \eqref{eq: prima-esistenza-energy-1} applied to $u_\eta$, which will be proved later, and \eqref{eq:
derivata-temporale-potenza-um} below), that enable us to pass to
the limit as $\eta\to 0$, and obtain a solution to problem
\eqref{eq: barenblatt-u_0}, by standard compactness arguments. }

\begin{proof} For any $\eta>0$ let $u_\eta$ be the unique solution to problem \eqref{eq: barenblatt_approx_2}.
 %For instance, one can
%pick
%\begin{equation}\label{eq: approx-rho-expl}
%\rho_\eta(x)= \left( |x|^2 + \eta \right)^{-\frac{\gamma}{2}} \ \
%\ \forall x \in \mathbb{R}^d \, .
%\end{equation}
%Note that $ u_0 \in L^1_{\rho_\eta}(\mathbb{R}^d) \cap
%L^\infty(\mathbb{R}^d) $.
Such solutions belong to $C([0,\infty);L^1_{\rho_\eta}(\mathbb{R}^d)) $ and satisfy the
bound $
%\begin{equation}\label{eq: PT-bound-infty}
\left\| u_\eta \right\|_{L^\infty(\mathbb{R}^d \times(0,\infty))}
\le \left\| u_0 \right\|_{L^\infty(\mathbb{R}^d)}.$
%\end{equation}
Exploiting these properties one can show that each $  u_\eta
$ satisfies a weak formulation which is slightly stronger than the
one of Definition \ref{eq: weak-sol-1}:
\begin{equation}\label{eq: formDebole-approx-1}
\begin{aligned}
&- \int_{0}^{T} \! \int_{\mathbb{R}^d}  u_{\eta}(x,t) \varphi_t(x,t) \, \rho_\eta(x) \mathrm{d}x \mathrm{d}t + \int_0^T \! \int_{\mathbb{R}^d} (-\Delta)^{\frac{s}{2}} (u_{\eta}^m)(x,t) \, (-\Delta)^{\frac{s}{2}} (\varphi)(x,t) \, \mathrm{d}x \mathrm{d}t \\
= & \int_{\mathbb{R}^d} u_0(x)  \varphi(x,0) \, \rho_\eta(x) \mathrm{d}x
\end{aligned}
\end{equation}
for all $T>0$ and $\varphi \in C^\infty_c( \mathbb{R}^d \times [0,T))$ (so that $\varphi(\cdot,T)=0$), where $ u_\eta^m \in
L^2((0,\infty);\dot{H}^s(\mathbb{R}^d))$. The latter property follows from the validity of the energy identity \eqref{eq: prima-esistenza-energy-1}
for $u_\eta$
%\begin{equation}\label{eq: energia_1}
%\begin{aligned}
%\int_{t_1}^{t_2} \! \int_{\mathbb{R}^d} \left|
%(-\Delta)^{\frac{s}{2}} \left( u_{\eta}^m \right) (x,t) \right|^2
%\mathrm{d}x \mathrm{d}t + \frac{1}{m+1} \int_{\mathbb{R}^d} \!
%u_{\eta}^{m+1}(x,t_2)  \rho_\eta(x) \mathrm{d}x = \frac{1}{m+1}
%\int_{\mathbb{R}^d} \! u_{\eta}^{m+1}(x,t_1)  \rho_\eta(x)
%\mathrm{d}x
%\end{aligned}
%\end{equation}
for all $ t_2>t_1 \ge 0 $. Formally, \eqref{eq: prima-esistenza-energy-1} can be
proved by plugging the test function $ \varphi(x,t):=\vartheta(t)
u^m_\eta(x,t) $ into the weak formulation \eqref{eq:
formDebole-approx-1} and letting $ \vartheta $ tend to $
\chi_{[t_1,t_2]} $ as in the proof of Proposition \ref{oss:
cons-mass}. %The problem is that, a priori, such a $ \varphi $ is
%not admissible as a test function.
In order to justify rigorously the validity of \eqref{eq: prima-esistenza-energy-1} for $u_\eta$, one must
proceed as in \cite[Section 8]{DQRV}. A crucial point concerns the
fact that our solutions are \emph{strong}, which follows by
techniques analogous to the ones used in \cite[Section 8.1]{DQRV}.
We refer the reader to Section \ref{sect: strong} below for more
details.
%The discussion there is focussed, for simplicity, on the special
%case of the weight $ |x|^{-\gamma} $, but as recalled in the
%Introduction it applies safely to any weight $ \rho $ complying
%with \eqref{eq: ass-rho}.
We have:
\begin{equation}\label{eq: derivata-temporale-potenza}
\int_{t_1}^{t_2} \! \int_{\mathbb{R}^d} \left| (z_\eta)_t(x,t)
\right|^2  \rho_\eta(x) \mathrm{d}x \mathrm{d}t  \le C \int_{\mathbb{R}^d} u_\eta^{m+1}\left(x,{t_1}/2\right) \, \rho(x)\mathrm{d}x\ \ \
\forall t_2>t_1>0 \, ,
\end{equation}
where $ z_\eta := u_\eta^{\frac{m+1}{2}} $ and $C$ depends only on $m$, $ t_1 $ and $t_2 $. %through the integral
%\begin{equation*}%\label{eq: prima-esistenza-energy-dependence-prova}
% \int_{\mathbb{R}^d} u_\eta^{m+1}(x,t_\ast) \, \rho_\eta(x) \mathrm{d}x \le \int_{\mathbb{R}^d} u_0^{m+1}(x) \, \rho_\eta(x) \mathrm{d}x
%\end{equation*}
%for some $ t_\ast \in (0,t_1) $.
%Again, \eqref{eq: derivata-temporale-potenza} can be \emph{formally} proved by
%picking the test function $ \varphi(x,t) =
%\zeta(t)(u^m_\eta)_t(x,t) $ and integrating by parts in time,
%where $ \zeta $ is any positive regular function with compact
%support in $ (0,\infty) $ ($ t_\ast $ is the infimum of its
%support) such that $ \zeta=1 $ on $ [t_1,t_2] $. Actually the
%validity of \eqref{eq: derivata-temporale-potenza} is one of the
%main tools that one uses to prove that solutions are strong, and
%its
\noindent Formula \eqref{eq: derivata-temporale-potenza} follows as in \cite[Lemma 8.1]{DQRV}. Since
%\begin{equation*}%\label{eq: derivata-t-potenza}
$$\left( u_{\eta}^{m} \right)_t = c_m\,
z_{\eta}^{\frac{m-1}{m+1}} \left(z_{\eta} \right)_t
\quad
\textrm{and}
\quad
\left\| z_\eta \right\|_{L^\infty(\mathbb{R}^d \times (0,\infty))}
= \left\| u_\eta \right\|_{L^\infty(\mathbb{R}^d \times
(0,\infty))}^{\frac{m+1}{2}} \le \left\| u_0
\right\|_{L^\infty(\mathbb{R}^d)} ^{\frac{m+1}{2}} \, , $$
from \eqref{eq: derivata-temporale-potenza} we deduce that
\begin{equation}\label{eq: derivata-temporale-potenza-um}
\int_{t_1}^{t_2} \! \int_{\mathbb{R}^d} \left|
\left(u_\eta^m\right)_t(x,t) \right|^2  \rho_\eta(x) \mathrm{d}x \mathrm{d}t  \le k \left\| u_0
\right\|_\infty^{{m-1}} \ \ \ \forall t_2>t_1 > 0
\end{equation}
for a suitable $ k>0 $ independent of $ \eta $. Moreover, the validity of
$ \int_{t_1}^{t_2} \! \int_{\mathbb{R}^d} \left| u_\eta^m(x,t) \right|^2  \rho_\eta(x) \mathrm{d}x \mathrm{d}t  \le C^{\prime\prime}$ for all $t_2 > t_1 \ge 0$ and
for another suitable positive constant $ C^{\prime\prime} $ that depends only on $m$, $ t_1 $, $ t_2 $ and $ u_0 $ is ensured by the
conservation of mass \eqref{eq: cons-mass} (with $ \rho = \rho_\eta $) and by the uniform bound on $\left\| u_\eta \right\|_{L^\infty(\mathbb{R}^d \times(0,\infty))}$. %, which yields
%\begin{equation*}%\label{eq: cons-mass-stima}
%\left\| u_\eta(t) \right\|_{1,\rho_\eta} = \left\| u_0
%\right\|_{1,\rho_\eta} \le \left\| u_0 \right\|_{1,\gamma}  \ \ \
%\forall t>0 \, ,
%\end{equation*}
%and by the uniform boundedness of $ u_\eta $ given by \eqref{eq:
%PT-bound-infty}.
Let $n\in \mathbb{N}$. We now use \eqref{eq: lemma-Hs-stima-Hs} with $\xi_1 = \xi_{1, n}\in C^\infty(\mathbb R^d)$ such that
\[\xi_1 \equiv 1\quad \textrm{in}\;\, B_n \, , \;\;\; \xi_1 \equiv 0 \quad \textrm{in}\;\, B_{2n} \, ,\]
and with $\xi_2 = \xi_{2, n}\in C^\infty((0, \infty))$ such that
\[\xi_2 \equiv 1 \quad \textrm{in}\;\, \left(\frac 1 n, n\right) , \;\;\; \xi_2 \equiv 0\quad \textrm{in}\;\, \left(0, \frac 1{2n}\right)\cup (2n,\infty)\, . \]
The fact that $ H^s(\mathbb{R}^{d+1}) $ is compactly embedded in $ L^2_{\rm loc}(\mathbb{R}^{d+1}) $ (see e.g.~\cite[Theorem 7.1]{hiker}), and a standard
diagonal procedure allow us to pass to the limit as $\eta\to0$ in \eqref{eq: formDebole-approx-1} and get that the weak limit $u$ of $\{u_\eta\}$
satisfies
%Thanks to \eqref{eq: energia_1}, \eqref{eq:
%derivata-temporale-potenza-um} and \eqref{eq: norma-L^2-u^m} we
%are in position to apply Lemma \ref{lem: Hs-loc} with the choice $
%v=u_\eta^m $. In place of the weight $ \rho $ there, exploiting
%the monotonicity of $\{ \rho_\eta \}$, we can pick for instance $
%\rho_1 $. Hence, estimate \eqref{eq: lemma-Hs-stima-Hs} and the
%fact that $ H^s(\mathbb{R}^{d+1}) $ is compactly embedded in $
%L^2_{loc}(\mathbb{R}^{d+1}) $ (see e.g.\ \cite[Th. 7.1]{hiker})
%imply that, up to subsequences, $\{ u_\eta\} $ converges at least
%pointwise to some limit function $ u $ as $ \eta \to 0 $.
%Furthermore, from \eqref{eq: energia_1} we deduce that $ \{
%u^m_\eta \} $ admits (still up to subsequences) a weak limit $ w $
%in $ L^2((0,T); \dot{H}^s(\mathbb{R}^d)) $ for all $ T > 0 $. The
%identification between $ w $ and $ u^m $ is just a consequence of
%the pointwise convergence of $ \{ u_\eta \} $ to $ u $. We can
%therefore pass to the limit in the weak formulation \eqref{eq:
%formDebole-approx-1} and obtain that such $ u $ satisfies
\begin{equation}\label{eq: formDebole-approx-1-limite}
\begin{aligned}
 & - \int_{0}^{T} \! \int_{\mathbb{R}^d}  u(x,t) \varphi_t(x,t) \, \rho(x) \mathrm{d}x \mathrm{d}t  + \int_0^T \! \int_{\mathbb{R}^d} (-\Delta)^{\frac{s}{2}} (u^m)(x,t) \, (-\Delta)^{\frac{s}{2}} (\varphi)(x,t) \,  \mathrm{d}x \mathrm{d}t & \\
 = & \int_{\mathbb{R}^d} u_0(x) \varphi(x,0) \, \rho(x) \mathrm{d}x
\end{aligned}
\end{equation}
for all $T>0$ and $\varphi \in C^\infty_c( \mathbb{R}^d \times [0,T) )$. The validity of \eqref{eq: weak-sol-initial}
follows by plugging into \eqref{eq: formDebole-approx-1-limite} the test function $ \varphi(x,t):=\vartheta(t) \xi_R(x)  $, where $ \xi_R $ is a
cut-off function as in Lemma \ref{lem:decay-lap-cutoff} and $ \vartheta $ is a regular approximation of $ \chi_{[0,t_2]} $. One then lets $ t_2 \to 0
$ and $ R \to \infty $.

The energy estimates \eqref{eq: prima-esistenza-energy-1} and
\eqref{eq: prima-esistenza-energy-2} for $u$ can be obtained
reasoning exactly as above (one uses again the fact
that solutions are strong).
%Alternatively, \eqref{eq: prima-esistenza-energy-1} for $ t_1=0 $ and \eqref{eq: prima-esistenza-energy-2} follow from weak lower-semicontinuity arguments. However, the validity of \eqref{eq: prima-esistenza-energy-1} for any $ t_1>0 $ and the dependence of the constant $ C $ in \eqref{eq: prima-esistenza-energy-2} on $ \| u(t_\ast) \|_{m+1,\gamma} $ cannot be proved by weak lower-semicontinuity, and below we shall need these two properties.
\end{proof}
\end{lem}

\subsection{Stroock-Varopoulos inequality and smoothing estimate}
Having at our disposal an existence result for problem \eqref{eq:
barenblatt-u_0}, we can now let $ \rho u_0 $ approximate $ \mu $.
In order to show that the corresponding solutions converge to a
solution of \eqref{eq: barenblatt-regular}, we need first some
technical results. We begin with a modification of the classical
Stroock-Varopoulos inequality: it is proved here for $v \in
L^\infty(\mathbb{R}^d) \cap \dot{H}^s(\mathbb{R}^d) $ with $
(-\Delta)^s(v) \in L^1(\mathbb{R}^d) $. Observe that, under the
hypothesis that $v \in L^q(\mathbb{R}^d) \cap
\dot{H}^s(\mathbb{R}^d) $ with $ (-\Delta)^s(v) \in
L^q(\mathbb{R}^d) $, for $q>1$, such an inequality can be found,
e.g., in \cite[Section 5]{DQRV} or \cite{BIK}. See also
\cite[formula (2.2.7)]{D} for a similar inequality involving
general Dirichlet forms. The present result seems to be new, in
view of its functional framework, therefore its proof is given in
some detail.
\begin{lem}\label{lem: stroock-var}
Let $ d > 2s $. For all nonnegative $v \in L^\infty(\mathbb{R}^d)
\cap \dot{H}^s(\mathbb{R}^d) $ such that $ (-\Delta)^s(v) \in
L^1(\mathbb{R}^d) $, the inequality
\begin{equation}\label{eq: strook-vareq}
\int_{\mathbb{R}^d} v^{q-1}(x) (-\Delta)^{s}(v)(x) \, \mathrm{d}x
\ge \frac{4(q-1)}{q^2} \int_{\mathbb{R}^d} \left|(-\Delta)^{\frac
s 2} \big( v^{\frac q 2} \big)(x) \right|^2 \mathrm{d}x
\end{equation}
holds for any $ q>1 $.
\end{lem}
\begin{proof} We shall assume, with no loss of generality, that $ v $ is a regular function. Indeed, by standard mollification arguments, one can always pick a sequence $ \{ v_n \} \subset C^\infty(\mathbb{R}^d)\cap L^\infty(\mathbb{R}^d) \cap \dot{H}^s(\mathbb{R}^d) $ such that $ \{ v_n \} $ converges pointwise to $ v $, $ \| v_n \|_\infty \le \| v \|_\infty $ and $\{ (-\Delta)^s(v_n) \} $ converges to $ (-\Delta)^s(v) $ in $ L^1(\mathbb{R}^d) $. This is enough to pass to the limit as $ n \to \infty $ on the l.h.s.\ of \eqref{eq: strook-vareq}, while on the r.h.s.\ one exploits the weak lower semi-continuity of the $L^2 $ norm.

Consider the following sequences of functions:
\begin{gather*}%\label{eq: SV-psi-n}
\psi_n(x) := \int_0^{x \wedge \frac1n} y^{\frac{4s}{d-2s}} \, \mathrm{d}y + (q-1) \int_{\frac1n}^{x \vee \frac1n } y^{q-2} \, \mathrm{d}y \ \ \ \forall x \in \mathbb{R}^{+}  \, , \\
\Psi_n(x) := \int_0^{x \wedge \frac1n} y^{\frac{2s}{d-2s}} \, \mathrm{d}y + (q-1)^{\frac12} \int_{\frac1n}^{x \vee \frac1n } y^{\frac{q}{2}-1} \, \mathrm{d}y \ \ \ \forall x \in \mathbb{R}^{+} \, . %\label{eq: SV-Psi-n}
\end{gather*}
It is plain that $ \psi_n$ and $ \Psi_n $ are absolutely continuous, monotone increasing functions such that
$\psi_n^\prime(x) = \left[ \Psi_n^\prime(x) \right]^2$ for all $x \in \mathbb{R}^{+}$. For any $ R>0 $, take a cut-off function $ \xi_R $ as in Lemma \ref{lem:decay-lap-cutoff}. To the function $ \xi_R v $ one can apply Lemma 5.2 of \cite{DQRV} with the choices $\psi= \psi_n $ and $\Psi=\Psi_n $, which yields
\begin{equation}\label{eq: strook-var-approx}
\int_{\mathbb{R}^d} \psi_n(\xi_R v)(x) \, (-\Delta)^{s}(\xi_R v)(x) \, \mathrm{d}x  \ge \int_{\mathbb{R}^d} \left|(-\Delta)^{\frac s
2}( \Psi_n(\xi_R v) )(x) \right|^2 \mathrm{d}x \, .
\end{equation}
Expanding the $s$-Laplacian of the product of two functions, we get that the l.h.s.\ of \eqref{eq: strook-var-approx} equals
\begin{equation}\label{eq: lapl-prod-1}
\begin{aligned}
& \int_{\mathbb{R}^d} \psi_n(\xi_R v)(x) \, \xi_R(x) (-\Delta)^{s}(v)(x) \, \mathrm{d}x + \int_{\mathbb{R}^d} \psi_n(\xi_R v)(x) (-\Delta)^{s}(\xi_R)(x) v(x) \, \mathrm{d}x  \\
 & + 2 \, C_{d,s} \int_{\mathbb{R}^d} \psi_n(\xi_R v)(x) \int_{\mathbb{R}^d} \frac{(\xi_R(x)-\xi_R(y))(v(x)-v(y))}{|x-y|^{d+2s}} \, \mathrm{d}y  \mathrm{d}x \, .
\end{aligned}
\end{equation}
By dominated convergence,
\begin{equation*}%\label{eq: lapl-prod-2}
\lim_{R \to \infty} \int_{\mathbb{R}^d} \psi_n(\xi_R v)(x) \, \xi_R(x) (-\Delta)^{s}(v)(x) \, \mathrm{d}x = \int_{\mathbb{R}^d} \psi_n(v)(x) (-\Delta)^{s}(v)(x) \, \mathrm{d}x \, .
\end{equation*}
Our aim is to show that the other two integrals in \eqref{eq: lapl-prod-1} go to zero as $R\to \infty $. We have:
\begin{equation}\label{eq: lapl-prod-3}
\begin{aligned}
& \left| \int_{\mathbb{R}^d} \psi_n(\xi_R v)(x) (-\Delta)^{s}(\xi_R)(x) v(x) \, \mathrm{d}x \right|  \\
\le &  \left\| (-\Delta)^{s}(\xi_R) \right\|_\infty \left( \frac{d-2s}{d+2s} \int_{\{ v \le \frac1n \}} v^{\frac{2d}{d-2s}}(x)\, \mathrm{d}x + \psi_n(\| v \|_\infty)\| v \|_\infty \int_{\{ v > \frac1n \}} \mathrm{d}x \right)
\end{aligned}
\end{equation}
and
\begin{equation}\label{eq: lapl-prod-4}
\begin{aligned}
& \left| \int_{\mathbb{R}^d} \psi_n(\xi_R v)(x) \int_{\mathbb{R}^d} \frac{(\xi_R(x)-\xi_R(y))(v(x)-v(y))}{|x-y|^{d+2s}} \, \mathrm{d}y \mathrm{d}x \right| \\
\le & \left\| v \right\|_{\dot{H}^s} \left( \int_{\mathbb{R}^d} \left[ \psi_n(\xi_R v)(x) \right]^2 \int_{\mathbb{R}^d} \frac{(\xi_R(x)-\xi_R(y))^2}{|x-y|^{d+2s}} \, \mathrm{d}y \mathrm{d}x \right)^\frac12  \\
\le & \left\| v \right\|_{\dot{H}^s} \left\|  l_s(\xi_R) \right\|_\infty^{\frac12} \left( \left[\frac{d-2s}{d+2s}\right]^2 \int_{\{ v \le \frac1n \}} v^{2\frac{d+2s}{d-2s}}(x)\, \mathrm{d}x + \left[\psi_n(\| v \|_\infty)\right]^2 \int_{\{ v > \frac1n \}} \mathrm{d}x \right)^{\frac12} ,
\end{aligned}
\end{equation}
where $l_s $ is defined in Lemma \ref{lem:decay-lap-2}. Thanks to the scaling properties of both $(-\Delta)^s(\xi_R) $ and $l_s(\xi_R) $ (Lemma \ref{lem:decay-lap-cutoff}), it is immediate to check that $ \lim_{R\to \infty} \| (-\Delta)^s(\xi_R) \|_\infty = \lim_{R\to \infty} \| l_s(\xi_R) \|_\infty =0 $. Moreover, notice that $ v \in L^{\frac{2d}{d-2s}}(\mathbb{R}^d) \cap L^\infty(\mathbb{R}^d) $ (see \cite[Section 6]{hiker} or Lemma \ref{lem: CKN} below). In particular, $v$ also belongs to $L^{2\frac{d+2s}{d-2s}}(\mathbb{R}^d)$. Thus, letting $ R \to \infty $ in \eqref{eq: lapl-prod-3} and \eqref{eq: lapl-prod-4}, we deduce that the last two integrals in \eqref{eq: lapl-prod-1} vanish, so that we can pass to the limit on the l.h.s.\ of \eqref{eq: strook-var-approx}. On the r.h.s.\ we just use the fact that $ (-\Delta)^{\frac{s}{2}}(\Psi_n(\xi_R v)) $ converges to $ (-\Delta)^{\frac{s}{2}}(\Psi_n(v)) $ weakly in $ L^2(\mathbb{R}^d) $. This proves the validity of
\begin{equation}\label{eq: strook-var-approx-bis}
\int_{\mathbb{R}^d} \psi_n(v)(x) (-\Delta)^{s}(v)(x) \, \mathrm{d}x  \ge \int_{\mathbb{R}^d} \left|(-\Delta)^{\frac s 2}( \Psi_n(v) )(x) \right|^2 \mathrm{d}x \, .
\end{equation}
The final step is to let $ n \to \infty $ in \eqref{eq: strook-var-approx-bis}. It is clear that the sequence $\{ \psi_n(x) \}$ converges locally uniformly to the function $ x^{q-1} $, while $ \{ \Psi_n(x) \} $ converges locally uniformly to $ 2(q-1)^{\frac12} x^{\frac{q}{2}}/{q} $. Hence, $\{ \psi_n(v) \} $ and $\{ \Psi_n(v) \} $ converge in $ L^\infty(\mathbb{R}^d) $ to $ v^{q-1} $ and $ 2(q-1)^{\frac12} v^{\frac{q}{2}}/{q} $, respectively. This is enough in order to pass to the limit in \eqref{eq: strook-var-approx-bis} and obtain \eqref{eq: strook-vareq}.
\end{proof}

\begin{lem}\label{lem: CKN}
Let $ d > 2s $ and assume that $\rho$ satisfies \eqref{eq:
ass-rho} for some $\gamma \in [0,2s) \cap [0,d-2s]$ and $ \gamma_0
\in [0,\gamma] $. There exists a positive constant
$C_{CKN}=C_{CKN}(C,\gamma,s,d)$ such that the
Caffarelli-Kohn-Nirenberg-type inequalities
\begin{equation*}\label{CKN-peso}
\left\| v  \right\|_{q,\rho} \leq  C_{CKN} \left\|
(-\Delta)^{\frac s 2} (v) \right\|_{2}^{\frac{1}{\alpha+1}}
\left\| v \right \|_{p,\rho}^{\frac{\alpha}{\alpha+1}} \ \ \
\forall v \in L^p_{\rho}(\mathbb{R}^d) \cap
\dot{H}^s(\mathbb{R}^d)
\end{equation*}
hold for any $\alpha \ge 0$, $ p \ge 1 $ and $q=2 (d-\gamma)(\alpha+1)/\big[(d-\gamma)\frac{\alpha}p+ d-2s\big]$.
%For $ \alpha=0 $ one recovers the fractional Sobolev inequalities
%\begin{equation}\label{eq: sobo-frac}
%\left\| v \right\|_{2\frac{d-\gamma}{d-2s},-\gamma} \le C_S
%\left\| (-\Delta)^{\frac{s}{2}} (v) \right\|_2   \ \ \ \forall v
%\in \dot{H}^s(\mathbb{R}^d) \, .
%\end{equation}
\begin{proof} See e.g.~\cite[Theorem 1.8]{DAL}, where one considers the Sobolev inequality corresponding to $\alpha=0$ here, and then uses an elementary interpolation.
%Inequality \eqref{CKN-peso} is just a particular case of \cite[Th.
%1.8]{DAL}. Alternatively, one can prove it by interpolating
%between the fractional Sobolev inequality \eqref{eq: sobo-frac} in
%the case $ \gamma=0 $ (see \cite[Th. 6.5]{hiker}) and the
%fractional Hardy inequality (see e.g.\ \cite{frac-hardy} and
%references quoted therein)
%\begin{equation*}%\label{eq: hardy-frac}
%\left\| v \right\|_{2,-2s} \le C_H \left\| (-\Delta)^{\frac{s}{2}}
%(v) \right\|_2 \ \ \ \forall v \in \dot{H}^s(\mathbb{R}^d) \, .
%\end{equation*}
\end{proof}
\end{lem}
Lemmas \ref{lem: stroock-var} and \ref{lem: CKN} provide us with
some functional inequalities which are crucial to prove the
following {\it smoothing effect} for solutions to \eqref{eq:
barenblatt-u_0}.
\begin{pro}\label{lem: smoothing-effect}
Let $ d > 2s $ and assume that $\rho$ satisfies \eqref{eq:
ass-rho} for some $\gamma \in [0,2s) \cap [0,d-2s]$ and $ \gamma_0
\in [0,\gamma] $. There exists a constant $ K>0 $ depending only
on $ m $, $\gamma$, $s$, $ d $ and $C$ such that, for all nonnegative
initial datum $u_0 \in L^1_{\rho}(\mathbb{R}^d) \cap
L^\infty(\mathbb{R}^d) $ and the corresponding weak solution $u$
to \eqref{eq: barenblatt-u_0} constructed in Lemma \ref{lem:
prima-esistenza}, the following $L^{p_0}_{\rho}$--$L^\infty$
smoothing effect holds for any $p_0 \in [1,\infty) $:
\begin{equation}\label{eq: smoothing-effect}
 \left\| u(t) \right\|_\infty \le K \, t^{- \alpha_0 } \, \left\| u_0 \right\|_{p_0,\rho}^{\beta_0}   \ \ \ \forall t > 0 \, ,
\end{equation}
where
\begin{equation}\label{eq: smoothing-effect-exp}
\alpha_0 := \frac{d-\gamma}{(m-1)(d-\gamma) + (2s-\gamma)p_0 } \, , \
\ \ \beta_0 := \frac{(2s-\gamma)p_0}{(m-1)(d-\gamma) + (2s-\gamma)p_0
}  \, .
\end{equation}
\end{pro}
\begin{proof} We omit the details, since the claim follows as in \cite[Theorem 8.2]{DQRV} by means
of a standard parabolic Moser iteration. Nevertheless, notice that the proof relies on the Stroock-Varopoulos inequality (which has to hold for the precise set of functions stated in Lemma \ref{lem: stroock-var}), the Caffarelli-Kohn-Nirenberg type inequalities provided by Lemma \ref{lem: CKN} and the fact that the $L^p_{\rho}$ norms do not increase along the evolution (see Section \ref{sect: strong}).
\end{proof}
%In this respect, note that exploiting the
%Stroock-Varopoulos inequality \eqref{eq: strook-vareq} (let
%$v=u^m$ and $ q=(p+m-1)/m $), we obtain:
%\begin{equation}\label{eq: est-strong}
%%\begin{aligned}
%\int_{\mathbb{R}^d} u^p(x,t_2)  \, |x|^{-\gamma} \mathrm{d}x -
%\int_{\mathbb{R}^d} u^p(x,t_1)  \, |x|^{-\gamma} \mathrm{d}x  = -p
%\int_{t_1}^{t_2} \! \int_{\mathbb{R}^d} u^{p-1}(x,t)
%(-\Delta)^s(u^m)(x,t) \, \mathrm{d}x \, \mathrm{d}t \le 0
%%\end{aligned}
%\end{equation}
%for all $ t_2>t_1>0 $. In order to retrieve the case $ t_1=0 $ we
%cannot simply let $ t_1 \to 0 $ in \eqref{eq: est-strong}, since a
%priori we have no information about the continuity of $ \| u(t)
%\|_{p,-\gamma} $ down to $t=0 $. However, in order to consider
%$t_1=0$, we use a suitable approximation argument, involving the
%solutions $ u_\eta $ of Lemma \ref{lem: prima-esistenza}.

\subsection{Proof of the existence result}
We outline the main steps of this proof. Suppose first that $ \mu $ is compactly supported. %In order to construct solutions to problem
%\eqref{eq: barenblatt} for such a $ \mu $, the idea is to exploit
%the existence result provided by Lemma \ref{lem: prima-esistenza}.
Consider the family $\{u_\varepsilon \}$ of weak solutions to \eqref{eq: barenblatt-regular} that take
on the regular initial data $\mu_\varepsilon := \psi_\varepsilon
\ast \mu$ (let $\varepsilon>0$),
%\begin{equation*}%\label{eq: mollification-mu}
%\mu_\varepsilon = \psi_\varepsilon \ast \mu  \ \ \ \forall \varepsilon>0 \, ,
%\end{equation*}
where $ \psi_\varepsilon := \frac{1}{\varepsilon^d}
\psi\left(\frac{x}{\varepsilon} \right)$
with $ \psi \in \mathcal{D}_{+}(\mathbb{R}^d) $ and $ \| \psi \|_1=1 $. The existence of such family is ensured by Lemma \ref{lem: prima-esistenza}, upon setting $ u_0=\rho^{-1} \mu_\varepsilon $.
%What follows in the proof aims at showing that, as $ \varepsilon \to 0
%$, $ \{ u_\varepsilon \} $ suitably converges to a weak solution
%of \eqref{eq: barenblatt} starting from $ \mu $.
In view of certain a priori estimates (see \eqref{eq:
smooth+interp}, \eqref{eq: prima-esistenza-energy-1-epsilon} and
\eqref{eq: prima-esistenza-energy-2-epsilon} below), we prove that
$ \{ u_\varepsilon \} $ converges (up to subsequences), as $
\varepsilon\to0 $, to a function $ u $ which satisfies \eqref{eq:
weak-sol-spazi-u-1}, \eqref{eq: weak-sol-spazi-u-2} and \eqref{eq:
weak-sol-eq-debole}. Afterwards we deal with \eqref{eq:
weak-sol-initial}. To do this, we exploit some results in
potential theory, following \cite{Pierre} or \cite{VazConstr},
using the Riesz potential $ U_\varepsilon(\cdot,t)$ of $
\rho(\cdot) u_\varepsilon(\cdot,t)$. Then we let $\varepsilon\to
0$; in doing this, a uniform estimate w.r.t. $\varepsilon$ for the
potentials (see \eqref{eq: u-tempo-4-c} below) will be crucial.
Finally, we consider general positive finite Radon measures $\mu$, by a further
approximation.

\begin{proof}[Proof of Theorem \ref{thm: teorema-esistenza}]

%To this end, the idea is to exploit estimates
%\eqref{eq: prima-esistenza-energy-1} and \eqref{eq:
%prima-esistenza-energy-2} (with $u$ replaced by $ u_\varepsilon $)
%from Lemma \ref{lem: prima-esistenza}.
%However, notice that now a priori we cannot control any more the corresponding right hand sides as $ \varepsilon \to 0 $. In fact, \eqref{eq: prima-esistenza-energy-1} depends on
%\begin{equation}\label{eq: thm-exixst-dependence-1}
%\int_{\mathbb{R}^d} u_\varepsilon^{m+1}(x,t_1) \, |x|^{-\gamma} \mathrm{d}x \, ,
%\end{equation}
%while \eqref{eq: prima-esistenza-energy-2} depends on
%\begin{equation}\label{eq: thm-exixst-dependence-2}
%\int_{\mathbb{R}^d} u_\varepsilon^{m+1}(x,t_\ast) \, |x|^{-\gamma} \mathrm{d}x \, .
%\end{equation}
%Of course, still from \eqref{eq: prima-esistenza-energy-1}, we might bound from above both \eqref{eq: thm-exixst-dependence-1} and \eqref{eq: thm-exixst-dependence-2} by
%\begin{equation}\label{eq: thm-exixst-dependence-3}
%\int_{\mathbb{R}^d} u_\varepsilon^{m+1}(x,0) \, |x|^{-\gamma} \mathrm{d}x = \int_{\mathbb{R}^d} \mu_\varepsilon^{m+1}(x) \, |x|^{\gamma m} \mathrm{d}x \, .
%\end{equation}
%Nevertheless, \eqref{eq: thm-exixst-dependence-3} in general blows up as $ \varepsilon \to 0 $ (think, for instance, of the case where $ \mu $ is a Dirac $ \delta $). It is actually here that the smoothing effect \eqref{eq: smoothing-effect} takes a crucial role.
For any $\varepsilon>0$, let $u_\varepsilon$ be as above.
Combining the smoothing effect \eqref{eq: smoothing-effect} with the fact that $ \left\| \mu_\varepsilon \right\|_1 = \mu(\mathbb{R}^d)$ and with the conservation of mass \eqref{eq: cons-mass}, we obtain:
\begin{equation}\label{eq: smooth+interp}
\begin{aligned}
\int_{\mathbb{R}^d} u_\varepsilon^{m+1}(x,t) \, \rho(x) \mathrm{d}x   \le  \left\| u_\varepsilon(t) \right\|_\infty^{m} \, \left\| \mu_\varepsilon \right\|_1
\le   K^m \, t^{- \alpha m } \, \mu(\mathbb{R}^d)^{1 + \beta m}
\end{aligned}
\end{equation}
for all $t>0$. Hence, using \eqref{eq: prima-esistenza-energy-1},
\eqref{eq: prima-esistenza-energy-2} and \eqref{eq: smooth+interp}
%(evaluated at $ t=t_1 $ and $ t=t_\ast $)
we get:
%the validity of
%the following energy estimates:
\begin{gather}
\int_{t_1}^{t_2} \! \int_{\mathbb{R}^d} \left| (-\Delta)^{\frac{s}{2}} \left( u_\varepsilon^m \right) (x,t) \right|^2
\mathrm{d}x \mathrm{d}t + \int_{\mathbb{R}^d} u_\varepsilon^{m+1}(x,t_2) \, \rho(x) \mathrm{d}x  \le K^m \, t_1^{- \alpha m } \, \mu(\mathbb{R}^d)^{1 + \beta m} \, ,
\label{eq: prima-esistenza-energy-1-epsilon}  \\
\int_{t_1}^{t_2} \! \int_{\mathbb{R}^d} \left|
\left(z_\varepsilon\right)_t(x,t)  \right|^2 \rho(x)
\mathrm{d}x \mathrm{d}t  \le C^\prime \int_{\mathbb{R}^d}u^{m+1}\left(x,{t_1}/2\right)\, \rho(x) \mathrm{d}x\label{eq: prima-esistenza-energy-2-epsilon}
\end{gather}
for all $t_2 > t_1 > 0 $, where
$ z_\varepsilon := u_\varepsilon^{\frac{m+1}{2}} $ and $C^\prime$ is a
positive constant that depends on $m$, $ t_1 $, $ t_2 $ but is independent of $ \varepsilon $. Thanks
to \eqref{eq: prima-esistenza-energy-1-epsilon}, \eqref{eq:
prima-esistenza-energy-2-epsilon}, the conservation of mass and the smoothing effect (which,
in particular, bounds $ \{u_\varepsilon\} $ in $
L^\infty(\mathbb{R}^d \times (\tau,\infty) )$ for all $ \tau>0 $
independently of $ \varepsilon$), we are allowed to proceed
exactly as in the proof of Lemma \ref{lem: prima-esistenza}. That
is, we obtain that the pointwise limit $ u $ of $ \{ u_\varepsilon
\} $, up to subsequences, satisfies \eqref{eq: weak-sol-spazi-u-1}, \eqref{eq:
weak-sol-spazi-u-2} and \eqref{eq: weak-sol-eq-debole}.
%\begin{gather}\label{eq: formDebole_approx_limite-epsilon-equazione}
% - \int_{0}^{\infty} \! \int_{\mathbb{R}^d}  u(x,t) \varphi_t(x,t) \, |x|^{-\gamma} \mathrm{d}x \, \mathrm{d}t + \int_0^\infty \! \int_{\mathbb{R}^d} (-\Delta)^{\frac{s}{2}} (u^m)(x,t) \, (-\Delta)^{\frac{s}{2}} \varphi(x,t) \,  \mathrm{d}x \, \mathrm{d}t = 0 \\
%\forall \varphi \in C^\infty_c( \mathbb{R}^d \times (0,\infty) )
%\, . \nonumber
%\end{gather}
%Notice that we cannot pass to the limit directly in the stronger
%weak formulation \eqref{eq: formDebole-approx-1-limite}: the
%problem is that estimate \eqref{eq:
%prima-esistenza-energy-1-epsilon} blows up as $ t_1 \rightarrow 0
%$. Hence, we can only take test functions which are compactly
%supported in $ \mathbb{R}^d \times (0,\infty) $. In particular,
%\eqref{eq:
%weak-sol-eq-debole} does not
%provide any information over the initial datum assumed by $ u(x,t)
%$.

 Let us now introduce the Riesz potential $
U_\varepsilon(\cdot,t)$ of $ \rho(\cdot) u_\varepsilon(\cdot,t)$.
%\begin{equation*}%\label{eq: def-potenziale-uepsilon}
%U_\varepsilon(t) = I_{2s} \ast \left(|x|^{-\gamma} u_\varepsilon(t) \right) \ \ \ \forall t>0 \, .
%\end{equation*}
The equation solved by $ u_\varepsilon $ is
\begin{equation}\label{eq: eq-formale-epsilon}
\rho(x)(u_\varepsilon)_t(x,t) =
-\left(-\Delta\right)^{s}(u^m_\varepsilon)(x,t) \ \ \ \forall
(x,t) \in \mathbb{R}^d \times \mathbb{R}^+ .
\end{equation}
Applying to both sides of \eqref{eq: eq-formale-epsilon} the operator $(-\Delta)^{-s} $, namely the convolution against the Riesz kernel $ I_{2s} $ (recall the discussion in Section \ref{sec: not-def}), formally yields
\begin{equation}\label{eq: eq-formale-potenziale-epsilon}
\left(U_{\varepsilon}\right)_t(x,t) = -u^m_\varepsilon(x,t) \ \ \ \forall (x,t) \in \mathbb{R}^d \times \mathbb{R}^+ \, .
\end{equation}
To prove rigorously \eqref{eq: eq-formale-potenziale-epsilon}, we plug into \eqref{eq: weak-sol-eq-debole} (with $ u = u_\varepsilon $) the test function $ \varphi(y,t) := \vartheta(t)\phi(y) $, where $ \vartheta $ is a smooth and compactly supported approximation of $ \chi_{[t_1,t_2]} $ and $ \phi \in \mathcal{D}(\mathbb{R}^d) $. Integrating by parts (in space), letting $ \vartheta $ tend to $ \chi_{[t_1,t_2]} $ and replacing the function $ \phi(y) $ by $\phi(y+x)$, with $x \in \mathbb{R}^d $ fixed, we get:
% come quasi ovunque basta prendere per esempio quello dei punti di Lebegue di u come curva a valori in L^2_{|x|^-\gamma}
%\begin{equation}\label{eq: u-tempo-a}
%%\begin{aligned}
% \int_{\mathbb{R}^d} u_\varepsilon(y,t_2) \phi(y) \, |y|^{-\gamma} \mathrm{d}y - \int_{\mathbb{R}^d}  u_\varepsilon(y,t_1) \phi(y) \, |y|^{-\gamma} \mathrm{d}y \\
%=  - \int_{\mathbb{R}^d} \left(\int_{t_1}^{t_2} u^m_\varepsilon(y,t) \, \mathrm{d}t \right) (-\Delta)^s(\phi)(y) \, \mathrm{d}y  \, .
%%\end{aligned}
%\end{equation}
%For any fixed $ x \in \mathbb{R}^d$ we can replace in \eqref{eq: u-tempo-a} the function $ \phi(y) $ by $\phi(y+x)$, thus obtaining
\begin{equation}\label{eq: u-tempo-bis-a}
\begin{aligned}
 & \int_{\mathbb{R}^d} u_\varepsilon(y,t_2) \phi(y+x) \, \rho(y) \mathrm{d}y - \int_{\mathbb{R}^d} u_\varepsilon(y,t_1) \phi(y+x) \, \rho(y) \mathrm{d}y  \\
= & - \int_{\mathbb{R}^d} \left(\int_{t_1}^{t_2} u^m_\varepsilon(y,t) \, \mathrm{d}t \right) (-\Delta)^s(\phi)(y+x) \, \mathrm{d}y \, .
\end{aligned}
\end{equation}
Integrating \eqref{eq: u-tempo-bis-a} against the Riesz kernel $ I_{2s}(x) $ and using Fubini's Theorem gives (let $z=y+x$)
%\begin{equation}\label{eq: u-tempo-ter-a}
%\begin{aligned}
%& \int_{\mathbb{R}^d} \int_{\mathbb{R}^d} |y|^{-\gamma} u_\varepsilon(y,t_2) \phi(y+x) I_{2s}(x) \, \mathrm{d}y \, \mathrm{d}x - \int_{\mathbb{R}^d} \int_{\mathbb{R}^d} |y|^{-\gamma} u_\varepsilon(y,t_1) \phi(y+x) I_{2s}(x) \, \mathrm{d}y \, \mathrm{d}x  \\
%= & - \int_{\mathbb{R}^d} \int_{\mathbb{R}^d} \left(\int_{t_1}^{t_2} u^m_\varepsilon(y,t) \, \mathrm{d}t \right) (-\Delta)^s(\phi)(y+x) I_{2s}(x) \, \mathrm{d}y \, \mathrm{d}x \, ,
%\end{aligned}
%\end{equation}
\begin{equation}\label{eq: u-tempo-4-a}
\begin{aligned}
 & \int_{\mathbb{R}^d} U_\varepsilon(z,t_2) \phi(z) \, \mathrm{d}z - \int_{\mathbb{R}^d} U_\varepsilon(z,t_1) \phi(z) \, \mathrm{d}z \\
= & - \int_{\mathbb{R}^d} \left(\int_{t_1}^{t_2} u^m_\varepsilon(y,t) \, \mathrm{d}t \right) \left( \int_{\mathbb{R}^d}  (-\Delta)^s(\phi)(y+x) I_{2s}(x) \, \mathrm{d}x  \right) \mathrm{d}y =  - \int_{\mathbb{R}^d} \left(\int_{t_1}^{t_2} u^m_\varepsilon(y,t) \, \mathrm{d}t \right) \phi(y) \, \mathrm{d}y \, .
\end{aligned}
\end{equation}
% Osservo che, evidentemente, l'equazione risolta da U mi dice che U(x,t^\prime)-U(x,s^\prime) \le M |t^\prime-s^\prime| per q.o. x, ma evidentemente il q.o. diventa poi "ogni" per continuità. Inoltre, la versione assolutamente continua in tempo di U conserva la stessa holderianità spaziale delle U potenziali: infatti \int_{s^\prime}^{t^prime}(u^m) è evidentemente holderiano!
%\begin{equation}\label{eq: u-tempo-giust-1-a}
%\int_{\mathbb{R}^d} \int_{\mathbb{R}^d} \left| |y|^{-\gamma} u_\varepsilon(y,t) \phi(y+x) I_{2s}(x) \right| \mathrm{d}y \, \mathrm{d}x < \infty \ \ \ t \in \{t_1,t_2 \}
%\end{equation}
%and
%\begin{equation}\label{eq: u-tempo-giust-2-a}
%\int_{\mathbb{R}^d} \int_{\mathbb{R}^d} \left| \left(\int_{t_1}^{t_2} u^m_\varepsilon(y,t) \, \mathrm{d}t \right) (-\Delta)^s(\phi)(y+x) I_{2s}(x) \right| \mathrm{d}y \, \mathrm{d}x < \infty \, .
%\end{equation}
%In fact, both the functions
%\begin{equation}\label{eq: func-scambio-int-1}
%y \rightarrow \int_{\mathbb{R}^d} |\phi(y+x)| I_{2s}(x) \, \mathrm{d}x
%\end{equation}
%and
%\begin{equation}\label{eq: func-scambio-int-2}
%y \rightarrow \int_{\mathbb{R}^d} |(-\Delta)^s(\phi)(y+x)| I_{2s}(x) \, \mathrm{d}x
%\end{equation}
The applicability of Fubini's Theorem is justified thanks to Lemma \ref{lem: decay-conv}, Lemma \ref{lem:decay-lap-1} (recall that $d-2s \ge \gamma$ by assumption)
and to the fact that
%(recall that $ (-\Delta)^s(\phi)(y) $ is regular and decays at least like $|y|^{-d-2s}$ as $ |y| \to \infty $, see Lemma \ref{lem:decay-lap-1}),
%In particular, \eqref{eq: func-scambio-int-1} is bounded and in \eqref{eq: u-tempo-giust-1-a} is integrated against $|y|^{-\gamma}u_\varepsilon(t)$, an %$L^1(\mathbb{R}^d)$ function. As concerns \eqref{eq: func-scambio-int-2}, we see that in \eqref{eq: u-tempo-giust-2-a} it is integrated against the function
$\int_{t_1}^{t_2} u^m_\varepsilon(\cdot,t) \, \mathrm{d}t$ belongs to $ L^1_{\rho}(\mathbb{R}^d) \cap L^\infty(\mathbb{R}^d)$ by \eqref{eq: weak-sol-spazi-u-1}.

By Lemma \ref{lem: prima-esistenza} and Definition \ref{eq:
weak-sol-1}, we know that $ \rho u_\varepsilon(t) $
 converges to $ \mu_\varepsilon $ in $ \sigma(\mathcal{M}(\mathbb{R}^d),C_b(\mathbb{R}^d)) $ as $ t \to 0 $. Hence, letting $ t_1 \to 0 $ in \eqref{eq: u-tempo-4-a}, we find that
\begin{equation}\label{eq: u-tempo-4-b}
\int_{\mathbb{R}^d} U_\varepsilon(x,t_2) \phi(x) \, \mathrm{d}x -
\int_{\mathbb{R}^d} U^{\mu_\varepsilon}(x) \phi(x) \, \mathrm{d}x
= - \int_{\mathbb{R}^d} \left(\int_{0}^{t_2} u^m_\varepsilon(x,t)
\, \mathrm{d}t \right) \phi(x) \, \mathrm{d}x
\end{equation}
for all $ t_2>0 $ and $ \phi \in \mathcal{D}(\mathbb R^d) $. In
fact,
\begin{equation*}\label{eq: u-tempo-4-temp}
\begin{aligned}
\int_{\mathbb{R}^d} U_\varepsilon(x,t_1) \phi(x) \, \mathrm{d}x =
& \int_{\mathbb{R}^d} \left( \int_{\mathbb R^d} I_{2s}(x-y) \,
\rho(y)u_\varepsilon(y,t_1) \,
\mathrm{d}y \right) \phi(x) \, \mathrm{d}x \\
 = &  \int_{\mathbb{R}^d} \underbrace{\left( \int_{\mathbb R^d} I_{2s}(y-x) \, \phi(x) \, \mathrm{d}x \right)}_{U^\phi(y)} \rho(y) u_\varepsilon(y,t_1) \, \mathrm{d}y \, ,
\end{aligned}
\end{equation*}
and in view of Lemma \ref{lem: decay-conv} we know that, in
particular, $ U^\phi \in C_0(\mathbb R^d) $, which allows to pass
to the limit in the integral as $ t_1 \to 0 $. Thanks to the
smoothing effect, the conservation of mass and the hypotheses on $
\rho $, we can provide the following bound for \eqref{eq:
u-tempo-4-b}:
\begin{equation}\label{eq: u-tempo-4-c}
\left| \int_{\mathbb{R}^d} U_\varepsilon(x,t_2) \phi(x) \,
\mathrm{d}x - \int_{\mathbb{R}^d} U^{\mu_\varepsilon}(x) \phi(x)
\, \mathrm{d}x \right| \le \left\| \rho^{-1} \phi \right\|_\infty
K^{m-1} \, \mu(\mathbb{R}^d)^{1+\beta(m-1)} \int_{0}^{t_2}
t^{-\alpha(m-1)} \mathrm{d}t \, .
\end{equation}
Note that the time integral in the r.h.s.\ is finite since $
\alpha(m-1) < 1 $ (recall \eqref{eq: smoothing-effect-exp} for
$p_0=1$). We proved above that $\{u_\varepsilon\}$ converges
pointwise a.e.\ (up to subsequences) to a function $u$ which
satisfies \eqref{eq: weak-sol-spazi-u-1}, \eqref{eq:
weak-sol-spazi-u-2} and \eqref{eq: weak-sol-eq-debole}. If we
exploit once again the smoothing effect and the conservation of
mass, we easily infer that such convergence also takes place in $
\sigma(\mathcal{M}(\mathbb{R}^d),C_0(\mathbb{R}^d))$:
\begin{equation}\label{eq: weak-sol-epsilon}
\lim_{\varepsilon \to 0} \rho u_\varepsilon(t) = \rho u(t)  \ \ \
\textrm{in} \ \sigma(\mathcal{M}(\mathbb{R}^d),C_0(\mathbb{R}^d))
\, , \ \textrm{for a.e.} \ t > 0 \, .
\end{equation}
Using \eqref{eq: weak-sol-epsilon}, the fact that $
\mu_\varepsilon \to \mu $ in $
\sigma(\mathcal{M}(\mathbb{R}^d),C_b(\mathbb{R}^d)) $  and
proceeding exactly as we did in the proof of \eqref{eq:
u-tempo-4-b}, we can let $ \varepsilon \to 0 $ in \eqref{eq:
u-tempo-4-c} to get
\begin{equation}\label{eq: u-tempo-4-d}
\left| \int_{\mathbb{R}^d} U(x,t_2) \phi(x) \, \mathrm{d}x -
\int_{\mathbb{R}^d} U^{\mu}(x) \phi(x) \, \mathrm{d}x \right| \le
\left\| \rho^{-1} \phi \right\|_\infty  K^{m-1} \,
\mu(\mathbb{R}^d)^{1+\beta(m-1)} \int_{0}^{t_2} t^{-\alpha(m-1)}
\mathrm{d}t
\end{equation}
for a.e.\ $ t_2>0 $ and $ \phi \in \mathcal{D}(\mathbb R^d) $,
where we denote as $ U $ the potential of $ \rho u $. Note that,
passing to the limit in \eqref{eq: u-tempo-4-b} for any
nonnegative $ \phi \in \mathcal{D}(\mathbb R^d) $, we deduce in
particular that $ U(x,t) $ is nonincreasing in $t$. Moreover,
\eqref{eq: u-tempo-4-d} implies that $ U(t) $ converges to $
U^{\mu} $ in $ L^1_{\rm loc}(\mathbb R^d) $, whence
\begin{equation}\label{eq: dis-L1-loc-smooth-limite-3}
\lim_{t \to 0} U(x,t) = U^\mu(x)  \ \ \ \textrm{for a.e.} \ x \in
\mathbb{R}^d \, .
\end{equation}
Letting $ \varepsilon \to 0 $ in the conservation of mass
\eqref{eq: cons-mass} (applied to $ u=u_\varepsilon $ and $
\mu=\mu_\varepsilon $), by means e.g.~of Fatou's Lemma we
obtain
\begin{equation}\label{eq: quasi-cons-mass}
\left\| u(t) \right\|_{1,\rho} \le \mu(\mathbb{R}^d) \ \ \
\textrm{for a.e.} \ t>0 \, .
\end{equation}
Due to the compactness results recalled in Section \ref{sec:
not-def}, from \eqref{eq: quasi-cons-mass} we infer that (almost)
every sequence $ t_n \to 0 $ admits a subsequence $ \{ t_{n_k} \}
$ such that $ \{ \rho u(t_{n_k}) \} $ converges to a certain
positive finite Radon measure $ \nu $ in
$\sigma(\mathcal{M}(\mathbb{R}^d),C_c(\mathbb{R}^d))$. Thanks to
\eqref{eq: dis-L1-loc-smooth-limite-3} and \cite[Theorem
3.8]{Landkof} we have that $U^{\nu}(x) = U^{\mu}(x)$ almost
everywhere. Alternatively, such identity can be proved by passing
to the limit in $ \int_{\mathbb{R}^d} U(x,t_{n_k}) \phi(x) \,
\mathrm{d}x $, recalling that $ U(t_{n_k}) \to U^\mu $ in $
L^1_{\rm loc}(\mathbb R^d) $ as $ k \to \infty $. Theorem 1.12 of
\cite{Landkof} then ensures that two positive finite Radon
measures whose potentials are equal almost everywhere must
coincide. Hence, $ \nu=\mu $ and the limit measure does not depend
on the particular subsequence, so that
\begin{equation*}%\label{eq: cond-iniziale-finale}
\lim_{t \to 0} \rho u(t)  = \mu  \ \ \ \textrm{in} \
\sigma(\mathcal{M}(\mathbb{R}^d),C_c(\mathbb{R}^d)) \, .
\end{equation*}
In order to show that convergence also takes place in $
\sigma(\mathcal{M}(\mathbb{R}^d),C_b(\mathbb{R}^d)) $, it is
enough to establish that
\begin{equation}\label{eq: weak-2}
\lim_{t \to 0} \left\| u(t) \right\|_{1,\rho} = \mu(\mathbb{R}^d)
\, .
\end{equation}
Since $ \rho u(t) $ converges to $ \mu $ in $
\sigma(\mathcal{M}(\mathbb{R}^d),C_c(\mathbb{R}^d)) $ as $ t \to 0
$, we know that
\begin{equation}\label{eq: weak-3}
\mu(\mathbb{R}^d) \le \liminf_{t \to 0} \left\| u(t)
\right\|_{1,\rho} ,
\end{equation}
see again Section \ref{sec: not-def}. But \eqref{eq: weak-3} and
\eqref{eq: quasi-cons-mass} entail \eqref{eq: weak-2}.

Finally, the validity of the smoothing estimate \eqref{eq:
smoothing-effect-general} is just a consequence of passing to the
limit in \eqref{eq: smoothing-effect} (applied to $u_\varepsilon$ and $p_0=1$) as $
\varepsilon \to 0 $ (recall that $ \{u_\varepsilon\} $ converges pointwise to $u$).

At the beginning of the proof we required $ \mu $ to be compactly
supported. Otherwise, take
a sequence of compactly supported measures $ \{ \mu_n \}  $
converging to $ \mu $ in $
\sigma(\mathcal{M}(\mathbb{R}^d),C_b(\mathbb{R}^d)) $  and
consider the corresponding sequence of solutions $ \{ u_n \} $ to
\eqref{eq: barenblatt-regular}. Estimates \eqref{eq:
prima-esistenza-energy-1-epsilon} and \eqref{eq:
prima-esistenza-energy-2-epsilon}, as well as the conservation of mass and the smoothing effect, are clearly stable as
$ \varepsilon \to 0 $, thus they also hold upon replacing $
u_\varepsilon $ with $ u_n $ and $ \mu_\varepsilon $ with $ \mu_n
$. Hence, using the same techniques as above, one proves that $\{ u_n \} $ converges to a solution $ u $
of \eqref{eq: barenblatt-regular} starting from $ \mu $.
\end{proof}

\subsection{Existence and uniqueness of initial traces}
In order to prove Theorem \ref{thm: initial}, we need the next preliminary result.
\begin{lem}\label{eq: lemma-massa-potenziali}
Let $ \nu $ be a signed finite Radon measure such that $ U^{\nu} \ge 0 $ almost everywhere. Then $ \nu(\mathbb{R}^d) \ge 0 $.
\begin{proof}
From the assumptions on $ U^{\nu} $ and thanks to Fubini's Theorem, there holds
\begin{equation}\label{eq: hp-U-nu}
\int_{\mathbb{R}^d} \chi_{B_n}(y) \, U^{\nu}(y) \, \mathrm{d}y = \int_{\mathbb{R}^d} \left( I_{2s} \ast \chi_{B_n} \right)(x)  \, \mathrm{d}\nu = k_{d,s} \int_{\mathbb{R}^d} \left( \int_{B_n} |x-y|^{-d+2s} \, \mathrm{d}y  \right) \mathrm{d}\nu \ge 0 \quad \forall n \in \mathbb{N} \, .
\end{equation}
Performing the change of variable $ z=y/n $, the last inequality in \eqref{eq: hp-U-nu} reads
\begin{equation}\label{eq: hp-U-nu-2}
\int_{\mathbb{R}^d} \left( \int_{B_1} |x/n-z|^{-d+2s} \, \mathrm{d}z  \right) \mathrm{d}\nu \ge 0 \quad \forall n \in \mathbb{N} \, .
\end{equation}
It is plain that for \emph{every} $x \in \mathbb{R}^d $ the sequence $ \{ \int_{B_1} |x/n-z|^{-d+2s} \, \mathrm{d}z \}  $ converges to the positive constant $ \int_{B_1} |z|^{-d+2s} \, \mathrm{d}z $ and it is dominated by the latter. Passing to the limit as $ n \to \infty $ in \eqref{eq: hp-U-nu-2}, we get the assertion by dominated convergence (recall that $ \nu $ is finite).
\end{proof}
\end{lem}

\begin{proof}[Proof of Theorem \ref{thm: initial}]\rm
Consider a function $u$ satisfying \eqref{eq: weak-sol-spazi-u-1}, \eqref{eq: weak-sol-spazi-u-2} and \eqref{eq: weak-sol-eq-debole}. Monotonicity in time of the associated potential is proved as we did after \eqref{eq: eq-formale-potenziale-epsilon}: notice that, for such an argument to work, the running assumptions on $\gamma$ are required. The same proof holds if, instead of $ u \in L^\infty( \mathbb{R}^d \times (\tau , \infty ) ) $, $u$ is only supposed to satisfy $\int_{t_1}^{t_2} u^m(\cdot,\tau) \, \mathrm{d}\tau\in L^1_\rho({\mathbb R}^d)$ for all $t_2>t_1>0$. Existence of an initial trace $ \mu $, meant as convergence in $\sigma(\mathcal{M}(\mathbb{R}^d),C_c(\mathbb{R}^d)) $ along subsequences of a given sequence of times tending to $t=0$, follows by compactness, since we are assuming that solutions belong to $L^\infty( (0,\infty); L^1_{\rho}(\mathbb{R}^d) )$.
%The extension to convergence in $\sigma(\mathcal{M}(\mathbb{R}^d),C_0(\mathbb{R}^d))$ follows by approximation.
Uniqueness of such a trace is established proceeding as we did after \eqref{eq: dis-L1-loc-smooth-limite-3}, using the monotonicity of potentials and the results of \cite{Landkof}.

We are left with proving that convergence to $ \mu $ takes places also in $\sigma(\mathcal{M}(\mathbb{R}^d),C_b(\mathbb{R}^d)) $, namely that $ \operatorname{ess}\lim_{t \to 0} \int_{\mathbb{R}^d} u(x,t) \, \rho(x)\mathrm{d}x = \mu(\mathbb{R}^d) $. By weak$^\ast$ lower semi-continuity, it is then enough to show that $\operatorname{ess}\limsup_{t \to 0} \int_{\mathbb{R}^d} u(x,t) \, \rho(x)\mathrm{d}x \le \mu(\mathbb{R}^d)$. Let $ U(\cdot,t) $ be the potential of $ \{ \rho(\cdot) u(\cdot,t) \} $. Again, the monotonicity in time of $ U(\cdot,t) $ and the first part of the proof ensure that $ U^\mu - U(\cdot,t) \ge 0 $ almost everywhere. Therefore, Lemma \ref{eq: lemma-massa-potenziali} applied to the signed finite Radon measure $ \mathrm{d}\nu=\mathrm{d}\mu-u(x,t)\rho(x)\mathrm{d}x $ entails $ \mu(\mathbb{R}^d) \ge \int_{\mathbb{R}^d} u(x,t) \, \rho(x)\mathrm{d}x $. Letting $ t \to 0 $ concludes the proof.

%In order to prove the last statement, it is enough to show that $ \lim_{t \to 0} \int_{\mathbb{R}^d} u(x,t) \, \rho(x)\mathrm{d}x = \mu(\mathbb{R}^d) $ for all $ t>0 $. To this end, we proceed exactly as in the proof of Proposition \ref{oss: cons-mass}: once we have obtained estimate \eqref{eq: uniq-trace-cons-mass}, the result follows by letting $ t_1 \to 0 $, using the first part of the theorem and the assumption that $\int_{0}^{t} u^m(\cdot,\tau)\,\mathrm{d}\tau\in L^1_\rho({\mathbb R}^d)$ for all $t>0$ and finally letting $ R \to \infty $.
%we plug into \eqref{eq: weak-sol-eq-debole} the test function $\varphi_{R}(x,t) = \vartheta(t) \xi_R(x)$, where $ \xi_R $ is a cut-off function and $ \vartheta$ is a suitable positive, regular and compactly supported approximation of $ \chi_{[t_1,t_2]}(t)$ (let $ t_2>t_1>0 $). Using decay and scaling properties of the fractional Laplacian (Lemmas \ref{lem:decay-lap-1} and \ref{lem:decay-lap-cutoff}) and letting $ \vartheta \to \chi_{[t_1,t_2]} $, we obtain the estimate
%\begin{equation}\label{eq: uniq-trace-1}
%\begin{aligned}
%& \left| \int_{\mathbb{R}^d}  u(x,t_2) \xi_R(x) \, |x|^{-\gamma} \mathrm{d}x - \int_{\mathbb{R}^d}  u(x,t_1) \xi_R(x) \, |x|^{-\gamma} \mathrm{d}x  \, \right| \\
%\le & \frac{\left\| |x|^\gamma (-\Delta)^{s}(\xi) \right\|_\infty }{R^{2s-\gamma}} \int_{t_1}^{t_2} \! \int_{\mathbb{R}^d} u^m(x,t) \, |x|^{-\gamma} \mathrm{d}x \, \mathrm{d}t \,  .
%\end{aligned}
% \end{equation}
\end{proof}

\subsection{Strong solutions and decrease of the norms}\label{sect: strong}
In order to justify rigorously some of the above computations,
% (in particular, we refer to the proofs of Lemma \ref{lem: prima-esistenza} and Proposition \ref{lem: smoothing-effect})
it is essential to show that the weak solutions constructed in Lemma \ref{lem: prima-esistenza} are strong. By a ``strong solution'', following \cite[Section 6.2]{DQRV}, we mean a weak solution $u$ % (in the sense of Definition \ref{eq: weak-sol-1})
such that
$u_t \in L^\infty((\tau,\infty),L^1_{\rho}(\mathbb{R}^d))$ for all $\tau>0$.
The fact that our solutions are indeed strong can be proved as in \cite[Section 8.1]{DQRV}. The first step consists in showing that $\rho(\cdot) u_t(\cdot,t)$ is a finite Radon measure which satisfies the estimate
\begin{equation}\label{eq: stima-Radon-1}
\left\| \rho \, u_t(t) \right\|_{\mathcal{M}(\mathbb{R}^d)} \le \frac{2 \left\| u_0 \right\|_{1,\rho} }{ (m-1) t }  \ \ \ \forall t>0 \, ,
\end{equation}
where now, by $\mathcal{M}(\mathbb{R}^d)$ we mean the Banach space of \emph{signed} finite Radon measures on $\mathbb{R}^d$, equipped with the usual norm of the variation. As in \cite[Lemma 8.5]{Vaz07}, this follows by using the inequality
% usare ad es. th. estensione limitata per dire che la stima vale q.o., solita continuità a priori delle solzuioni in L^1 "locale"
\begin{equation}\label{eq: L1-cont}
\int_{\mathbb{R}^d} \left[ u(x,t)- \tilde{u}(x,t) \right]_{+} \, \rho(x) \mathrm{d}x
\leq \int_{\mathbb{R}^d}  \left[ u_0(x)- \tilde{u}_0(x) \right]_{+} \, \rho(x) \mathrm{d}x \ \ \ \forall t>0 \, ,
\end{equation}
where $u$ and $\tilde{u}$ are the solutions to \eqref{eq: barenblatt-u_0} \emph{constructed in Lemma \ref{lem: prima-esistenza}} corresponding to the initial data $u_0 $ and $\tilde{u}_0$, respectively. Such inequality does hold for the approximate solutions $u_\eta $ and $\tilde{u}_\eta $ used in the proof of Lemma \ref{lem: prima-esistenza} (see \cite[Proposition 3.4]{PT2}), whence \eqref{eq: L1-cont} follows by passing to the limit.
%\begin{equation}\label{eq: L1-cont-approx}
%\int_{\mathbb{R}^d} \left[ u_\eta(x,t)- \tilde{u}_\eta(x,t) \right]_{+} \, \rho_\eta(x) \mathrm{d}x
%\leq \int_{\mathbb{R}^d}  \left[ u_0(x)- \tilde{u}_0(x) \right]_{+} \, \rho_\eta(x) \mathrm{d}x \ \ \ \forall t>0 \, .
%\end{equation}
% This is proved in \cite[Prop. 3.3]{PT1}. Hence, \eqref{eq: L1-cont} is just a consequence of passing to the limit in \eqref{eq: L1-cont-approx} as $ \eta \to 0$.
% Di fatto vazquez usa in maniera fondamentale solo il fatto che tra le soluzioni u(x,t) e \lambda*u(x,\lambda^{m-1}*t) valga il principio di confronto L^1_\gamma. Questo è garantito da Punzo-Terrone per le soluzioni u_\eta, e quindi passando al limite anche per le nostre u.
Afterwards, as \cite[Lemma 8.1]{DQRV}, one proves that $ z:=u^{\frac{m+1}{2}} $ fulfills \eqref{eq: prima-esistenza-energy-2}. In particular,
\begin{equation}\label{eq z-L2-est}
z_t \in L^2_{\rm loc}((0,\infty); L^2_{\rho}(\mathbb{R}^d)) \, .
\end{equation}
%The dependence of the constant $ C $ in \eqref{eq: prima-esistenza-energy-2} on the initial datum as in \eqref{eq: prima-esistenza-energy-dependence} is then a consequence of the energy identity \eqref{eq: prima-esistenza-energy-1} (the proof of which requires however that solutions are strong, see Section \ref{sect: decrease-norms} below).
Thanks to \eqref{eq: stima-Radon-1} and \eqref{eq z-L2-est}, the abstract result contained in \cite[Theorem 1.1]{BG} ensures that $u_t  \in L^1_{\rm loc}((0,\infty); L^1_{\rho}(\mathbb{R}^d))$.
In particular, \eqref{eq: stima-Radon-1} holds with $ \| \rho \, u_t(t) \|_{\mathcal{M}(\mathbb{R}^d)} $ replaced by $ \| u_t(t) \|_{1,\rho} $, whence the assertion.
%Thanks to \eqref{eq: stima-Radon-1} and \eqref{eq: ut-L1-loc} we then get the estimate
%\begin{equation*}%\label{eq: stima-Radon-L1}
%\left\| u_t(t) \right\|_{1,-\gamma} \le \frac{2 \left\| u_0 \right\|_{1,-\gamma} }{ (m-1) t } \, .
%\end{equation*}
%In particular, \eqref{eq: def-strong} holds true and solutions are strong.
% remark su soluzioni strong-decay anche per le soluzioni costruite nel teorema principale, modulo l'unicità che si dimostra dopo
%\begin{oss}%\label{oss: strong-measure}
%\rm
%We have shown that the weak solutions to \eqref{eq: barenblatt-u_0} constructed in Lemma \ref{lem: prima-esistenza} are strong. Since, for any $ \tau>0 $, every weak solution $ u $ to \eqref{eq: barenblatt-regular} provided by Theorem \ref{thm: teorema-esistenza} is a weak solution to \eqref{eq: barenblatt-u_0} corresponding to the initial datum $ u(\tau) \in L^1_{\rho}(\mathbb{R}^d) \cap L^\infty(\mathbb{R}^d) $, one may claim that also such $ u $ is a strong solution. In order to prove this rigorously, we need however the uniqueness Theorem \ref{thm: teorema-uniqueness}.
%% (see Section \ref{sect: uniqueness}).
%%Knowing that also the weak solutions provided by Theorem \ref{thm: teorema-esistenza} are strong allows us (a posteriori) to state properties of such solutions \emph{for all} $t>0$ rather than only \emph{for a.e.}\ $t>0$, which we do in the corresponding statement.
%\end{oss}
% which ensures that $ u$ coincides (up to time shifts) with the weak solution starting from $ u(\tau)$ constructed in Lemma \ref{lem: prima-esistenza}.

\smallskip

An important consequence of the fact that the solutions constructed in Lemma \ref{lem: prima-esistenza} are strong is the {\it decrease of their $ L^p_{\rho}$ norms} for any $ p \in [1,\infty] $. Indeed, by definition of strong solution, for any $p\in (1,\infty),$ we are allowed
 to multiply the differential equation in \eqref{eq: barenblatt-u_0} by $u^{p-1} $ and integrate in $ \mathbb{R}^d \times [t_1,t_2] $. By  Stroock-Varopoulos inequality \eqref{eq: strook-vareq} (let $v=u^m$ and $ q=(p+m-1)/m $), we get
\begin{equation}\label{eq: est-strong}
%\begin{aligned}
\int_{\mathbb{R}^d} u^p(x,t_2)  \, \rho(x) \mathrm{d}x - \int_{\mathbb{R}^d} u^p(x,t_1)  \, \rho(x) \mathrm{d}x  = -p \int_{t_1}^{t_2} \! \int_{\mathbb{R}^d} u^{p-1}(x,t) (-\Delta)^s(u^m)(x,t) \, \mathrm{d}x \mathrm{d}t \le 0
%\end{aligned}
\end{equation}
for all $ t_2>t_1>0 $. The validity of \eqref{eq: est-strong} down to $t_1=0$ %cannot be proved by letting $ t_1 \to 0 $, since a priori we have no information over the continuity of $ \| u(t) \|_{p,\rho} $ at $t=0 $. Nevertheless, this turns out to be true for
follows by using the approximate solutions $ \{ u_\eta \} $ from the proof of Lemma \ref{lem: prima-esistenza} %, as a consequence of the property $ u_\eta \in C([0,\infty);L^1_{\rho_\eta}(\mathbb{R}^d)) \cap L^\infty(\mathbb{R}^d \times (0,\infty)) $ (see \cite[Definition 2.2 and Theorem 3.1]{PT2}). Hence
and letting $ \eta \to 0$.
%, we finally get that
%\begin{equation}\label{eq: est-strong-zero}
 %\int_{\mathbb{R}^d} u^p(x,t_2) \, \rho(x) \mathrm{d}x \le \int_{\mathbb{R}^d} u^p(x,t_1) \, \rho(x) \mathrm{d}x
%\end{equation}
%for all $ t_2 > t_1 \ge 0$.
The case $ p=\infty $ can be handled by approximation.

\section{Uniqueness of weak solutions}\label{sect: uniqueness}
%Again, we shall prove our results only for
%weak solutions to \eqref{eq: barenblatt}. The modifications
%required to deal with \eqref{eq: barenblatt-regular}, provided $
%\rho $ is any weight complying with \eqref{eq: ass-rho}, are
%straightforward.
Prior to the proof of Theorem \ref{thm: teorema-uniqueness}, we
need some technical lemmas. Hereafter, by ``weak solution'' to
\eqref{eq: barenblatt-regular}, we shall mean a solution in the
sense of Definition \ref{eq: weak-sol-1}.
\begin{lem}\label{lem: monotonicity}
Let $d>2s$ and assume that $\rho$ satisfies \eqref{eq: ass-rho}
for some $\gamma \in [0,2s) \cap [0,d-2s]$ and $ \gamma_0 \in
[0,\gamma] $. Let $ u $ be a weak solution to \eqref{eq:
barenblatt-regular}. Then the potential $U(\cdot,t)$ of $
\rho(\cdot) u(\cdot,t) $ admits an absolutely continuous version
(in $ L^1_{\rm loc}(\mathbb{R}^d) $) which is nonincreasing in $t$.
\end{lem}
\begin{proof}
One proceeds as in the first part of the proof of Theorem \ref{thm: teorema-esistenza},
using the same techniques we exploited to prove \eqref{eq: eq-formale-potenziale-epsilon} rigorously.
\end{proof}
\begin{lem}\label{lem: potenziale}
Let $d>2s$ and assume that $\rho$ satisfies \eqref{eq: ass-rho}
for some $\gamma \in [0,2s) \cap [0,d-2s]$ and $ \gamma_0 \in
[0,\gamma] $. Let $ u $ be a weak solution to \eqref{eq:
barenblatt-regular}, taking the initial datum $\mu$ whose
potential is $ U^\mu $. Then there holds
\begin{equation}\label{eq: dato-iniziale-ovunque}
 \lim_{t \downarrow 0} U(x,t)  = U^\mu(x)  \ \ \ \forall x \in \mathbb{R}^d \, .
\end{equation}
\end{lem}
\begin{proof}
It is a direct application of Theorem 3.9 of \cite{Landkof} but, for the reader's convenience, we give some details.

Thanks to Theorem 3.8 of \cite{Landkof} and to the monotonicity
ensured by Lemma \ref{lem: monotonicity}, we have that the limit in
\eqref{eq: dato-iniziale-ovunque} is taken at least for a.e.\ $ x
\in \mathbb{R}^d$. However, for what follows it will be crucial to
prove that it is taken \emph{for every} $ x \in
\mathbb{R}^d$. To this end we make use again of the monotonicity
property provided by Lemma \ref{lem: monotonicity}. In fact, Lemma
1.12 of \cite{Landkof} shows that, as a consequence of the
monotonicity of potentials, there exists a positive finite Radon measure
$\nu$, whose potential is denoted by $U^{\nu}$, and a constant $A
\ge 0$ such that
% LANDKOF
% Un commento anche qui: il citato Lemma 1.12 del Landkof (a proposito, osservare che nella prima formula di tale lemma c'è un errore: evidentemente deve essere \equiv+\infty e non \not\equiv+\infty), assicura che \nu (\nu per noi, \mu per la notazione del Landkof) è una misura positiva, ma a priori non mi dice nulla sul fatto che tale misura sia finita, come scrivevamo erroneamente prima. Tuttavia è sicuramente una misura che ammette potenziale finito quasi ovunque (notare che il potenziale della misura di Lebesgue, ad esempio, è identicamente +\infty...). Questa è una cosa vera in tutto il libro: si lavora sempre implicitamente con misure (positive o con segno) localmente finite e che ammettono potenziale finito almeno quasi ovunque; ciò è spiegato a pagina 62 del Landkof, la condizione che conta è la 1.3.10.
\begin{equation*}%\label{eq: dato-iniziale-ovunque-prova-1}
 \lim_{t \downarrow 0} U(x,t)  = U^\nu(x) + A \ \ \ \forall x \in \mathbb{R}^d  \, .
\end{equation*}
Since \eqref{eq: dato-iniziale-ovunque} holds almost everywhere,
\begin{equation}\label{eq: dato-iniziale-ovunque-prova-2}
U^\mu(x)  = U^\nu(x) + A \ \ \ \textrm{for a.e.\ }  x \in
\mathbb{R}^d \, .
\end{equation}
But using the corollary at page 129 of \cite{Landkof}, from
\eqref{eq: dato-iniziale-ovunque-prova-2} we deduce that
necessarily $A=0$. Hence, \eqref{eq:
dato-iniziale-ovunque-prova-2} implies that $U^\nu=U^\mu$ almost
everywhere, and from Theorem 1.12 of \cite{Landkof} we know that
two potentials coinciding a.e.\ in fact coincide everywhere,
whence \eqref{eq: dato-iniziale-ovunque} follows.
% LANDKOF
% Un'osservazione sul Teorema 1.12 del Landkof: vale per misure con segno "potenziabili", cioè per cui sia vera la 1.3.10. La nostra misura con segno è evidentemente \nu-\mu, che è sicuramente ben definita (\nu è una misura positiva e \mu è una misura positiva finita, vedere a pag. 2 del Landkof cosa s'intende per misure con segno, ma è tutto abbastanza classico) e "potenziabile", perché lo sono \nu e \mu singolarmente.
\end{proof}

\subsection{Main ideas in the proof of uniqueness} Since the proof of Theorem \ref{thm:
teorema-uniqueness} is rather delicate, we point out its main ingredients. We should note that from a
general viewpoint it is based on a ``duality method'', and in
particular it is modeled on the uniqueness proof given by M. Pierre in
\cite{Pierre}. We comment again that our uniqueness
result seems to be new even if $s=1$, in the weighted case, or if $\rho\equiv 1$ when $s\in(0,1)$.

\smallskip

Let $u_1$ and $u_2$ be two weak solutions to \eqref{eq:
barenblatt-regular} such that they both take a common positive,
finite Radon measure $\mu$ as initial datum. We denote as $U_1(\cdot,t)$
and $ U_2(\cdot,t)$ the potentials of $ \rho(\cdot) u_1(\cdot,t) $
and $ \rho(\cdot) u_2(\cdot,t) $, respectively. Fix once for all
the parameters $h,T>0$ and consider the function
\begin{equation}\label{eq: definition-g}
g(x,t) := {U}_2(x,t+h) - {U}_1(x,t) \ \ \ \forall (x,t) \in
\mathbb{R}^d \times (0,T] \, .
\end{equation}
Proceeding again as in the proof of Theorem \ref{thm: teorema-esistenza} (under the hypothesis $ \gamma \le d-2s $, see
the proof of \eqref{eq: eq-formale-potenziale-epsilon}), we get that $ g(\cdot,t) $ is an
absolutely continuous curve (for instance in
$L^1_{\rm loc}(\mathbb{R}^d)$) satisfying
\begin{equation}\label{eq: equation-g}
\rho(x) g_t(x,t) =\rho(x) \left( u_1^m(x,t) -
u_2^m(x,t+h) \right)=-a(x,t)(-\Delta)^s(g)(x,t) \ \
\end{equation}
for a.e.\ $(x,t) \in  \mathbb{R}^d \times (0,T)$, where
we define the function $a$ as
\begin{equation}\label{eq: defa}
a(x,t):=
\begin{cases}
\frac{{u}_1^m(x,t)-{u}_2^m(x,t+h)}{{u}_1(x,t)-{u}_2(x,t+h)}\ \ \ & \textrm{if}\ {u}_1(x,t)\not={u}_2(x,t+h) \, , \\
0 & \textrm{if}\ {u}_1(x,t)={u}_2(x,t+h) \, ,
\end{cases}
\end{equation}
and we used the fact that, thanks to the properties of Riesz potentials,
$$ (-\Delta)^s(g)(x,t)= \rho(x) u_2(x,t+h)-\rho(x) u_1(x,t) \, . $$
Note that, since $m>1$ and $ u_1,u_2 \in L^\infty(\mathbb{R}^d \times (\tau,\infty)) $ for all $ \tau>0 $, $ a $ is a nonnegative
function belonging to $ L^\infty(\mathbb{R}^d \times
(\tau,\infty)) $ for all $ \tau>0 $.

Hence $g$ is a solution to the {\it linear fractional} equation
\eqref{eq: equation-g}. Moreover, by Lemmas \ref{lem:
monotonicity} and \ref{lem: potenziale}, $g(x, 0)\leq 0$ for a.e.\ $x  \in \mathbb R^d$. If we could apply the maximum principle, then we would
get $g\leq 0$ in $\mathbb R^d\times (0,\infty)$. This would
imply $u_1\leq u_2$ and, by interchanging the roles of $u_1$ and $u_2$, $u_1 = u_2$. However, a priori a maximum principle is not available for solutions to \eqref{eq: equation-g}. We then consider the ``dual'' problem
\begin{equation*}\label{e500}
\begin{cases}
\rho(x) \varphi_t = (-\Delta)^s (a\varphi)  & \textrm{in}\,\; \mathbb
R^d\times (0,T) \, , \\
\varphi(x,T)=\psi(x) & \textrm{on } \mathbb{R}^d \times
\{T\} \, ,
\end{cases}
\end{equation*} for any
$\psi\in \mathcal{D}_+(\mathbb{R}^d)$. Suppose for a moment that
it admits a unique smooth solution $\varphi$. Multiplying \eqref{eq: equation-g} by $\varphi$ and integrating by parts
we formally obtain
\begin{equation}\label{e501} \int_{\mathbb R^d}
g(x, T) \rho(x)\psi(x)\,\mathrm{d}x =\,\int_{\mathbb
R^d}\varphi(x,0) g(x,0)\,\mathrm{d}x.
\end{equation}
The conclusion would again follow should a maximum principle for \eqref{e501} hold, and in order to justify rigorously its applicability a further approximation is necessary. In fact, for every $n\in \mathbb N$ and
$\varepsilon>0$, we consider a family $\{\psi_{n,\varepsilon}\}$ which solves, in a sense that will be clarified later, the problem
\begin{equation}\label{eq: equation-psi}
\begin{cases}
\rho(x) \left(\psi_{n,\varepsilon}\right)_t = (-\Delta)^s\left[\left(a_n+\varepsilon\right)\psi_{n,\varepsilon}\right] & \textrm{in } \mathbb{R}^d \times (0,T) \, , \\
\psi_{n,\varepsilon} =  \psi  & \textrm{on } \mathbb{R}^d \times
\{T\} \, ,
\end{cases}
\end{equation}
where $ \psi\in\mathcal{D}_+({\mathbb R}^d) $. The sequence
$\{a_n\} $ is a suitable approximation of the function $a $
defined in \eqref{eq: defa}. In particular we suppose that, for
every $n\in \mathbb N$, $a_n(x,t)$ is a piecewise constant
function of $t$ (regular in $x$) on the time intervals
$(T-{(k+1)T}/{n}, T-{kT}/{n}] $, for any $ k \in \{0, \ldots, n-1
\}$. Thanks to Theorem \ref{thm: self-adj} and to Proposition
\ref{pro: laplaciano-Lp} below, we are then able to treat problem
\eqref{eq: equation-psi} by means of standard semigroup theory.
Here the Markov property for the linear semigroup associated to
the operator $A = \rho^{-1}(-\Delta)^s$ will have a crucial role.
Let us mention that in \cite[Theorem 1]{Pierre}, where $s=1$, $
\rho \equiv 1$, in view of standard parabolic theory it was not
necessary to approximate the function $a$ by a piecewise constant
function of $t$. Using the family $ \{ \psi_{n,\varepsilon} \} $ and
passing to the limit as $n\to \infty$ and then as $\varepsilon\to
0$ we get the next crucial identity:
\begin{equation}\label{eq: equation-psi-g-derivata-parti-tempo-k-limite}
\int_{\mathbb{R}^d} g(x,T) \psi(x) \, \rho(x) \mathrm{d}x =
\int_{\mathbb{R}^d} g(x,t) \, \mathrm{d}\nu(t) \ \ \ \textrm{for
a.e.\ } t \in (0,T) \,,
\end{equation}
where $\{\nu(t)\}$ is a specific family of positive finite Radon
measures. More precisely, $\nu(t)$ is the limit in
$\sigma(\mathcal M(\mathbb R^d), C_b(\mathbb R^d))$ as
$\varepsilon\to 0$ of $\{ \rho(\cdot)
\psi_\varepsilon(\cdot,t)\}$, where $\psi_\varepsilon $ is in turn
the weak limit in $L^2_{\rho}(\mathbb R^d\times (\tau, T))$ (for
all $\tau\in (0,T)$) as $n\to\infty$ of $\{\psi_{n,\varepsilon} \}.$ Note that, roughly speaking,
\eqref{eq: equation-psi-g-derivata-parti-tempo-k-limite}
corresponds to identity \eqref{e501} in the previous formal
argument. Finally, we prove rigorously that the r.h.s.~of
\eqref{eq: equation-psi-g-derivata-parti-tempo-k-limite} has a
nonpositive limit as $t\to 0$, whence the conclusion follows.

\subsection{Construction and properties of the family $\{\psi_{n,\varepsilon}\}$}%\label{psine}
We begin our proof by introducing the functions $ \psi_{n,\varepsilon} $, which formally solve \eqref{eq: equation-psi}.
\begin{lem}\label{lem: approx}
Let $d>2s$ and assume that $\rho$ satisfies \eqref{eq: ass-rho} for some $\gamma \in (0, 2s) $. Let $\{ a_{n} \}$ be a
sequence of functions converging a.e.\ to the function $a$ as in
\eqref{eq: defa} such that:
\begin{itemize}
%\item[$\bullet$] $a_n\in L^\infty(\mathbb R^d\times (0,T))$;
\item[$\bullet$] for any $ n \in \mathbb{N} $ and $t>0$, $a_n(x,t)
$ is a regular function of $x$; \item[$\bullet$] for any $ n \in
\mathbb{N} $ and $x \in \mathbb{R}^d$, $a_n(x,t) $ is a piecewise
constant function of $t$ on the time intervals $(T-{(k+1)T}/{n},
T-{kT}/{n} ] $, for any $ k \in \{0, \ldots, n-1 \} $;
\item[$\bullet$] $\{ \| a_n \|_{L^\infty(\mathbb{R}^d
\times(\tau,\infty))} \} $ is uniformly bounded in $n$ for any $
\tau>0 $.
\end{itemize}
Then, for any $\varepsilon>0$ and any $ \psi \in
\mathcal{D}_{+}(\mathbb{R}^d) $, there exists a nonnegative solution
$\psi_{n,\varepsilon}$ to problem \eqref{eq: equation-psi}, in the
sense that $ \psi_{n,\varepsilon}(\cdot, t) $ is a continuous curve in
$L^p_{\rho}(\mathbb{R}^d)$ (for all $p \in (1,\infty)$)
satisfying $ \psi_{n,\varepsilon}(\cdot,0)=\psi(\cdot,0) $ and it is
absolutely continuous on $(T-{(k+1)T}/{n}, T-{kT}/{n})$ for all
$k\in \{0, \ldots, n-1 \} $, so that the identity
\begin{gather}\label{eq: equation-psi-integrale}
\psi_{n,\varepsilon}(\cdot,t_2)-\psi_{n,\varepsilon}(\cdot,t_1) = \int_{t_1}^{t_2} \rho^{-1}(\cdot) (-\Delta)^s\left[\left(a_n+\varepsilon\right)\psi_{n,\varepsilon}\right](\cdot,\tau) \, \mathrm{d}\tau \\
 \forall t_1, t_2 \in \left( T-\frac{(k+1) T}{n}, T-\frac{kT}{n} \right) , \ \forall k \in \{0, \ldots, n-1 \} \nonumber
\end{gather}
holds in $L^p_{\rho}(\mathbb{R}^d)$ for all $p \in (1,\infty)$. Moreover,
\begin{equation}\label{eq: markov-nonexp-L1}
\psi_{n,\varepsilon} \in L^\infty((0,T);L^p_{\rho}({\mathbb
R}^d)) \ \ \forall p\in[1,\infty] \quad \textrm{and} \quad
\left\| \psi_{n,\varepsilon}(t) \right\|_{1,\rho} \le \left\|
\psi \right\|_{1,\rho} \ \ \forall t \in [0,T] \, .
\end{equation}
\end{lem}
\begin{proof}
To construct $\psi_{n,\varepsilon}$ as in the statement, we first
define $\zeta_1$ as the solution of
\begin{equation}\label{eq: equation-psi-approx-1}
\begin{cases}
\rho(x) \left(\zeta_1\right)_t = (-\Delta)^s\left[\left(a_n(T)+\varepsilon\right)\zeta_1 \right] & \textrm{in } \mathbb{R}^d \times \left(T-\frac{T}{n},T\right) , \\
\zeta_1 = \psi & \textrm{on } \mathbb{R}^d \times \{ T \} \, .
\end{cases}
\end{equation}
To construct such a solution, one can for instance exploit the change of variable
\begin{equation}\label{eq: cambio-phi}
\phi_1(x,t):=\left(a_n\left(x,T\right)+\varepsilon\right)
\zeta_1(x,t)\, ,
\end{equation}
where $\phi_1 $ is the solution of
\begin{equation}\label{eq: equation-psi-approx-1-phi}
\begin{cases}
 \left( \phi_1 \right)_t = \left(a_n(T)+\varepsilon\right) \rho^{-1} (-\Delta)^s(\phi_1) & \textrm{in } \mathbb{R}^d \times \left(T-\frac{T}{n},T\right) , \\
\phi_1= \left(a_n(T) + \varepsilon \right) \psi & \textrm{on }
\mathbb{R}^d \times \{ T \} \, .
\end{cases}
\end{equation}
Problem \eqref{eq: equation-psi-approx-1-phi} is indeed solvable
by standard semigroup theory. In fact, consider the operator $A_1:=\rho_1^{-1} (-\Delta)^s$, where we have set $\rho_1(x) := \left(a_n\left(x,T\right)+\varepsilon\right)^{-1}
\rho(x)$, with domain $X_{s,\rho_1}=X_{s,\rho}$ (see Definition \ref{den: spazio-Xs}). $A_1$ is positive, self-adjoint
and generates a Markov semigroup on $L^2_{\rho_1}(\mathbb{R}^d)$. These properties follow from Theorem \ref{thm: self-adj}. Our initial datum $ \phi_1 $
belongs to
$L^p_{\rho_1}(\mathbb{R}^d) $ for all $p \in [1,\infty]$, and this
is enough in order to have a solution to \eqref{eq:
equation-psi-approx-1-phi} which is continuous up to $ t=T $ and
absolutely continuous in $\left(T-\frac{T}{n},T\right)$ in
$L^p_{\rho_1}(\mathbb{R}^d) $ for all $p \in (1,\infty)$. In fact, the semigroup
associated with $A_1$ enjoys the Markov property and therefore, as a consequence of
\cite[Theorems 1.4.1, 1.4.2]{D}, can be extended to a contraction
semigroup on $L^p_{\rho_1}(\mathbb{R}^d)$ (consistent with the
original semigroup on $L^2_{\rho_1}(\mathbb{R}^d)\cap
L^p_{\rho_1}(\mathbb{R}^d)$) for all $p \in [1,\infty]$, which is
analytic with a suitable angle $\theta_p>0$ if $p\in(1,\infty)$.
By classical results (see e.g.~\cite[Theorem 5.2 at p.\ 61]{Pazy})
the latter property ensures in particular that problem \eqref{eq:
equation-psi-approx-1-phi} is solved by a \emph{differentiable}
curve $\phi_1(\cdot,t) $ in $L^p_{\rho_1}(\mathbb{R}^d)$ for all
$p\in(1,\infty)$. Going back to the original variable $ \zeta_1 $ through
\eqref{eq: cambio-phi}, we deduce that it solves \eqref{eq:
equation-psi-approx-1} in the same sense in which $\phi_1$ solves
\eqref{eq: equation-psi-approx-1-phi}. Having at our disposal such
a $ \zeta_1 $, we can then solve the problem
\begin{equation*}%\label{eq: equation-psi-approx-2}
\begin{cases}
\rho(x) \left(\zeta_2 \right)_t = (-\Delta)^s\left[\left(a_n\left(T-\frac{T}{n}\right)+\varepsilon\right)\zeta_2\right] &  \textrm{in } \mathbb{R}^d \times \left(T-\frac{2T}{n},T-\frac{T}{n}\right) , \\
\zeta_2 =
\left(a_n\left(x,T\right)+\varepsilon\right)^{-1} \phi_1 &
\textrm{on } \mathbb{R}^d \times \left\{ T-\frac{T}{n}
\right\}  ,
\end{cases}
\end{equation*}
just by proceeding as above. That is, we perform the change of
variable
\begin{equation*}%\label{eq: cambio-phi-2}
\phi_2(x,t):=\left(a_n\left(x,T-\frac{T}{n}\right)+\varepsilon\right)
\zeta_2(x,t)
\end{equation*}
and take $ \phi_2 $ as the solution of
\begin{equation*}%\label{eq: equation-psi-approx-1-phi-2}
\begin{cases}
 \left( \phi_2 \right)_t = \left(a_n(T-\frac{T}{n})+\varepsilon\right) \rho^{-1} (-\Delta)^s(\phi_2) & \textrm{in } \mathbb{R}^d \times \left(T-\frac{2T}{n},T-\frac{T}{n}\right) , \\
\phi_2= \left(a_n(T-\frac{T}{n}) + \varepsilon \right) \zeta_1 =
\frac{\left(a_n(T-\frac{T}{n}) + \varepsilon \right)}{\left(a_n(T)
+ \varepsilon \right)} \phi_1 & \textrm{on } \mathbb{R}^d
\times \left\{ T-\frac{T}{n} \right\} .
\end{cases}
\end{equation*}
It is clear how the procedure goes on and allows us to obtain a
solution $ \psi_{n,\varepsilon} $ to \eqref{eq: equation-psi} in
the sense of the statement, just by defining it as
\begin{equation*}%\label{eq: def-psi-pezzi}
\psi_{n,\varepsilon}(\cdot,t) := \zeta_{k+1}(\cdot,t) \ \ \ \forall t \in
\left( T-\frac{(k+1) T}{n}, T-\frac{kT}{n} \right] , \ \ \forall k
\in \{0, \ldots, n-1 \} \, .
\end{equation*}
Finally, since
\begin{equation*}%\label{eq: operatore-autoagg-k}
 \rho_{k+1}^{-1} (-\Delta)^s
\end{equation*}
generates a contraction semigroup on
$L^p_{\rho_{k+1}}(\mathbb{R}^d) $ for all $p\in[1,\infty]$, where
\begin{equation} \label{eq: weight-rho-markov-k}
\rho_{k+1}(x):=
\left(a_n\left(x,T-\frac{kT}{n}\right)+\varepsilon\right)^{-1}
\rho(x) \, ,
\end{equation}
the inequalities
\begin{gather}\label{eq: markov-contract}
\left\| \phi_{k+1}(t) \right\|_{p,\rho_{k+1}} \le \left\| \frac{\left(a_n(T-\frac{kT}{n}) + \varepsilon \right)}{\left(a_n(T-\frac{(k-1)T}{n}) + \varepsilon \right)} \phi_{k}\left(T-\frac{kT}{n}\right) \right\|_{p,\rho_{k+1}}  \\
\forall t \in \left( T-\frac{(k+1) T}{n}, T-\frac{kT}{n} \right] ,
\ \forall p \in [1,\infty]  \nonumber
\end{gather}
hold for any $ k\in \{0, \ldots, n-1 \} $ (on the r.h.s.\ of
\eqref{eq: markov-contract} for $ k=0 $ we conventionally set $
\phi_0=\psi $ and $ a_n(T+{T}/{n}) + \varepsilon =1 $). Going back
to the variables $ \zeta_{k+1} $ and recalling \eqref{eq:
weight-rho-markov-k}, from \eqref{eq: markov-contract} one deduces
\eqref{eq: markov-nonexp-L1}: in fact, for $p=1 $ it is easy to
see that the terms containing $a_n$ cancel out and give the corresponding
inequality, while for $p>1$ such
terms remain and one obtains an inequality of the type of $\| \psi_{n,\varepsilon}(t) \|_{p,\rho} \le C(n,\varepsilon) \|
\psi \|_{p,\rho} $, where $C(n,\varepsilon)$ is a
positive constant depending on $n,\varepsilon$.
\end{proof}

\begin{lem}\label{lem: lemma-prima-eq}
Let $d>2s$ and assume that $\rho$ satisfies \eqref{eq: ass-rho} for some $\gamma \in (0,2s) \cap (0,d-2s] $. Let $g$ be as
in \eqref{eq: definition-g}, $a$ as in \eqref{eq: defa} and $a_n$,
$\psi_{n,\varepsilon}$, $ \psi$ as in Lemma \ref{lem: approx}.
Then the identity
\begin{equation}\label{eq: equation-psi-g-derivata-parti-tempo}
\begin{aligned}
& \int_{\mathbb{R}^d} g(x,T) \psi(x) \, \rho(x)\mathrm{d}x - \int_{\mathbb{R}^d} g(x,t) \psi_{n,\varepsilon}(x,t) \, \rho(x)\mathrm{d}x  \\
 = & \int_t^T \int_{\mathbb{R}^d} \left(a_n(x,\tau)+\varepsilon -a(x,\tau)\right)(-\Delta)^s(g)(x,\tau) \, \psi_{n,\varepsilon}(x,\tau) \, \mathrm{d}x \mathrm{d}\tau
\end{aligned}
\end{equation}
holds for all $t \in (0,T]$.
\end{lem}
\begin{proof}
To begin with, let us set
$t_k :=  T(n-k)/n$ for all $k \in \{0, \ldots, n \}$.
Recall that, from Lemma \ref{lem: approx}, $ \psi_{n,\varepsilon}(\cdot, t) $ is a continuous curve in $
L^p_{\rho}(\mathbb{R}^d) $ on $(0,T]$, absolutely continuous on
any interval $ (t_{k+1},t_{k}) $ for $ k \in \{0, \ldots, n-1 \} $
and satisfying the differential equation in \eqref{eq:
equation-psi} on such intervals, for all $p \in (1,\infty)$.
Moreover, $ g(\cdot, t) $ is an absolutely continuous curve in $
L^p_{\rho}(\mathbb{R}^d) $ on $(0,T]$ for all $p$ such that
\begin{equation}\label{eq: interv-p-U}
p \in \left(\frac{d-\gamma}{d-2s} ,\infty \right) .
\end{equation}
Since $g(x,t)$ is a continuous function of $x$ (recall Lemma \ref{lem: potential-basic-prop-1}) and the weight $ \rho(x) $ is locally integrable, in order to prove that $g(\cdot,t)\in L^p_{\rho}(\mathbb{R}^d)$ for all $p$ as in \eqref{eq: interv-p-U} it suffices to show that $g(\cdot,t)\in L^p_{\rho}(B_1^c)$. To this end, still Lemma \ref{lem: potential-basic-prop-1} ensures that $ g(\cdot,t) \in L^p(\mathbb{R}^d) $ for all $p$ satisfying \eqref{eq: potential-v-Lp}: the latter property and H\"{o}lder's inequality imply that $g(\cdot,t)\in L^p_{\rho}(B_1^c)$ for all $p$ as in \eqref{eq: interv-p-U}.

The fact that $ g(\cdot,t) $ is also absolutely continuous in $
L^p_{\rho}(\mathbb{R}^d) $ on the time interval
$(0,T]$ is a consequence of \eqref{eq: equation-g} and of the
integrability properties of $u_1,u_2$. Hence, due to Lemma
\ref{lem: approx}, we get that
\begin{equation}\label{eq: scal-prod-psi-gi}
t \mapsto \int_{\mathbb{R}^d} g(x,t) \psi_{n,\varepsilon}(x,t) \,
\rho(x) \mathrm{d}x
\end{equation}
is a continuous function on $(0,T]$, absolutely continuous on each interval $ (t_{k+1},t_k) $ and satisfies
\begin{equation}\label{eq: equation-psi-g-derivata}
\begin{aligned}
& \frac{\mathrm{d}}{\mathrm{d}t} \int_{\mathbb{R}^d} g(x,t) \psi_{n,\varepsilon}(x,t) \, \rho(x) \mathrm{d}x \\
 = & \int_{\mathbb{R}^d} \left\{ -a(x,t)(-\Delta)^s(g)(x,t) \, \psi_{n,\varepsilon}(x,t) + g(x,t) \, (-\Delta)^s\left[\left(a_n+\varepsilon\right)\psi_{n,\varepsilon}\right](x,t) \right\} \mathrm{d}x
\end{aligned}
\end{equation}
there. As we have just seen, $ g(\cdot,t) \in
L^p_{\rho}(\mathbb{R}^d) $ for all $p$ satisfying
\eqref{eq: interv-p-U} and $ \rho^{-1}(\cdot)(-\Delta)^s(g)(\cdot,t) \in
L^p_{\rho}(\mathbb{R}^d) $ for all $p \in [1,\infty] $.
Moreover, as a consequence of Lemma \ref{lem: approx}, we have
that $ (a_n(\cdot,t)+\varepsilon)\psi_{n,\varepsilon}(\cdot,t) \in
L^p_{\rho}(\mathbb{R}^d) $ for all $p \in [1,\infty] $ and
$\rho^{-1}(\cdot)(-\Delta)^s[(a_n(\cdot,t)+\varepsilon)\psi_{n,\varepsilon}(\cdot,t)]
\in L^p_{\rho}(\mathbb{R}^d) $ for all $p \in (1,\infty)$. We
are therefore in position to apply Proposition \ref{pro:
laplaciano-Lp} to the r.h.s.\ of \eqref{eq:
equation-psi-g-derivata} (note that the interval $((d-\gamma)/(d-2s),\infty)
\cap [2,2(d-\gamma)/(d-2s))$ is not empty) to get that
\begin{equation}\label{eq: equation-psi-g-derivata-parti}
\frac{\mathrm{d}}{\mathrm{d}t} \int_{\mathbb{R}^d} g(x,t)
\psi_{n,\varepsilon}(x,t) \, \rho(x) \mathrm{d}x =
\int_{\mathbb{R}^d} \left(a_n(x,t)+\varepsilon
-a(x,t)\right)(-\Delta)^s(g)(x,t) \, \psi_{n,\varepsilon}(x,t) \, \mathrm{d}x \, .
\end{equation}
But the r.h.s.\ of \eqref{eq: equation-psi-g-derivata-parti} is in $
L^1((\tau,T)) $ for any $ \tau\in(0,T) $, from which \eqref{eq:
scal-prod-psi-gi} is absolutely continuous on the whole of $ (0,T]
$ and not only on $(t_{k+1},t_{k})$. Integrating \eqref{eq: equation-psi-g-derivata-parti} between $ t
$ and $T$ then yields \eqref{eq: equation-psi-g-derivata-parti-tempo}.
\end{proof}
Now we prove a key ``conservation of mass'' property for $ \psi_{n,\varepsilon} $.
\begin{lem}\label{lem: lem-uniq-pre}
Let $d>2s$ and assume that $\rho$ satisfies \eqref{eq: ass-rho} for some $ \gamma \in (0,2s) $. Let $\psi_{n,\varepsilon} $
and $ \psi$ be as in Lemma \ref{lem: approx}. Then the
$L^1_{\rho}$ norm of $\psi_{n,\varepsilon}(\cdot, t)$ is preserved,
that is
\begin{equation}\label{eq: 15pierre-4}
\int_{\mathbb{R}^d} \psi_{n,\varepsilon}(x,t) \, \rho(x)
\mathrm{d}x = \int_{\mathbb{R}^d} \psi(x) \, \rho(x)
\mathrm{d}x  \ \ \ \forall t \in (0,T] \, .
\end{equation}
\end{lem}
\begin{proof}
Multiplying \eqref{eq: equation-psi-integrale} by any $\varphi \in
\mathcal{D}(\mathbb{R}^d) $ and integrating in $ \mathbb{R}^d $,
we obtain:
\begin{equation}\label{eq: 15pierre-1}
\begin{aligned}
& \int_{\mathbb{R}^d} \psi_{n,\varepsilon}(x,t^\ast) \varphi(x) \, \rho(x) \mathrm{d}x - \int_{\mathbb{R}^d} \psi_{n,\varepsilon}(x,t_\ast)  \varphi(x) \, \rho(x) \mathrm{d}x \\
= & \int_{\mathbb{R}^d} (-\Delta)^s(\varphi)(x) \left(
\int_{t_\ast}^{t^\ast}
\left(a_n(x,\tau)+\varepsilon\right)\psi_{n,\varepsilon}(x,\tau)
\, \mathrm{d}\tau \right) \mathrm{d}x
\end{aligned}
\end{equation}
for all $t_\ast,t^\ast \in (t_{k+1},t_k) $. Since the $L^1_{\rho}$
norm of $\psi_{n,\varepsilon}(\cdot,t)$ is bounded by the
$L^1_{\rho}$ norm of the final datum $\psi$ (recall \eqref{eq:
markov-nonexp-L1}), from \eqref{eq: 15pierre-1} we get:
\begin{equation}\label{eq: 15pierre-2}
\begin{aligned}
\left| \int_{\mathbb{R}^d} \psi_{n,\varepsilon}(x,t^\ast)
\varphi(x) \, \rho(x) \mathrm{d}x - \int_{\mathbb{R}^d}
\psi_{n,\varepsilon}(x,t_\ast) \varphi(x) \, \rho(x)
\mathrm{d}x \right| \le \underline{C} \, |t^\ast-t_\ast| \left\| \psi
\right\|_{1,\rho} \left\| \rho^{-1} (-\Delta)^s(\varphi)
\right\|_\infty ,
\end{aligned}
\end{equation}
where $ \underline{C} :=\| a_n+\varepsilon \|_{L^\infty(\mathbb{R}^d \times (t_\ast \wedge t^\ast,T))} $ is a positive constant independent of $n$ and $\varepsilon$. Replacing $ \varphi $ with the cut-off function $ \xi_R $ (defined in Lemma \ref{lem:decay-lap-cutoff}) and estimating the r.h.s.\ of \eqref{eq: 15pierre-2} as in the proof of Proposition \ref{oss: cons-mass} yields
\begin{equation}\label{eq: 15pierre-3}
\begin{aligned}
& \left| \int_{\mathbb{R}^d} \psi_{n,\varepsilon}(x,t^\ast) \xi_R(x)
\rho(x) \mathrm{d}x - \int_{\mathbb{R}^d}
\psi_{n,\varepsilon}(x,t_\ast) \xi_R(x) \rho(x) \mathrm{d}x
\right| \\
\le &  \underline{C} \, |t^\ast-t_\ast| \left\| \psi \right\|_{1,\rho}
c^{-1} \left( R^{-2s} + R^{-2s+\gamma} \right) \left\| (1+|x|^\gamma) (-\Delta)^{s}(\xi) \right\|_\infty
\end{aligned}
\end{equation}
for all $ R>0 $ and $ t_\ast, t^\ast \in (t_{k+1},t_k) $,  $c$ being as in \eqref{eq: ass-rho}. Recalling that
$ \psi_{n,\varepsilon}(\cdot, t) $ is a continuous curve (for instance in
$ L^2_{\rho}(\mathbb{R}^d) $) on $(0,T]$, we can extend the
validity of \eqref{eq: 15pierre-3} (and \eqref{eq: 15pierre-2}) to
any $ t_\ast,t^\ast  \in (0,T] $. By choosing $ t^\ast=T $ and letting $ R
\to \infty $ in \eqref{eq: 15pierre-3} we finally get \eqref{eq:
15pierre-4}.
\end{proof}
In the next lemma we introduce the Riesz potential of
$ \rho(\cdot) \psi_{n,\varepsilon}(\cdot,t)$, which will play a fundamental role below.
\begin{lem}
Let $d>2s$ and assume that $\rho$ satisfies \eqref{eq: ass-rho} for some $\gamma \in (0,2s) \cap (0,d-2s] $. Let $a_n$,
$\psi_{n,\varepsilon}$ and $ \psi$ be as in Lemma \ref{lem: approx}. We denote as $H_{n,\varepsilon}(\cdot, t)$ the Riesz potential
of $ \rho(\cdot) \psi_{n,\varepsilon}(\cdot, t)$, that is
\begin{equation*}%\label{eq: potential-epsilon}
H_{n,\varepsilon}(x,t) := [I_{2s} \ast \left(\rho(\cdot) \psi_{n,\varepsilon}(\cdot, t)\right)](x) \ \ \ \forall (x,t) \in \mathbb{R}^d \times (0,T] \, .
\end{equation*}
Then $H_{n,\varepsilon}(\cdot, t) \in \dot{H}^s(\mathbb{R}^d)$ and the identity
\begin{equation}\label{eq: per-parti-H-dt-parti}
%\begin{aligned}
\left\| I_{2s} \ast ( \rho \psi) \right\|_{\dot{H}^s}^2 = \left\|  H_{n,\varepsilon}(t) \right\|_{\dot{H}^s}^2 + 2 \int_t^T \! \int_{\mathbb{R}^d} \left(a_n(x,\tau)+\varepsilon\right)\psi^2_{n,\varepsilon}(x,\tau) \, \rho(x) \mathrm{d}x \mathrm{d}\tau
%\end{aligned}
\end{equation}
holds for all $t\in (0,T] $.
\end{lem}
\begin{proof}
First notice that $ \rho^{-1}(\cdot) (-\Delta)^s(H_{n,\varepsilon})(\cdot,t) =
\psi_{n,\varepsilon}(\cdot,t) \in L^p_{\rho}({\mathbb R}^d)$ for all
$p\in[1,\infty]$ (recall \eqref{eq: markov-nonexp-L1}) and
$H_{n,\varepsilon}(\cdot,t) \in L^p_{\rho}({\mathbb R}^d) $ for all
$p$ satisfying \eqref{eq: interv-p-U} (this can be proved by
exploiting Lemma \ref{lem: potential-basic-prop-1} exactly as in
the proof of Lemma \ref{lem: lemma-prima-eq}). Again, since the
interval $((d-\gamma)/(d-2s),\infty) \cap [2,2(d-\gamma)/(d-2s))$
is not empty, applying Proposition \ref{pro: laplaciano-Lp}
we get that $H_{n,\varepsilon}(\cdot,t) \in \dot{H}^s(\mathbb{R}^d)$ and
the identity
\begin{equation}\label{eq: per-parti-H}
%\begin{aligned}
\left\|  H_{n,\varepsilon}(t)\right\|_{\dot{H}^s}^2 =
\int_{\mathbb{R}^d} H_{n,\varepsilon}(x,t) \,
(-\Delta)^s\left(H_{n,\varepsilon}\right)(x,t) \, \mathrm{d}x =
\int_{\mathbb{R}^d} H_{n,\varepsilon}(x,t)
\psi_{n,\varepsilon}(x,t) \, \rho(x) \mathrm{d}x
%\end{aligned}
\end{equation}
holds. Thanks to the validity of the differential equation
\begin{equation}\label{eq: deriv-t-H}
\left( H_{n,\varepsilon} \right)_t(x,t) =
\left(a_n(x,t)+\varepsilon\right)\psi_{n,\varepsilon}(x,t) \ \ \
\textrm{for a.e.\ } (x,t) \in \mathbb{R}^d \times (0,T) \, ,
\end{equation}
which can be justified as we did for \eqref{eq: equation-g}, taking the time derivative of \eqref{eq: per-parti-H} in
the intervals $(t_{k+1},t_k) $, using \eqref{eq: deriv-t-H}, \eqref{eq:
equation-psi} and again Proposition \ref{pro: laplaciano-Lp}, we obtain:
%the fact that for $ t \in (t_{k+1},t_k) $ both $ H_{n,\varepsilon}(t) $ and $(a_n(t)+\varepsilon) \psi_{n,\varepsilon}(t) $ belong to $X_{s,\gamma}$ (i.e.\ we can ``integrate by parts''), we get:
\begin{equation}\label{eq: per-parti-H-dt}
\begin{aligned}
\frac{\mathrm{d}}{\mathrm{d}t} \left\|  H_{n,\varepsilon}(t)
\right\|_{\dot{H}^s}^2 = 2 \int_{\mathbb{R}^d}
\left(a_n(x,t)+\varepsilon\right)\psi^2_{n,\varepsilon}(x,t) \,
\rho(x) \mathrm{d}x \, .
\end{aligned}
\end{equation}
A priori, from \eqref{eq: per-parti-H}, we have that $ \|
H_{n,\varepsilon}(t) \|^2_{\dot{H}^s} $ is continuous on $ (0,T] $
and absolutely continuous only on $ (t_{k+1},t_k) $. However, the
r.h.s.\ of \eqref{eq: per-parti-H-dt} is in $L^1((\tau,T)) $ for any
$ \tau \in (0,T) $. Hence, \eqref{eq: per-parti-H-dt-parti} just
follows by integrating \eqref{eq: per-parti-H-dt} from $t$ to $T$.
\end{proof}

\subsection{Passing to the limit as $n\to\infty$}%\label{psie}
The goal of the next lemma is to show that, as $n\to\infty$, $ \{
\psi_{n,\varepsilon} \} $ suitably converges to a limit function $
\psi_\varepsilon $ that enjoys some crucial properties.
%(that is, as $ \{ a_n \} $ tends to $a$)
\begin{lem}
Let $d>2s$ and assume that $\rho$ satisfies \eqref{eq: ass-rho} for some $ \gamma \in (0,2s) \cap (0,d-2s] $. Let $u_1$ and
$u_2$ be two weak solutions to problem \eqref{eq: barenblatt-regular},
taking the common positive finite Radon measure $\mu$ as initial datum.
Let $g$ be as in \eqref{eq: definition-g}, $a$ as in \eqref{eq:
defa} and $\psi_{n,\varepsilon}, \psi$ as in Lemma \ref{lem:
approx}. Then, up to subsequences, $\{ \psi_{n,\varepsilon} \}$
converges weakly in $L^2_{\rho}(\mathbb{R}^d \times (\tau,T) )$
(for all $ \tau \in (0,T) $) to a suitable nonnegative function
$\psi_\varepsilon$ and $\{ \rho(\cdot) \psi_{n,\varepsilon}(\cdot,t)
\}$ converges to $\rho(\cdot) \psi_\varepsilon(\cdot,t)$ in
$\sigma(\mathcal{M}(\mathbb{R}^d),C_b(\mathbb{R}^d)) $ for a.e.\
$t \in (0,T)$. Moreover, $\psi_{\varepsilon}$ enjoys the
following properties:
\begin{equation}\label{eq: cons-mass-epsilon}
\int_{\mathbb{R}^d} \psi_{\varepsilon}(x,t) \, \rho(x)
\mathrm{d}x = \int_{\mathbb{R}^d} \psi(x) \, \rho(x)
\mathrm{d}x \, ,
\end{equation}
\begin{equation}\label{eq: 15pierre-1-epsilon}
\begin{aligned}
& \int_{\mathbb{R}^d} \psi(x) \varphi(x) \, \rho(x)  \mathrm{d}x - \int_{\mathbb{R}^d}  \psi_{\varepsilon}(x,t) \varphi(x) \, \rho(x) \mathrm{d}x \\
= &  \int_{\mathbb{R}^d} (-\Delta)^s(\varphi)(x) \left(
\int_{t}^{T}
\left(a(x,\tau)+\varepsilon\right)\psi_{\varepsilon}(x,\tau) \,
\mathrm{d}\tau \right) \mathrm{d}x \, ,
\end{aligned}
\end{equation}
\begin{equation}\label{eq: equation-psi-g-derivata-parti-tempo-epsilon-1}
\begin{aligned}
& \left| \int_{\mathbb{R}^d} g(x,T) \psi(x) \, \rho(x) \mathrm{d}x - \int_{\mathbb{R}^d} g(x,t) \psi_{\varepsilon}(x,t) \, \rho(x) \mathrm{d}x \right| \\
\le & \varepsilon \, (T-t) \, \left\| \psi \right\|_{1,\rho}
\left\|  u_2(\tau+h)-u_1(\tau) \right\|_{L^\infty(\mathbb{R}^d
\times (t,T))}
\end{aligned}
\end{equation}
for a.e.\ $t \in (0,T)$, for any $ \varphi \in\mathcal{D}(\mathbb{R}^d) $.
\end{lem}
\begin{proof}
From \eqref{eq: per-parti-H-dt-parti} one gets that, up to
subsequences, $\{ \psi_{n,\varepsilon}\} $ converges weakly in
$L^2_{\rho}\left(\mathbb{R}^d \times (\tau,T) \right)$ (for all
$ \tau \in (0,T) $) to a suitable $\psi_\varepsilon$. Moreover,
thanks to the uniform boundedness of $ \{\rho(\cdot)
\psi_{n,\varepsilon}(\cdot,t)\} $ in $L^1(\mathbb{R}^d)$ (see \eqref{eq:
markov-nonexp-L1}), for every $ t \in (0,T) $ there exists a
subsequence (which a priori may depend on $t$) such that $
\{\rho(\cdot)
\psi_{n,\varepsilon}(\cdot,t)\} $ converges in
$\sigma(\mathcal{M}(\mathbb{R}^d),C_c(\mathbb{R}^d)) $ to some
positive, finite Radon measure $\nu(t)$ (recall the preliminary results
of Section \ref{sec: not-def}). We aim at identifying (at least
for almost every $ t \in (0,T) $) $ \nu(t) $ with $\rho(\cdot)
\psi_{\varepsilon}(\cdot,t)$, so that a posteriori the subsequence does
not depend on $t$. In order to do that, let $t \in (0,T) $ be a
Lebesgue point of $ \psi_\varepsilon(\cdot,t) $ (as a curve in
$L^1((\tau,T); L^2_{\rho}({\mathbb R}^d)$).
Taking any $\varphi \in \mathcal{D}(\mathbb{R}^d) $ and using
\eqref{eq: 15pierre-2}, we obtain:
\begin{equation}\label{eq: 15pierre-idea-1}
\begin{aligned}
& \left| \int_{t}^{t+\delta} \int_{\mathbb{R}^d} \psi_{n,\varepsilon}(x,\tau) \varphi(x) \, \rho(x) \mathrm{d}x \mathrm{d}\tau - \int_{t}^{t+\delta} \int_{\mathbb{R}^d} \psi_{n,\varepsilon}(x,t) \varphi(x) \, \rho(x) \mathrm{d}x \mathrm{d}\tau \right| \\
\le &  \int_{t}^{t+\delta} \left| \int_{\mathbb{R}^d} \psi_{n,\varepsilon}(x,\tau) \varphi(x) \, \rho(x) \mathrm{d}x - \int_{\mathbb{R}^d} \psi_{n,\varepsilon}(x,t) \varphi(x) \, \rho(x) \mathrm{d}x \right| \mathrm{d}\tau  \\
\le & \int_{t}^{t+\delta} \underline{C} \, (\tau-t) \left\| \psi
\right\|_{1,\rho} \left\| \rho^{-1}(-\Delta)^s(\varphi)
\right\|_\infty \mathrm{d}\tau = \frac{\delta^2}{2} \underline{C} \left\| \psi
\right\|_{1,\rho} \left\|
\rho^{-1}(-\Delta)^s(\varphi)\right\|_\infty
\end{aligned}
\end{equation}
for all $ \delta $ sufficiently small. Letting $n \to \infty$ (up
to subsequences) in \eqref{eq: 15pierre-idea-1} yields
\begin{equation}\label{eq: 15pierre-idea-2}
%\begin{aligned}
\left| \int_{t}^{t+\delta} \int_{\mathbb{R}^d}
\psi_{\varepsilon}(x,\tau) \varphi(x) \, \rho(x)
\mathrm{d}x \mathrm{d}\tau - \delta \int_{\mathbb{R}^d}
\varphi(x) \, \mathrm{d}\nu(t) \right| \le \frac{\delta^2}{2} \underline{C}
\left\| \psi \right\|_{1,\rho} \left\|
\rho^{-1}(-\Delta)^s(\varphi)\right\|_\infty .
%\end{aligned}
\end{equation}
Dividing \eqref{eq: 15pierre-idea-2} by $\delta$ and letting
$\delta \to 0$ one deduces that (recall that $t$ is a Lebesgue
point for $ \psi_{\varepsilon}(\cdot,t) $)
\begin{equation*}%\label{eq: 15pierre-idea-3}
\int_{\mathbb{R}^d}  \psi_{\varepsilon}(x,t) \varphi(x) \,
\rho(x) \mathrm{d}x = \int_{\mathbb{R}^d} \varphi(x) \,
\mathrm{d}\nu(t) \, ,
\end{equation*}
which is valid for any $\varphi \in \mathcal{D}(\mathbb{R}^d)$,
whence $ \psi_{\varepsilon}(x,t) \rho(x) \mathrm{d}x =
\mathrm{d}\nu(t)$.

We now prove the claimed properties of
$\psi_{\varepsilon}$. Letting $ n \to \infty $ in \eqref{eq:
15pierre-3} (with $ t^\ast=T $ and $ t_\ast=t $) and using the just
proved convergence of $ \{\rho(\cdot)\psi_{n,\varepsilon}(\cdot,t)\} $
to $\rho(\cdot)\psi_{\varepsilon}(\cdot,t) $ in
$\sigma(\mathcal{M}(\mathbb{R}^d),C_c(\mathbb{R}^d)) $, we get
\begin{equation}\label{eq: 15pierre-lim}
\begin{aligned}
& \left| \int_{\mathbb{R}^d}  \psi(x) \xi_R(x) \, \rho(x)
\mathrm{d}x - \int_{\mathbb{R}^d} \psi_{\varepsilon}(x,t)
\xi_R(x) \, \rho(x)  \mathrm{d}x \right| \\
\le & \underline{C} \, (T-t) \left\| \psi \right\|_{1,\rho} c^{-1} \left( R^{-2s} + R^{-2s+\gamma} \right) \left\| (1+|x|^\gamma) (-\Delta)^{s}(\xi) \right\|_\infty
\end{aligned}
\end{equation}
for a.e.\ $t \in (0,T)$, $c$ being as in \eqref{eq: ass-rho}. Letting $R \to \infty $ in \eqref{eq:
15pierre-lim} we deduce \eqref{eq: cons-mass-epsilon}. Thanks to
\eqref{eq: 15pierre-4} and \eqref{eq: cons-mass-epsilon} we infer
in particular that
$$ \lim_{n \to \infty} \| \psi_{n,\varepsilon}(t) \|_{1,\rho} = \| \psi_{\varepsilon}(t) \|_{1,\rho} \, , $$
so that the convergence of
$ \{\rho(\cdot)\psi_{n,\varepsilon}(\cdot,t)\} $ to
$\rho(\cdot)\psi_{\varepsilon}(\cdot,t) $ also takes place in
$\sigma(\mathcal{M}(\mathbb{R}^d),C_b(\mathbb{R}^d))$. Recalling
that $g(\cdot,t)$ belongs to $C_b(\mathbb{R}^d)$ (Lemma \ref{lem:
potential-basic-prop-1}), we can let $n \to \infty$ in \eqref{eq:
equation-psi-g-derivata-parti-tempo} to obtain
\begin{equation}\label{eq: equation-psi-g-derivata-parti-tempo-epsilon}
\begin{aligned}
& \int_{\mathbb{R}^d} g(x,T) \psi(x) \, \rho(x) \mathrm{d}x - \int_{\mathbb{R}^d} g(x,t) \psi_{\varepsilon}(x,t) \, \rho(x) \mathrm{d}x  \\
 = & \lim_{n \to \infty} \left( \int_t^T \int_{\mathbb{R}^d} \left(a_n(x,\tau)+\varepsilon -a(x,\tau)\right)(-\Delta)^s(g)(x,\tau) \, \psi_{n,\varepsilon}(x,\tau) \, \mathrm{d}x \mathrm{d}\tau \right) \\
= &  \lim_{n \to \infty} \left( \int_t^T \int_{\mathbb{R}^d} \left(a_n(x,\tau)+\varepsilon -a(x,\tau)\right)\left(u_2(x,\tau+h)-u_1(x,\tau)\right) \psi_{n,\varepsilon}(x,\tau) \, \rho(x) \mathrm{d}x \mathrm{d}\tau \right) \\
= & \varepsilon \int_t^T \int_{\mathbb{R}^d}
\left(u_2(x,\tau+h)-u_1(x,\tau)\right) \psi_{\varepsilon}(x,\tau)
\, \rho(x) \mathrm{d}x \mathrm{d}\tau \ \ \
\textrm{for a.e.\ } t \in (0,T)  \, ,
\end{aligned}
\end{equation}
where in the last integral we can pass to the limit since $
\{\psi_{n,\varepsilon}\} $ tends to $\psi_{\varepsilon}$ weakly in
$L^2_{\rho}(\mathbb{R}^d \times (t,T) )$, $ \{a_n\} $ tends to
$a$ pointwise with $ \{ \| a_n \|_{L^\infty(\mathbb{R}^d \times(t,T) )} \} $ bounded, and $u_1, u_2 $
belong to $L^p_{\rho}(\mathbb{R}^d \times (t,T+h) )$ for all $p
\in [1,\infty]$. In particular, from \eqref{eq:
equation-psi-g-derivata-parti-tempo-epsilon} and \eqref{eq:
cons-mass-epsilon} we get \eqref{eq:
equation-psi-g-derivata-parti-tempo-epsilon-1}.
%\begin{equation}\label{eq: equation-psi-g-derivata-parti-tempo-epsilon-1}
%\begin{aligned}
%& \left| \int_{\mathbb{R}^d} g(x,T) \psi_T(x) \, |x|^{-\gamma} \, \mathrm{d}x - \int_{\mathbb{R}^d} g(x,s) \psi_{\varepsilon}(x,s) \, |x|^{-\gamma} \, \mathrm{d}x \right| \\
%\le & \varepsilon \, (T-s) \, \| \psi_T \|_{1,-\gamma} \left\|  \widetilde{u}(x,t+h)-{u}(x,t) \right\|_{L^\infty(\mathbb{R}^d \times (s,T))} \ \ \ \forall 0<s<T \, .
%\end{aligned}
%\end{equation}
Notice that, in a similarly way, we can pass to the limit in
\eqref{eq: 15pierre-1} (which actually holds for any $ t_\ast ,t^\ast \in
(0,T) $) and get \eqref{eq: 15pierre-1-epsilon}.
%\begin{equation}\label{eq: 15pierre-1-epsilon}
%\begin{aligned}
%& \int_{\mathbb{R}^d} |x|^{-\gamma} \psi_{\varepsilon}(x,t) \, \varphi(x) \, \mathrm{d}x - \int_{\mathbb{R}^d} |x|^{-\gamma} \psi_{\varepsilon}(x,s) \, \varphi(x) \, \mathrm{d}x \\
%= &  \int_{\mathbb{R}^d} (-\Delta)^s(\varphi)(x) \left( \int_{s}^t \left(a(x,r)+\varepsilon\right)\psi_{\varepsilon}(x,r) \, \mathrm{d}r \right) \mathrm{d}x \ \ \ \forall 0<s<t\le T \, .
%\end{aligned}
%\end{equation}
\end{proof}
% QUI

\subsection{Passing to the limit as $\varepsilon\to 0$ and proof of Theorem \ref{thm: teorema-uniqueness}}%\label{psi}
We are now in position to prove Theorem \ref{thm:
teorema-uniqueness}, using the strategy of \cite{Pierre}: we give some
detail for the reader's convenience.

% di nuovo, a posteriori, vista l'equazione che soddisfa il potenziale di \psi, \psi può essere definita per ogni t
\begin{proof}[Proof of Theorem \ref{thm: teorema-uniqueness}]
To begin with, we introduce the Riesz potential $H_\varepsilon(\cdot,t)$
of $\rho(\cdot) \psi_{\varepsilon}(\cdot,t)$.
%\begin{equation*}%\label{eq: potenziali-epsilon}
%H_\varepsilon(t)= I_{2s} \ast \left( |x|^{-\gamma} \psi_{\varepsilon}(t) \right) .
%\end{equation*}
Since we only know that $ \rho(\cdot) \psi_{\varepsilon}(\cdot,t) \in
L^1(\mathbb{R}^d) $, we have no information over the integrability
of $ H_\varepsilon(\cdot,t) $ other than $ L^1_{\rm loc}(\mathbb{R}^d) $ (by
classical results, see e.g.\ \cite[p.\ 61]{Landkof}). However, exploiting
\eqref{eq: 15pierre-1-epsilon} and proceeding once again as in the
proof of \eqref{eq: eq-formale-potenziale-epsilon}, we obtain
\begin{equation*}%\label{eq: potenziali-epsilon-decrescenza}
I_{2s} \ast \left( \rho \psi \right) - H_\varepsilon(\cdot,t) =
\int_t^T (a(\cdot,\tau)+\varepsilon) \, \psi_{\varepsilon}(\cdot,\tau) \,
\mathrm{d}\tau \ge 0 \ \ \ \textrm{for a.e.\ } t \in (0,T) \,
,
\end{equation*}
% ad es. da questa eq. vedo che H(x,t) è sicuramente misurabile in (x,t)
whence, in particular,
\begin{equation}\label{eq: potenziali-epsilon-decrescenza-bis}
0 \le  H_\varepsilon(x,t_1) \le H_\varepsilon(x,t_2) \le
H_\varepsilon(x,T) = I_{2s} \ast \left( \rho \psi \right)(x)
\end{equation}
for $\textrm{a.e.\ } 0 < t_1 \le t_2 \le T $ and $ \textrm{a.e.\ } x \in \mathbb{R}^d$.
The above inequality shows that  $H_\varepsilon(\cdot,t)$ belongs to $
L^p(\mathbb{R}^d) $ at least for the same $p $ for which $
H_\varepsilon(\cdot,T) $ does, namely for any $ p \in (d/(d-2),\infty]
$.
% cazzata
%The fact that \eqref{eq: potenziali-epsilon-decrescenza-bis}
%holds for \emph{every} $ x $ rather than for \emph{almost every}
%$x$ follows by standard potential theory: one uses the strategy of
%\cite[Theorem 1.12]{Landkof} and exploits \cite[Lemma 1.1]{Landkof} to
%find that two potentials ordered for a.e.\ $x$ are actually
%ordered for every $x$.

Our next goal is to let $ \varepsilon \to 0 $ (along a fixed sequence whose index for the moment we omit, in order to improve readability). % fisso una successione così non ho i casini dei quasi ovunque
Thanks to the boundedness of $\{ \rho(\cdot) \psi_{\varepsilon}(\cdot,t) \} $ in
$L^1(\mathbb{R}^d) $ (trivial consequence of \eqref{eq:
cons-mass-epsilon}), for a.e.\ $t\in (0,T) $ there exists a
subsequence $\{ \varepsilon_n \} $ (a priori depending on $t$)
such that $\{ \rho(\cdot) \psi_{\varepsilon}(\cdot,t)\} $ converges
to a positive finite Radon measure $\nu(t)$ in
$\sigma(\mathcal{M}(\mathbb{R}^d),C_c(\mathbb{R}^d))$. In order to
overcome the possible dependence of $ \{ \varepsilon_n \} $ on $ t
$, we exploit the properties of $ \{ H_\varepsilon \} $. First
notice that \eqref{eq: potenziali-epsilon-decrescenza-bis} ensures
the uniform boundedness of $ \{H_\varepsilon \}  $ in $
L^p(\mathbb{R}^d \times (0,T)) $ for any $ p \in (d/(d-2),\infty]
$. This entails the existence of a decreasing subsequence $\{
\varepsilon_m \} $ such that
$ \{H_{\varepsilon_m}\} $ converges weakly in $L^p(\mathbb{R}^d
\times (0,T))$ to a suitable limit $H$. Mazur's Lemma implies
that there exists a sequence $ \{H_k\} $ of convex combinations of
$ \{H_{\varepsilon_m}\} $ that converges \emph{strongly} to $H$ in
$L^p(\mathbb{R}^d \times (0,T))$. By definition, the sequence $ \{
H_k \} $ is of the form
\begin{equation*}%\label{eq: comb-conv-H}
H_k = \sum_{m=1}^{M_k} \lambda_{m,k} H_{\varepsilon_m} \, , \ \ \
\sum_{m=1}^{M_k} \lambda_{m,k}=1
\end{equation*}
for some sequence $ \{M_k\} \subset \mathbb{N} $ and a suitable
choice of the coefficients $\lambda_{m,k} \in [0,1] $. With no
loss of generality we shall also assume that
\begin{equation*}\label{eq: comb-conv-H-epsilon}
\lim_{k \rightarrow \infty} \left(\sum_{m=1}^{M_k} \varepsilon_m
\lambda_{m,k}\right) = 0 \, .
\end{equation*}
This can be justified by applying iteratively Mazur's Lemma on
suitable subsequences of $\{ H_{\varepsilon_m} \}$.
Now notice that the function whose Riesz potential is $H_k$ is
$$ f_k(x,t) = \sum_{m=1}^{M_k} \lambda_{m,k} \, \rho(x) \psi_{\varepsilon_m}(x,t) \, . $$
Multiplying \eqref{eq:
equation-psi-g-derivata-parti-tempo-epsilon-1} (with $
\varepsilon=\varepsilon_m $) by $ \lambda_{m,k} $ and summing over
$k$, one gets that $f_k$ satisfies
\begin{equation}\label{eq: equation-psi-g-derivata-parti-tempo-k}
\begin{aligned}
& \left| \int_{\mathbb{R}^d} g(x,T) \psi(x) \, \rho(x) \mathrm{d}x - \int_{\mathbb{R}^d} g(x,t) f_k(x,t) \, \mathrm{d}x \right| \\
\le & \left(\sum_{m=1}^{M_k} \varepsilon_m \lambda_{m,k}\right)
(T-t) \, \| \psi \|_{1,\rho} \left\|  u_2(\tau+h)-u_1(\tau)
\right\|_{L^\infty(\mathbb{R}^d \times (t,T))}
\end{aligned}
\end{equation}
for a.e.\ $t \in (0,T)$, whereas from \eqref{eq:
cons-mass-epsilon} and \eqref{eq: 15pierre-lim} we infer that
\begin{equation}\label{eq: 15pierre-lim-k}
\begin{aligned}
 &\left| \int_{\mathbb{R}^d} \psi(x) \xi_R(x) \, \rho(x) \mathrm{d}x - \int_{\mathbb{R}^d} f_k(x,t) \, \xi_R(x) \, \mathrm{d}x \right|\\
 \le  & \underline{C} \,(T-t) \left\| \psi \right\|_{1,\rho} c^{-1} \left( R^{-2s} + R^{-2s+\gamma} \right) \left\| (1+|x|^\gamma) (-\Delta)^{s}(\xi) \right\|_\infty
\end{aligned}
\end{equation}
for a.e.\ $ t \in (0,T) $ and
\begin{equation}\label{eq: 15pierre-lim-k-cons-massa}
\int_{\mathbb{R}^d}  \psi(x) \, \rho(x) \mathrm{d}x =
\int_{\mathbb{R}^d} f_k(x,t) \, \mathrm{d}x \ \ \ \textrm{for
a.e.\ } t \in (0,T)  \, .
\end{equation}
Letting $k \to \infty$ we find that, for a.e.\ $ t \in (0,T) $,
there exists a subsequence of $ \{f_k(\cdot,t)\} $ (a priori depending
on $t$) that converges in
$\sigma(\mathcal{M}(\mathbb{R}^d),C_c(\mathbb{R}^d))$ to a
positive, finite Radon measure $\nu(t)$. But the fact that $ \{ H_k \} $
converges strongly in $ L^p(\mathbb{R}^d\times(0,T)) $ to $ H $
forces the potential of $ \nu(t) $ to coincide a.e.\ with $H(\cdot,t)$.
This is a consequence of \cite[Theorem 3.8]{Landkof}. By \cite[Theorem 1.12]{Landkof} we therefore deduce that the limit $ \nu(t) $ is
uniquely determined by its potential $H(\cdot,t)$. This identification
allows to assert that for a.e.\ $ t \in (0,T) $ the \emph{whole}
sequence $ \{ f_k(\cdot,t) \} $ converges to $ \nu(t) $ in
$\sigma(\mathcal{M}(\mathbb{R}^d),C_c(\mathbb{R}^d))$.

Passing to the limit in \eqref{eq:
potenziali-epsilon-decrescenza-bis} (after having set $
\varepsilon=\varepsilon_m $, multiplied by $ \lambda_{m,k} $ and
summed over $k$) we deduce that also the potentials $ H(\cdot,t) $ of
$\nu(t)$ are ordered and bounded above by $ I_{2s} \ast (
\rho \psi ) $:
\begin{equation}\label{eq: potenziali-epsilon-decrescenza-k}
0 \le  H(x,t_1) \le H(x,t_2) \le I_{2s} \ast \left( \rho \psi \right)(x) \quad \textrm{for a.e.\ } 0 < t_1 \le t_2 \le T \, , \ \textrm{for a.e.\ } x \in \mathbb{R}^d .
\end{equation}
%where the passage from \emph{almost every} $x$ to \emph{every} $x$ in \eqref{eq: potenziali-epsilon-decrescenza-k} follows again by \cite[Lem. 1.1 and Th. 1.12]{Landkof}.
Letting $k \to \infty$ in \eqref{eq: 15pierre-lim-k} yields
\begin{equation}\label{eq: 15pierre-lim-k-bis}
\begin{aligned}
 &\left| \int_{\mathbb{R}^d}  \psi(x) \xi_R(x) \, \rho(x) \mathrm{d}x - \int_{\mathbb{R}^d}  \xi_R(x) \, \mathrm{d}\nu(t) \right| \\
 \le  & \underline{C} \, (T-t) \left\| \psi \right\|_{1,\rho} c^{-1} \left( R^{-2s} + R^{-2s+\gamma} \right) \left\| (1+|x|^\gamma) (-\Delta)^{s}(\xi) \right\|_\infty
\end{aligned}
\end{equation}
for a.e.\ $t \in (0,T)$, whence, letting $R\to\infty$ in
\eqref{eq: 15pierre-lim-k-bis}, we obtain
\begin{equation}\label{eq: 15pierre-lim-k-cons-massa-bis}
\int_{\mathbb{R}^d}  \psi(x) \, \rho(x) \mathrm{d}x =
\int_{\mathbb{R}^d} \mathrm{d}\nu(t) \ \ \ \textrm{for a.e.\ }
t \in (0,T) \, .
\end{equation}
Gathering \eqref{eq: 15pierre-lim-k-cons-massa} and \eqref{eq:
15pierre-lim-k-cons-massa-bis} we infer that $ \{f_k(\cdot,t)\} $
converges to $\nu(t)$ also in
$\sigma(\mathcal{M}(\mathbb{R}^d),C_b(\mathbb{R}^d))$: this allows us
to pass to the limit in \eqref{eq:
equation-psi-g-derivata-parti-tempo-k} to get (by
exploiting \eqref{eq: comb-conv-H-epsilon} as well) identity \eqref{eq:
equation-psi-g-derivata-parti-tempo-k-limite}. As a consequence of
the monotonicity given by \eqref{eq:
potenziali-epsilon-decrescenza-k} and thanks to \eqref{eq:
15pierre-lim-k-bis}-\eqref{eq: 15pierre-lim-k-cons-massa-bis}, the curve $\nu(t)$
can be extended to \emph{every} $ t \in (0,T] $ so that it still
satisfies \eqref{eq: potenziali-epsilon-decrescenza-k}-\eqref{eq: 15pierre-lim-k-cons-massa-bis} (one uses again
\cite[Theorem 3.8]{Landkof}).
%From now on we shall assume this fact tacitly.
Recalling that $g(x,t)=U_2(x,t+h)-U_1(x,t)$ and that
potentials do not increase in time (Lemma \ref{lem:
monotonicity}), we have that $g(x,t) \le U_2(x,h)-U_1(x,t_0)$
holds for all $x \in \mathbb{R}^d $ and all $t_0 >t $. Because
$\nu(t)$ is a positive finite Radon measure, this fact and \eqref{eq:
equation-psi-g-derivata-parti-tempo-k-limite} imply that
\begin{equation}\label{eq: equation-psi-g-derivata-parti-tempo-prova-1}
\int_{\mathbb{R}^d} g(x,T) \psi(x) \, \rho(x) \mathrm{d}x
\le \int_{\mathbb{R}^d} \left(U_2(x,h)-U_1(x,t_0) \right)
\mathrm{d}\nu(t) \ \ \ \forall t_0>t \, .
\end{equation}
Our next goal is to let $t$ tend to zero in \eqref{eq:
equation-psi-g-derivata-parti-tempo-prova-1}. Since the mass of
$\nu(t)$ is constant (formula \eqref{eq:
15pierre-lim-k-cons-massa-bis}), up to subsequences $ \nu(t) $
converges to a suitable positive finite Radon measure $\nu$ in
$\sigma(\mathcal{M}(\mathbb{R}^d),C_c(\mathbb{R}^d))$ as $ t \downarrow 0 $. Moreover,
by \eqref{eq: potenziali-epsilon-decrescenza-k}, we know that the
potentials $H(\cdot,t)$ of $\nu(t)$ are nondecreasing in $t$ (for a.e.\ $x$): in particular, $ H(\cdot,t) $ admits a pointwise limit almost everywhere $H_0$ as $t \downarrow 0$. Theorem 3.8 of \cite{Landkof} ensures that $H_0$ coincides almost everywhere with the potential of the limit measure $\nu$ (which therefore does not depend on the subsequence). We can then pass to the limit in the integral
%However, in order to pass to the limit on the r.h.s.\ of \eqref{eq: equation-psi-g-derivata-parti-tempo-prova-1}, this is not sufficient in order to pass to the limit on the r.h.s.\ of \eqref{eq: equation-psi-g-derivata-parti-tempo-prova-1}, because the potentials of $ u_1 $ and $ u_2 $ are not compactly supported functions. To circumvent this difficulty, notice
\begin{equation}\label{eq: integrale-finale-lim}
\int_{\mathbb{R}^d} U_1(x,t_0) \, \mathrm{d}\nu(t) \, .
\end{equation}
Indeed, by Fubini's Theorem, \eqref{eq:
integrale-finale-lim} is equal to
\begin{equation}\label{eq: integrale-finale-lim-2}
\int_{\mathbb{R}^d} u_1(x,t_0) H(x,t) \, \rho(x)\mathrm{d}x
\, .
\end{equation}
Passing to the limit in \eqref{eq: integrale-finale-lim-2} as $ t
\downarrow 0 $ we get that
\begin{equation}\label{eq: integrale-finale-lim-3}
\lim_{t \downarrow 0} \int_{\mathbb{R}^d} u_1(x,t_0) H(x,t) \,
\rho(x) \mathrm{d}x = \int_{\mathbb{R}^d} u_1(x,t_0) H_0(x)
\, \rho(x) \mathrm{d}x
\end{equation}
%since $H(t)$ is nonincreasing as $t \downarrow 0$ and converges a.e.\ to $H_0$.
by dominated convergence. Recalling that $H_0$ is the
potential of $\nu$, and using again Fubini's Theorem, \eqref{eq:
integrale-finale-lim-3} can be rewritten as
\begin{equation*}%\label{eq: integrale-finale-lim-4}
\lim_{t \downarrow 0} \int_{\mathbb{R}^d} U_1(x,t_0) \,
\mathrm{d}\nu(t) = \int_{\mathbb{R}^d} U_1(x,t_0)   \,
\mathrm{d}\nu \, .
\end{equation*}
One proceeds similarly for the integral
\begin{equation*}%\label{eq: integrale-finale-lim-secondo}
\int_{\mathbb{R}^d} U_2(x,h) \,  \mathrm{d}\nu(t) \, .
\end{equation*}
Hence, passing to the limit as $t \downarrow 0$ in \eqref{eq:
equation-psi-g-derivata-parti-tempo-prova-1} yields
\begin{equation}\label{eq: equation-psi-g-derivata-parti-tempo-prova-1-limite}
\int_{\mathbb{R}^d} g(x,T) \psi(x) \, \rho(x) \mathrm{d}x
\le \int_{\mathbb{R}^d} \left( U_2(x,h)-U_1(x,t_0) \right)
\mathrm{d}\nu \ \ \ \forall t_0 > 0 \, .
\end{equation}
Now we let $t_0 \downarrow 0 $ in \eqref{eq:
equation-psi-g-derivata-parti-tempo-prova-1-limite}. By monotone
convergence (Lemmas \ref{lem: monotonicity} and \ref{lem:
potenziale}) we obtain
\begin{equation}\label{eq: equation-psi-g-derivata-parti-tempo-prova-1-limite-mono}
\int_{\mathbb{R}^d} g(x,T) \psi(x) \, \rho(x) \mathrm{d}x \le
\int_{\mathbb{R}^d} \left( U_2(x,h)-U^\mu(x) \right) \mathrm{d}\nu
\, ;
\end{equation}
in this step it is crucial that the limit of $ U_1(x,t_0) $ to
$U^\mu(x)$ is taken \emph{for every} $x \in \mathbb{R}^d$ (Lemma
\ref{lem: potenziale}), because we have no information over $\nu$
besides the fact that it is a positive finite Radon measure.
%(in particular, we cannot claim that it is absolutely continuous w.r.t.\ the Lebesgue measure).
Still by monotonicity we have that $ U_2(x,h) \le U^\mu(x)$ for
every $x \in \mathbb{R}^d$. Thus, from \eqref{eq:
equation-psi-g-derivata-parti-tempo-prova-1-limite-mono} it
follows that
\begin{equation}\label{eq: equation-psi-g-derivata-parti-tempo-prova-2-limite-mono}
\int_{\mathbb{R}^d} g(x,T) \psi(x) \, \rho(x) \mathrm{d}x
\le 0 \, .
\end{equation}
Since \eqref{eq: equation-psi-g-derivata-parti-tempo-prova-2-limite-mono} holds for
any $h,T>0 $ and any $ \psi \in
\mathcal{D}_{+}(\mathbb{R}^d) $, we infer that $ U_2 \le U_1 $.
By interchanging the role of $ u_1 $ and $ u_2 $ we get that $
U_1 \le U_2 $, whence $ U_1 = U_2 $ and $ u_1=u_2 $.
\end{proof}

\appendix

\section{}\label{sect: app-basic}
We recall here some basic properties of the fractional Laplacian
(and of a similar nonlocal, nonlinear operator) of functions in $ \mathcal{D}(\mathbb{R}^d) $. We omit the proofs of the first two lemmas, since they follow by exploiting the same strategy of \cite[Lemma 2.1]{BV}.

\begin{lem} \label{lem:decay-lap-1}
The $s$-Laplacian $(-\Delta)^s(\phi)(x) $  of any $
\phi \in \mathcal{D}(\mathbb{R}^d) $ is a regular function which
decays (together with its derivatives) at least like $|x|^{-d-2s}$
as $|x| \to \infty $.
\end{lem}

\begin{lem} \label{lem:decay-lap-2}
For any $ \phi \in \mathcal{D}(\mathbb{R}^d) $, the function
\begin{equation*}%\label{eq: def-lap-quadrato}
l_s(\phi)(x) := \int_{\mathbb{R}^d}
\frac{(\phi(x)-\phi(y))^2}{|x-y|^{d+2s}} \, \mathrm{d}y \ \ \
\forall x \in \mathbb{R}^d
\end{equation*}
is regular and decays (together with its derivatives) at least
like $ |x|^{-d-2s} $ as $|x| \to \infty $.
\end{lem}

%\TR{The next result shows a crucial scaling property for $ (-\Delta)^s $ and $ l_s $.}
\begin{lem} \label{lem:decay-lap-cutoff}
For any $R>0$, let $\xi_R$ be the cut-off function
\begin{equation*}%\label{eq: def-general-cutoff}
\xi_R(x):=\xi\left(\frac{x}{R}\right)  \ \ \ \forall x \in
\mathbb{R}^d \, ,
\end{equation*}
where $ \xi(x) $ is a positive, regular function such that $
\|\xi\|_\infty \le 1 $, $ \xi \equiv 1$ in $B_1 $ and $ \xi \equiv 0$ in $B_2^c
$. Then, $ (-\Delta)^s(\xi_R) $ and $ l_s(\xi_R) $ enjoy the following property:
\begin{equation*}%\label{eq: scaling-lap}
(-\Delta)^s(\xi_R)(x)=\frac{1}{R^{2s}}(-\Delta)^s(\xi)\left(
\frac{x}{R} \right) , \quad l_s(\xi_R)(x)=\frac{1}{R^{2s}}l_s(\xi)\left( \frac{x}{R} \right) \ \ \ \forall x \in \mathbb{R}^d \, .
\end{equation*}
\begin{proof}
We only prove the result for $l_s(\xi_R)$, since the proof for
$(-\Delta)^s(\xi_R)$ is identical. Letting $ \widetilde{y}=y/R $, one has:
\begin{equation*}%\label{eq: lap-quadrato-xiR}
l_s(\xi_R)(x)=  \int_{\mathbb{R}^d} \frac{(\xi_R(x)-\xi_R(y))^2}{|x-y|^{d+2s}} \, \mathrm{d}y
 %= & R^d \int_{\mathbb{R}^d} \frac{(\xi( {x}/{R} )-\xi(\widetilde{y}) )^2}{|x-R \widetilde{y}|^{d+2s}} \, \mathrm{d}\widetilde{y}
 = \frac{1}{R^{2s}} \int_{\mathbb{R}^d} \frac{(\xi( {x}/{R} )-\xi(\widetilde{y}) )^2}{|x/R-\widetilde{y}|^{d+2s}} \, \mathrm{d}\widetilde{y} = \frac{1}{R^{2s}}l_s(\xi)\left( \frac{x}{R} \right).
\end{equation*}
\end{proof}
\end{lem}

The next lemmas contain technical ingredients concerning fractional Sobolev spaces and Riesz potentials, which we need in the proofs of our existence and uniqueness results.
\begin{lem}\label{lem: Hs-loc}
Let $d>2s $ and assume that $\rho$ satisfies \eqref{eq: ass-rho} for some $ \gamma \in (0,d+2s] $. Consider a function $v \in
L^2_{\rm loc}((0,\infty);\dot{H}^s(\mathbb{R}^d)) $ such that, for all
$t_2>t_1>0$,
\begin{equation}\label{eq: lemma-Hs-3}
\int_{t_1}^{t_2} \! \int_{\mathbb{R}^d} \left| v(x,t) \right|^2
\rho(x) \mathrm{d}x \mathrm{d}t \le C_0 \, ,
\end{equation}
\begin{equation}\label{eq: lemma-Hs-1}
\int_{t_1}^{t_2} \! \int_{\mathbb{R}^d} \left|
(-\Delta)^{\frac{s}{2}} \left( v \right) (x,t) \right|^2
\mathrm{d}x \mathrm{d}t \le C_0
\end{equation}
and
\begin{equation}\label{eq: lemma-Hs-2}
\int_{t_1}^{t_2} \! \int_{\mathbb{R}^d} \left| v_t(x,t)  \right|^2
\rho(x) \mathrm{d}x \mathrm{d}t  \le C_0 \, ,
\end{equation}
where $C_0$ is a positive constant depending only on $t_1$ and
$t_2$. Take any cut-off functions $ \xi_1 \in
C^\infty_c(\mathbb{R}^d) $, $ \xi_2 \in C^\infty_c((0,\infty)) $
and define $ v_c: \mathbb{R}^d \rightarrow \mathbb{R}$ as follows:
\begin{equation*}%\label{eq: lemma-Hs-def-vc}
v_c(x,t):= \xi_1(x) \xi_2(t) v(x,t)  \ \ \ \forall(x,t) \in
\mathbb{R}^d \times \mathbb{R} \, ,
\end{equation*}
where we implicitly assume $ \xi_2 $ and $ v $ to be zero for $
t<0 $. Then
\begin{equation}\label{eq: lemma-Hs-stima-Hs}
\left\| v_c \right\|_{{H^s}  (\mathbb{R}^{d+1})}^2 = \left\| v_c
\right\|_{L^2(\mathbb{R}^{d+1})}^2 + \left\| v_c
\right\|_{\dot{H}^s(\mathbb{R}^{d+1})}^2  \le C^\prime
\end{equation}
for a positive constant $C^\prime$ that depends only on $ \rho $, $ \xi_1 $
and $ \xi_2 $ (also through $C_0$).
\begin{proof}
The validity of
\begin{equation}\label{eq: lemma-Hs-stima-L2-prova}
\left\| v_c  \right\|_{L^2(\mathbb{R}^{d+1})}^2  \le C^\prime
\end{equation}
is an immediate consequence of \eqref{eq: lemma-Hs-3} and of the
fact that $\rho$ is bounded away from zero on compact sets (from
now on $ C^\prime $ will be a constant as in the statement that may change from line to line). Moreover,
since $ (v_c)_t = \xi_1 \xi_2^\prime v + \xi_1 \xi_2 v_t$,
by \eqref{eq: lemma-Hs-3}, \eqref{eq: lemma-Hs-2} and again the
fact that $\rho$ is bounded away from zero on compact sets we
deduce that
\begin{equation}\label{eq: lemma-Hs-stima-L2dt-prova}
\left\| (v_c)_t \right\|_{L^2(\mathbb{R}^{d+1})}^2 \le C^\prime \,
.
\end{equation}
Now we have to handle the spatial regularity of $ v_c $. Straightforward computations show that
\begin{equation}\label{eq: derivate-spaziali-w-2}
\begin{aligned}
\left\|  v_c(t)  \right\|_{\dot{H}^s(\mathbb{R}^d)}^2  = & \frac{C_{d,s}}{2} \, \xi_2^2(t) \int_{\mathbb{R}^d} \xi_1^2(x) \left( \int_{\mathbb{R}^d}  \frac{\left( v(x,t) - v(y,t) \right)^2}{|x-y|^{d+2s}}  \, \mathrm{d}y \right) \mathrm{d}x \\
  & + \frac{C_{d,s}}{2} \, \xi_2^2(t) \int_{\mathbb{R}^d} \left|v(y,t)\right|^2 \left( \int_{\mathbb{R}^d}  \frac{\left( \xi_1(x) - \xi_1(y) \right)^2}{|x-y|^{d+2s}}  \, \mathrm{d}x \right) \mathrm{d}y  \\
& + C_{d,s}  \, \xi_2^2(t) \int_{\mathbb{R}^d} \!
\int_{\mathbb{R}^d} \xi_1(x) v(y,t) \frac{\left(v(x,t) - v(y,t)
\right) \left( \xi_1(x) - \xi_1(y) \right) }{|x-y|^{d+2s}}  \,
\mathrm{d}x \mathrm{d}y  \, .
\end{aligned}
\end{equation}
The Cauchy-Schwarz inequality allows us to bound the third integral on the r.h.s.\ of \eqref{eq:
derivate-spaziali-w-2} by the first two integrals. As concerns the first one, we have:
\begin{equation}\label{eq: derivate-spaziali-w-3}
\frac{C_{d,s}}{2} \, \xi_2^2(t) \int_{\mathbb{R}^d} \xi_1^2(x) \left( \int_{\mathbb{R}^d}  \frac{\left( v(x,t) - v(y,t) \right)^2}{|x-y|^{d+2s}}  \, \mathrm{d}y \right) \mathrm{d}x \\
\le \chi_{\operatorname{supp}{\xi_2} }(t) \left\| \xi_2
\right\|_\infty^2 \left\| \xi_1 \right\|_\infty^2 \left\|  v(t)
\right\|_{\dot{H}^s(\mathbb{R}^d)}^2 .
\end{equation}
In order to bound the second integral, it is important to recall
that the function $l_s(\xi_1)(y)$
% $$ l(\xi_1)(y) = \int_{\mathbb{R}^d}  \frac{\left( \xi_1(x) - \xi_1(y) \right)^2}{|x-y|^{d+2s}}  \, \mathrm{d}x \ \ \ \forall y \in \mathbb{R}^d $$
is regular and decays at least like $|y|^{-d-2s} $ as $ |y| \to
\infty $ (for the definition and properties of $l_s$ see
Lemmas \ref{lem:decay-lap-2} and \ref{lem:decay-lap-cutoff}). Hence, thanks to the assumptions on $\rho$
and $\gamma$, we infer that
%\begin{equation*}%\label{eq: controllo-int-rho}
%\int_{\mathbb{R}^d}  \frac{\left( \xi_1(x) - \xi_1(y) \right)^2}{|x-y|^{d+2s}}  \, \mathrm{d}x \le c^\prime \rho(y) \ \ \ \forall y \in \mathbb{R}^d
%\end{equation*}
%for a suitable positive constant $c^\prime $. Therefore,
\begin{equation}\label{eq: controllo-int-rho-2}
\xi_2^2(t) \int_{\mathbb{R}^d} \left|v(y,t)\right|^2 \left(
\int_{\mathbb{R}^d}  \frac{\left( \xi_1(x) - \xi_1(y)
\right)^2}{|x-y|^{d+2s}}  \, \mathrm{d}x \right) \mathrm{d}y \le
C^\prime \chi_{\operatorname{supp}\xi_2 }(t) \left\| \xi_2
\right\|_\infty^2 \int_{\mathbb{R}^d} \left| v(y,t) \right|^2
\rho(y) \mathrm{d}y \, .
\end{equation}
Integrating in time
\eqref{eq: derivate-spaziali-w-2}, using \eqref{eq:
derivate-spaziali-w-3}, \eqref{eq: controllo-int-rho-2},
\eqref{eq: lemma-Hs-3}, \eqref{eq: lemma-Hs-1} and recalling the validity of the identity $ \left\| (-\Delta)^{\frac{s}{2}} (v_c)(t) \right\|_{L^2(\mathbb{R}^d)}^2 = \left\| v_c(t) \right\|_{\dot{H}^s(\mathbb{R}^d)}^2 $, we then get
\begin{equation}\label{eq: lemma-Hs-stima-Hs-prova}
\left\| (-\Delta)^{\frac{s}{2}} ( v_c )
\right\|_{L^2(\mathbb{R}^{d+1})}^2 \le C^\prime \, .
\end{equation}
By exploiting \eqref{eq: lemma-Hs-stima-L2-prova}, \eqref{eq:
lemma-Hs-stima-L2dt-prova} and \eqref{eq: lemma-Hs-stima-Hs-prova}
one deduces \eqref{eq: lemma-Hs-stima-Hs}, e.g.~by using Fourier transform methods.
%First we recall the identity
%\begin{equation*}%\label{eq: derivate-spaziali-w}
%\int_{\mathbb{R}^d} \left| (-\Delta)^{\frac{s}{2}} \left( v_c
%\right) (x,t) \right|^2  \mathrm{d}x  = C_{d,s} \,
%\int_{\mathbb{R}^d} \! \int_{\mathbb{R}^d}
%\frac{\left(v_c(x,t)-v_c(y,t) \right)^2}{|x-y|^{d+2s}} \,
%\mathrm{d}x \, \mathrm{d}y = \left\|  v_c(t)
%\right\|_{\dot{H}^s(\mathbb{R}^d)}^2 .
%\end{equation*}
\end{proof}
\end{lem}

\begin{lem}\label{lem: decay-conv}
Let $d>2s$ and $ \phi:\mathbb{R}^d \rightarrow \mathbb{R} $ be a
continuous function which belongs to $L^1(\mathbb{R}^d)$ and
decays at least like $ |x|^{-d} $ as $ |x| \to \infty $. Then, the
convolution $I_{2s} \ast \phi$ (namely, the Riesz potential of
$\phi$) is also a continuous function, decaying at least like $
|x|^{-d+2s} $ as $ |x| \to \infty $.
\begin{proof}
The idea of the proof is to split the convolution $ (I_{2s} \ast \phi)(x) $ in the three regions $ B^c_{2|x|}(0)$, $ B_{{|x|}/{2}}(x) $, $ B_{2|x|}(0) \setminus B_{|x|/2}(x) $ and use there the decay and integrability properties of $ \phi $ and $ I_{2s} $. We omit the details.
\end{proof}
\end{lem}

\begin{lem}\label{lem: potential-basic-prop-1}
Let $ d > 2s $ and assume that $\rho$ satisfies \eqref{eq: ass-rho} for some $ \gamma \in (0,2s) $.  Let $ v \in
L^1_{\rho}(\mathbb{R}^d) \cap L^\infty(\mathbb{R}^d) $ and $
U^{v}_{\rho} $ be the Riesz potential of $ \rho v$. Then $U^{v}_\rho$ belongs to $ C(\mathbb{R}^d)
\cap L^p(\mathbb{R}^d) $ for all $p$ such that
\begin{equation}\label{eq: potential-v-Lp}
p \in \left( \frac{d}{d-2s} , \infty \right] .
\end{equation}
\begin{proof}
In order to prove that $ U^v_\rho $ belongs to $ C(\mathbb{R}^d)
\cap L^p(\mathbb{R}^d) $ for all $p$ satisfying \eqref{eq:
potential-v-Lp}, we proceed as follows:
\[
%\begin{aligned}
U^v_\rho(x) = \underbrace{ \int_{B_1(0)} \rho(y)\,v(y) \,
I_{2s}(x-y) \, \mathrm{d}y}_{U^v_{\rho,1}(x)} + \underbrace{
\int_{\mathbb{R}^d} \chi_{B_1^c(0)}(y) \, \rho(y)\,v(y) \,
I_{2s}(x-y) \, \mathrm{d}y }_{U^v_{\rho,2}(x)} \, .
%\end{aligned}
\]
Exploiting the fact that $ v \in L^\infty(\mathbb{R}^d) $ and $
\gamma<2s $ (so that $ |y|^{-d+2s}\,\rho(y) $ is locally
integrable), it is easily seen that $ U^v_{\rho,1}(x) $ is a
continuous function which decays at least like $ |x|^{-d+2s} $ as
$ |x| \to \infty $. In particular, it belongs to $
L^p(\mathbb{R}^d) $ for all $ p $ satisfying \eqref{eq:
potential-v-Lp}. As concerns $U^v_{\rho,2}(x)$, notice that since $ v \in
L^1_{\rho}(\mathbb{R}^d) \cap L^\infty(\mathbb{R}^d) $ we have
that the function $\chi_{B_1^c(0)} \rho v$ belongs to $
L^1(\mathbb{R}^d) \cap L^\infty(\mathbb{R}^d)$. Hence $ U^v_{\rho,2}(x)$ is continuous too. To
prove that it belongs to $ L^p(\mathbb{R}^d) $ for all $ p $
satisfying \eqref{eq: potential-v-Lp}, we write:
\begin{equation}\label{eq: prop-int-uv-2}
U^v_{\rho,2} = \left( \chi_{B_1(0)} \,  I_{2s}  \right) \ast
\left( \chi_{B_1^c(0)} \rho v \right) + \left(
\chi_{B_1^c(0)} \, I_{2s} \right) \ast
\left(\chi_{B_1^c(0)}\rho v \right) ;
\end{equation}
since $\chi_{B_1(0)} \,  I_{2s} \in L^1(\mathbb{R}^d)$ and $\chi_{B_1^c(0)} \rho v \in
L^1(\mathbb{R}^d) \cap L^\infty(\mathbb{R}^d)$, the first
convolution in \eqref{eq: prop-int-uv-2} belongs to $
L^1(\mathbb{R}^d) \cap L^\infty(\mathbb{R}^d) $. Using the fact
that $\chi_{B_1^c(0)} \, I_{2s} \in L^p(\mathbb{R}^d) $ for all $
p $ as in \eqref{eq: potential-v-Lp} and
$\chi_{B_1^c(0)} \rho v \in L^1(\mathbb{R}^d)$, we infer
that the second convolution in \eqref{eq: prop-int-uv-2} belongs
to $L^p(\mathbb{R}^d)$ for all such $ p $. The latter property is then inherited by $
U^v_{\rho,2} $.
\end{proof}
%Finally, the fact that in all of the cases the norms $ \|
%U^{v}_\gamma \|_{L^p(\mathbb{R}^d)} $ and $ \| U^{v}_\gamma
%\|_{W^{r,p}(\mathbb{R}^d)} $ can be bounded from above by
%constants depending on $ v $ only through $ \| v \|_{-1,\gamma} $
%and $ \| v \|_\infty $ is just a consequence of the above
%computations.
\end{lem}

\section{}\label{sect: app-operatore}
This section is devoted to give a sketch of the
proofs of Theorem \ref{thm: self-adj} and of the forthcoming Proposition \ref{pro: laplaciano-Lp}. %\TR{We refer to the notation introduced in Section \ref{sect: state}}.
% We denote by $X_{s,\gamma}$ the space $X_{s,\rho}$ when $\rho(x)=|x|^{-\gamma}$.

\begin{proof}[Sketch of proof of Theorem \ref{thm: self-adj}]
We start from the validity of the fractional
``integration by parts'' formula
\begin{equation}\label{eq: int-parti-app}
\frac{C_{d,s}}{2} \int_{\mathbb{R}^d} \int_{\mathbb{R}^d}
\frac{(\phi(x)-\phi(y))(\psi(x)-\psi(y))}{|x-y|^{d+2s}} \,
\mathrm{d}x \mathrm{d}y = \int_{\mathbb{R}^d} \phi(x)
(-\Delta)^s(\psi)(x)  \, \mathrm{d}x
\end{equation}
for all $ \phi,\psi \in \mathcal{D}(\mathbb{R}^d) $, and our aim is to
extend it to all functions of $X_{s,\rho} $. In order to do it,
the first step consists in showing that $ C^\infty(\mathbb{R}^d)
\cap X_{s,\rho} $ is dense in $ X_{s,\rho} $. This can be done
by mollification arguments, which however are slightly more
complicated than the standard ones, since we work with the weighted spaces $
L^2_{\rho}(\mathbb{R}^d) $ and $ L^2_{\rho^{-1}}(\mathbb{R}^d) $
instead of $ L^2(\mathbb{R}^d) $. Hence, given $ v,w \in
C^\infty(\mathbb{R}^d) \cap X_{s,\rho} $, one plugs the cut-off
functions $ \phi := \xi_R v $ and $ \psi := \xi_R w $ into \eqref{eq:
int-parti-app} and lets $ R \to \infty $. The problem is that on
the r.h.s.\ there appear terms involving $ \| \xi_R w
\|_{\dot{H}^s}$, and a priori we do not know whether $
C^\infty(\mathbb{R}^d) \cap  X_{s,\rho} $ is continuously
embedded in $ \dot{H}^s(\mathbb{R}^d) $. But this turns out to be
true: the inequality
\begin{equation}\label{eq: ineq-apriori}
\frac{C_{d,s}}{2} \int_{\mathbb{R}^d} \int_{\mathbb{R}^d}
\frac{(w(x)-w(y))^2}{|x-y|^{d+2s}} \, \mathrm{d}x \mathrm{d}y
\le \int_{\mathbb{R}^d} w(x) (-\Delta)^s(w)(x)  \, \mathrm{d}x \ \
\ \forall w \in C^\infty(\mathbb{R}^d) \cap X_{s,\rho}
\end{equation}
can be proved just by repeating the above scheme with $ \phi=\psi=
\xi_R w $. In fact, on the r.h.s.\ of \eqref{eq: int-parti-app}
we still have terms involving $ \| \xi_R w \|_{\dot{H}^s}$, but
the latter are small and can be absorbed into the l.h.s.; passing
to the limit as $ R \to \infty $ yields \eqref{eq: ineq-apriori}.
Therefore, we can now let $R \to \infty $ safely in \eqref{eq: int-parti-app} (with
$ \phi= \xi_R v $ and $ \psi=\xi_R w $) and obtain that
\begin{equation}\label{eq: int-parti-app-step-1}
\frac{C_{d,s}}{2} \int_{\mathbb{R}^d} \int_{\mathbb{R}^d}
\frac{(v(x)-v(y))(w(x)-w(y))}{|x-y|^{d+2s}} \, \mathrm{d}x
\mathrm{d}y = \int_{\mathbb{R}^d} v(x) (-\Delta)^s(w)(x)  \,
\mathrm{d}x
\end{equation}
for all $v,w \in C^\infty(\mathbb{R}^d) \cap X_{s,\rho} $, which
in particular shows that \eqref{eq: ineq-apriori} is actually an equality.
Notice that in all these approximation procedures using cut-off
functions, to prove that ``remainder'' terms go to zero we deeply
exploit the results provided by Lemmas \ref{lem:decay-lap-1},
\ref{lem:decay-lap-2} and \ref{lem:decay-lap-cutoff}. It is in
fact here that the condition $ \gamma < 2s $ plays a fundamental
role: in particular, it ensures that both $\| \rho^{-1}
(-\Delta)^s(\xi_R) \|_\infty $ and $\| \rho^{-1} l_s(\xi_R)
\|_\infty $ vanish as $R \to \infty$. As already mentioned, we
refer the reader to the note \cite{Nota-oper} for the details.
However, for similar computations involving $(-\Delta)^s(\xi_R)$
and $l_s(\xi_R)$, see also the proofs of Proposition \ref{oss:
cons-mass}, Lemma \ref{lem: stroock-var} and Lemma \ref{lem:
lem-uniq-pre}.

By the claimed density of $ C^\infty(\mathbb{R}^d) \cap
X_{s,\rho} $, we are allowed to extend \eqref{eq:
int-parti-app-step-1} to the whole of $ X_{s,\rho} $. Clearly,
the r.h.s.\ of \eqref{eq: int-parti-app-step-1} can be rewritten
as
$$ \int_{\mathbb{R}^d} v(x) \, A(w)(x)  \, \rho(x) \mathrm{d}x \, , $$
and letting $v=w$ we obtain that the operator $ A $ is
positive. The fact that it is densely defined is trivial since,
for instance, $ \mathcal{D}(\mathbb{R}^d) \subset X_{s,\rho} $.
Because in \eqref{eq: int-parti-app-step-1} one can interchange
the role of $v$ and $w$, we also have that $ A $ is symmetric. In
order to prove that it is self-adjoint we need to show that $
D(A^\ast) \subset D(A) $, namely that any function of $ D(A^\ast)
$ also belongs to $X_{s,\rho}$. It is indeed straightforward to
check this fact, and we leave it to the reader.

We finally deal with the quadratic form $Q$ associated to $A$. Thanks
to \eqref{eq: int-parti-app-step-1}, we have that
\begin{equation}\label{eq: quad-form-dom-A}
Q(v,v)= \frac{C_{d,s}}{2} \int_{\mathbb{R}^d} \int_{\mathbb{R}^d}
\frac{(v(x)-v(y))^2}{|x-y|^{d+2s}} \, \mathrm{d}x \mathrm{d}y
\ \ \ \forall v \in D(A) \, .
\end{equation}
As it is well known (see e.g.~\cite{D}), the domain $D(Q)$ of $ Q
$ is just the closure of $D(A) $ w.r.t.\ the norm
$$ \left\| v \right\|_{Q}^2 := \left\| v \right\|_{2,\rho^{-1}}^2 + Q(v,v) = \left\| v \right\|_{2,\rho^{-1}}^2 + \left\| v \right\|_{\dot{H}^s}^2 . $$
It is then easy to see that such a closure is nothing but
$L^2_{\rho}(\mathbb{R}^d) \cap \dot{H}^s(\mathbb{R}^d) $ and
the quadratic form on $D(Q)=L^2_{\rho}(\mathbb{R}^d) \cap
\dot{H}^s(\mathbb{R}^d)$ is still represented by \eqref{eq:
quad-form-dom-A}.

By classical results (we refer again to \cite{D}), proving that $
A $ generates a Markov semigroup is equivalent to proving that if
$ v $ belongs to $D(Q) $ then both $ v \vee 0 $ and $ v \wedge 1 $
belong to $D(Q)$ and satisfy
\begin{equation*}%\label{eq: verifica-markov}
Q(v \vee 0,v \vee 0) \le Q(v,v) \, , \ \ \ Q(v \wedge 1,v \wedge
1) \le Q(v,v) \, .
\end{equation*}
But the latter properties are straightforward consequences of the
characterization of $Q$ given above.

The last assertions follow from the general theory of symmetric
Markov semigroups (cf.\ \cite[Section 1.4]{D}) and from their known
analiticity properties (cf.\ \cite[Theorem 1.4.2]{D}). See also the discussion in the proof of Lemma \ref{lem: approx}.
\end{proof}

The next proposition extends the symmetry property of
the operator $A=\rho^{-1}\,(-\Delta)^s$ to functions which belong
to other suitable $L^p_{\rho}$ spaces. This is essential in
proving our uniqueness Theorem \ref{thm: teorema-uniqueness} for
certain values of $\gamma$ and $s$ in low dimensions $d \le 3 $, more
precisely whenever $(d-\gamma)/(d-2s) > 2 $.

\begin{pro}\label{pro: laplaciano-Lp}
Let $d>2s$ and assume that $\rho$ satisfies \eqref{eq: ass-rho}
for some $\gamma \in [0,2s) \cap [0,d-2s]$ and $ \gamma_0 \in
[0,\gamma] $. Let $p \in [2,2(d-\gamma)/(d-2s)) $ and
$p^\prime={p}/(p-1)$ be its conjugate exponent. Suppose that $v,w
\in L^p_{\rho}(\mathbb{R}^d) $ are such that $A(v),A(w) \in
L^{p^\prime}_{\rho}(\mathbb{R}^d) $. Then $v,w \in
\dot{H}^s(\mathbb{R}^d) $ and the following formula holds:
\begin{equation*}%\label{eq: laplaciano-Lp}
\begin{aligned}
\int_{\mathbb{R}^d} v(x) (-\Delta)^s(w)(x)  \, \mathrm{d}x = & \int_{\mathbb{R}^d}  (-\Delta)^s(v)(x) \, w(x)  \, \mathrm{d}x  \\
= & \frac{C_{d,s}}{2} \int_{\mathbb{R}^d} \int_{\mathbb{R}^d}
\frac{(v(x)-v(y))(w(x)-w(y))}{|x-y|^{d+2s}} \, \mathrm{d}x
\mathrm{d}y \, .
\end{aligned}
\end{equation*}
\end{pro}
\begin{proof}[Sketch of proof]
The method of proof proceeds along the lines of the one of Theorem
\ref{thm: self-adj}. The main difference here lies in the fact
that, when using the approximation procedure by cut-off functions
mentioned above, if $p$ is \emph{strictly larger} than $2$ in
order to prove that ``remainder'' terms go to zero one cannot
exploit the fact that $ \rho^{-1} (-\Delta)^s(\xi_R) $ and $
\rho^{-1} l_s(\xi_R) $ vanish in $L^\infty(\mathbb{R}^d)$ as $R
\to \infty$. In fact, such remainder terms are of the form
\begin{equation}\label{eq: remainder-term}
\int_{\mathbb{R}^d}  v^2(x)  (-\Delta)^s(\xi_R)(x) \, \mathrm{d}x
\quad \textrm{or} \quad \int_{\mathbb{R}^d}  v^2(x) \,
l_s(\xi_R)(x) \, \mathrm{d}x \, .
\end{equation}
Thanks to Lemmas \ref{lem:decay-lap-1}, \ref{lem:decay-lap-2} and
\ref{lem:decay-lap-cutoff}, it is direct to see that $\| \rho^{-1}
(-\Delta)^s(\xi_R) \|_{q,-\gamma} $ and $\| \rho^{-1} l_s(\xi_R)
\|_{q,-\gamma} $ vanish as $R\to\infty$ provided $q  >
(d-\gamma)/(2s-\gamma) $, whence the condition $p \in
[2,2(d-\gamma)/(d-2s))$ to ensure that also the integrals in
\eqref{eq: remainder-term} go to zero as $R\to\infty$.
\end{proof}
%\noindent Again, for the details we refer to the note \cite{Nota-oper}.

\section*{Acknowledgements}
G. G. and M. M. have partially been supported by the MIUR-PRIN 2012  grant ``Equazioni alle derivate parziali di tipo ellittico e parabolico:
aspetti geometrici, disuguaglianze collegate, e applicazioni''. F. P. has been supported by MIUR-PRIN 2012 grant ``Critical Point Theory and Perturbative Methods for Nonlinear Differential Equations''. All authors thank the Gruppo Nazionale per l’Analisi Matematica, la Probabilit\`a e le loro Applicazioni (GNAMPA) of the Istituto Nazionale di Alta Matematica (INdAM).

\bibliographystyle{plainnat}

%\addcontentsline{toc}{chapter}{Bibliografia}

\end{document}